\documentclass[10pt]{amsart}
\usepackage{amscd,amssymb,graphics}
\usepackage[hidelinks]{hyperref}
\usepackage{amsfonts}
\usepackage{amsmath}
\usepackage{amsxtra}
\usepackage{latexsym}
\usepackage[mathcal]{eucal}

\input xy
\xyoption{all}
\usepackage{epsfig}

\newtheorem{thm}{Theorem}[section]
\newtheorem{f}[thm]{Fact}
\newtheorem{cor}[thm]{Corollary}
\newtheorem{lem}[thm]{Lemma}
\newtheorem{prop}[thm]{Proposition}

\newtheorem{problem}[thm]{Problem}

\newtheorem{question}[thm]{Question}
\theoremstyle{definition}
\newtheorem{defin}[thm]{Definition}

\theoremstyle{remark}
\newtheorem{remark}[thm]{Remark}
\newtheorem{remarks}[thm]{Remarks}

\newtheorem{ex}[thm]{Example}

\newtheorem{claim}[thm]{Claim}
\numberwithin{equation}{section}

\newcommand{\delete}[1]{} 
\newcommand{\nt}{\noindent}

\newcommand{\sig}{\sigma}

\def\eps{{\varepsilon}}
\newcommand{\sk}{\vskip 0.1cm}

\newcommand{\ben}{\begin{enumerate}}

\newcommand{\een}{\end{enumerate}}
\newcommand{\bit}{\begin{itemize}}

\newcommand{\eit}{\end{itemize}}

\newcommand{\TFAE}{The following are equivalent: }

\def\R {{\mathbb R}}
\def\N {{\mathbb N}}
\def\Z {{\mathbb Z}}

\def\U{{\mathbb U}}

\def\norm#1{\left\Vert#1\right\Vert}

\def\E{{\mathrm{E}}}
\def\Iso{{\mathrm{Iso}}\,}

\def\Homeo{{\mathrm{Homeo}}\,}

\newcommand{\id}{{\rm{id}}}

\def\A{{\mathcal{A}}}

\def\E{{\mathcal E}}
\def\F{{\mathcal F}}

\def\K{{\mathcal K}}
\def\calB{{\mathcal B}}

\def\Rcal {{\mathcal{R}}}

\def\QED{\nobreak\quad\ifmmode\roman{Q.E.D.}\else{\rm Q.E.D.}\fi}

\def\a{\alpha}
\def\om{\omega}
\def\Om{\Omega}

\def\s{\sigma}
\def\g{\gamma}

\newcommand{\del}{\delta}
\newcommand{\Del}{\Delta}
\newcommand{\br}{\vspace{4 mm}}
\newcommand{\imp}{\ \Rightarrow\ }
\newcommand{\cls}{\rm{cl\,}}

\newcommand{\la}{\lambda}
\newcommand{\al}{\alpha}

\newcommand{\inv}{\rm{inv}}

\newcommand{\Acal}{\mathcal{A}}

\newcommand{\lan}{\langle}
\newcommand{\ran}{\rangle}

\newcommand{\Pcal}{\mathcal{P}}

\newcommand{\ch}{\mathbf{1}}

\newcommand{\OC}{\bar{\mathcal{O}}}

\newcommand{\co}{{\rm{co\,}}}

\begin{document}

\title[]
{Banach representations and affine compactifications of dynamical systems}

\author[]{Eli Glasner}
\address{Department of Mathematics,
Tel-Aviv University, Ramat Aviv, Israel}
\email{glasner@math.tau.ac.il}
\urladdr{http://www.math.tau.ac.il/$^\sim$glasner}

\author[]{Michael Megrelishvili}
\address{Department of Mathematics,
Bar-Ilan University, 52900 Ramat-Gan, Israel}
\email{megereli@math.biu.ac.il}
\urladdr{http://www.math.biu.ac.il/$^\sim$megereli}

\date{April 18, 2013}

\begin{abstract}
To every Banach space $V$ we associate a compact right topological affine semigroup $\mathcal{E}(V)$.
We show that a separable Banach space $V$ is Asplund if and only if
$\E(V)$ is metrizable, and it is Rosenthal (i.e. it does not contain an isomorphic copy of $l_1$)
if and only if $\E(V)$ is a Rosenthal compactum.
We study representations of compact right topological semigroups in $\E(V)$.
In particular, representations of tame and HNS-semigroups arise naturally
as enveloping semigroups of tame and HNS (hereditarily non-sensitive) dynamical systems,
respectively. As an application we obtain a generalization
of a theorem of R. Ellis.
A main theme of our investigation is the relationship between the enveloping semigroup of
a dynamical system $X$ and the enveloping semigroup of its various affine compactifications $Q(X)$.
When the two coincide we say that the affine compactification $Q(X)$ is $E$-compatible.
This is a refinement of the notion of injectivity.
We show that distal non-equicontinuous systems do not admit any
$E$-compatible compactification.
We present several new examples of non-injective dynamical systems
and examine the relationship between injectivity and $E$-compatibility.
\end{abstract}



\maketitle
\setcounter{tocdepth}{1}
\tableofcontents

\section*{Introduction}
In this work we pursue our ongoing investigation of
representations of dynamical systems on Banach spaces
(see \cite{Me-op, Me-nz, GM1, GM-suc, Me-hilb, GMU, GW-hilb, GM-rose}).


Recall that a representation of a dynamical system $(G,X)$ on a
Banach space $V$ is given by a pair $(h,\a)$,
where $h: G \to \Iso(V)$ is a
\emph{co-homomorphism} (i.e., $h(g_1g_2)=h(g_2)h(g_1)$ for all $g_1, g_2 \in G$)
of the group $G$
into the group $\Iso(V)$ of linear isometries of $V$, and
$\a: X \to V^*$ is a weak$^*$ continuous
bounded
$G$-map with
respect to the dual action of $h(G)$ on $V^*$.
For semigroup actions $(S,X)$ we consider the co-homomorphisms $h: S \to \Theta(V)$, where $\Theta(V)$ is the semigroup of all
contractive operators.
For every representation $(h,\a)$, taking
$Q =\overline{co}^{w^*} (\a(X))$, we get natural affine $S$-compactifications $\a: X \to Q$.
This way of obtaining affine
compactifications
establishes a direct link to our earlier works which were mainly concerned with
representations on reflexive, Asplund and Rosenthal Banach spaces.

In Section \ref{s:OCandW} we discuss semigroup compactifications which arise from certain linear representations, the so-called
\emph{operator compactifications}. 
These were studied by Witz \cite{Wi}
and Junghenn \cite{Ju}.
In the weakly almost periodic case this 
approach retrieves the classical work of de Leeuw and Glicksberg \cite{deLG}.

To every Banach space $V$ we associate a compact right topological
affine semigroup $\E(V)$.   
This is actually the enveloping semigroup of the natural dynamical system
$(\Theta(V)^{op}, B^*)$,
where $B^* \subset V^*$ is the weak$^*$ compact unit ball and $\Theta(V)^{op}$ is the adjoint semigroup of $\Theta(V)$.
We show that
a separable Banach space $V$ is Asplund if and only if $\E(V)$ is metrizable,
and it is Rosenthal ( i.e. it does not contain an isomorphic copy of $l_1$)
if and only if $\E(V)$ is a Rosenthal compactum,
Theorems \ref{t:metr-W} and \ref{t:W=Ros-c} respectively.
We note that the first assertion, about Asplund spaces, can
in essence be already found in \cite{GMU}.

Among 
the representations of compact right topological semigroups in $\E(V)$
we are especially interested in tame and HNS-semigroups.
These arise naturally in the study of tame and HNS (= hereditarily non-sensitive) dynamical systems.


Tame dynamical metric systems appeared first in the work of K\"{o}hler \cite{Ko} under the name of \emph{regular} systems.
In \cite{GM1} we formulated a dynamical version of the Bourgain-Fremlin-Talagrand (in short: BFT) dichotomy (Fact \ref{D-BFT} below).
According to this an enveloping semigroup
is either {\it tame}: has cardinality $\le 2^{\aleph_0}$ and consists of Baire
class 1 maps, or it is topologically
wild and contains a copy of $\beta \N$,
the \v{C}ech-Stone compactificaltion of a discrete countable set.
This dichotomy combined with a characterization of Rosenthal Banach spaces, Theorem \ref{t:tame2}, lead to a dichotomy theorem for Banach spaces
(Theorem \ref{t:W=Ros-c}).

The enveloping semigroup characterization of (metric) tame systems in \cite{GMU} led us
in \cite{GM-rose} to a general, more flexible definition of tame systems. A (not necessarily
metrizable) compact dynamical system $X$ is {\em tame} if every member of its enveloping semigroup is a fragmented (Baire 1, for metrizable $X$) self-map on $X$.


In the papers \cite{GM-rose, GM1, Me-nz} we have shown that a metric system
is HNS (tame, WAP) if and only if it admits a faithful representation on an
Asplund (respectively, Rosenthal, reflexive) Banach space.
The algebra of all Asplund (tame) functions on a semigroup $S$ is defined as the collection of all functions on $S$ which come
from HNS (respectively, tame) $S$-compactifications $S \to X$.
These algebras are denoted by $\mathrm{Asp}(S)$ and $\mathrm{Tame}(S)$
respectively.
Tame and HNS dynamical systems were investigated in several recent publications.
See for example the papers by Huang \cite{Huang} and Kerr-Li \cite{KL}.

In Section \ref{s:repr} we strengthen some of our earlier results regarding representations on Banach spaces.
We show in Theorem \ref{H_+} that the Polish group $G=H_+[0,1]$,
which admits only trivial Asplund representations, is however Rosenthal representable.

One of the main topics treated in this work is a refinement of the notion of ``injectivity". The latter was introduced by K\"{o}hler \cite{Ko}
(who, in turn, was motivated by a problem of Pym \cite{Pym90})
and examined systematically in \cite{Gl-tame,Gl-str}.
A compact dynamical $G$-system $X$ is called \emph{injective} if the canonical (restriction) homomorphism 
$r : E(P(X)) \to E(X)$ --- where $E(X)$ denotes the enveloping
semigroup of the system $(G,X)$ and $P(X)$ is the compact space
of probability measures on $X$ --- is an injection, hence an isomorphism.
The refinement we investigate in the present work is the following one (Section \ref{s:Op-env}).
Instead of considering just the space $P(X)$ we consider any
embedding $(G,X) \hookrightarrow (G,Q)$ into an affine $G$-system
$(G,Q)$
with $Q =\overline{co}^{w^*} (X)$
and we say that this embedding is {\em $E$-compatible} if the homomorphism $r : E(Q) \to E(X)$ is injective (hence an isomorphism).

Distal affine dynamical systems have quite rigid properties. See for example the work of Namioka \cite{Na83}.
It was shown in \cite{Gl-af} that a minimally generated metric distal affine $G$-flow is equicontinuous.
Using a version of this result we show that for a minimal distal
dynamical system $E$-compatibility in any
faithful
affine compactification
implies equicontinuity. Thus such embedding is never $E$-compatible
when the system is distal but not equicontinuous (Proposition \ref{p:not-E-fl}).
In particular this way we obtain in Theorem \ref{e:dist-not} a concrete example of
a semigroup compactification which is not an \emph{operator compactification}. More precisely,
for the algebra $D(\Z)$ of all distal functions on $\Z$ the corresponding semigroup compactification $\a: \Z \to \Z^{D(\Z)}$
is not an operator compactification.

Non-injectivity is not restricted to distal systems.
We construct examples of Toeplitz systems which are not
injective, Theorem \ref{e: Toeplic}.
We don't have such examples for a weakly mixing system.
We also describe an example of a $\Z^2$-system which
admits an $E$-compatible embedding yet is not injective
(Example \ref{e:z2}).
We don't have such an example for $\Z$-systems.

The notion of a \emph{left introverted} (we say shortly: \emph{introverted})
linear subspace of $C(S)$ was introduced by M.M. Day in 1957.
It is an important tool in the study of semigroups of means and affine semigroup compactifications.
It also plays a major role in the theory of Banach semigroup algebras and their second duals,
see for example, \cite{BJM, Pym93, DLS}.
A weaker
property of subalgebras of $C(S)$ is being m-Introverted.
It turns out that a subalgebra of $\mathrm{RUC}(G)$ is m-introverted iff
the corresponding dynamical system is point-universal iff it is isomorphic
(as a dynamical system) to its own enveloping semigroup.
It is well known that the algebras $\mathrm{RUC}(G)$ and $\mathrm{WAP}(G)$ are 
introverted.
In general there is a large room between the algebras $\mathrm{RUC}(G)$ and
$\mathrm{WAP}(G)$
for topological groups $G$.
Indeed, by \cite{MPU} (\cite{AG} for monothetic $G$),  $\mathrm{RUC}(G)=\mathrm{WAP}(G)$ iff $G$ is precompact.

We provide new 
non-trivial examples of 
introverted spaces.
We show that $\mathrm{Tame}(S)$ is always introverted. Moreover,
all of its m-introverted $S$-subalgebras
(like, $\mathrm{Asp}(S)$ and $\mathrm{WAP}(S)$) are introverted.
As a particular case (Theorem \ref{m-intro to intro}) it follows
that every m-introverted \emph{separable} $S$-subalgebra of $C(S)$ is introverted.
Note also that, by \cite{GM-fp}, the algebra $\mathrm{Asp}(G)$
(which contains the algebra $\mathrm{WAP}(G)$) is (left) amenable
for every topological group $G$.
This is in contrast to the fact that the larger algebra
$\mathrm{Tame}(G)$ is, in general, non-amenable.

We show that a semigroup compactification $\nu: S \to P$ is an operator compactification iff the corresponding algebra
of this compactification $\A_{\nu}$ is \emph{intro-generated}. The latter
means that $\A_{\nu}$  contains an introverted subspace $F \subset \A_{\nu}$ such that the minimal closed subalgebra of $C(S)$ containing $F$ is $\A_{\nu}$.
(This phenomenon reflects the existence of an $E$-compatible system which is not injective.)
The space $D(\Z)$ of all distal functions on $\Z$ is not intro-generated (Theorem \ref{e:dist-not}).
The $\Z^2$-flow from Example \ref{e:z2} mentioned above provides
an intro-generated subspace of
$\l_{\infty}(\Z^2)$
which is not introverted.

In Section \ref{s:appl} we first show, in Theorem \ref{t:Ros-E}, that affine compactifications coming from
representations on Rosenthal spaces are E-compatible.
The core of the proof is Haydon's characterization of Rosenthal spaces in terms of the $w^*$-Krein-Milman property.
Using results of Section \ref{s:repr} about representations of tame systems
on Rosenthal spaces we show in Theorem \ref{t:tame is inj} that every tame $S$-space $X$ is injective.
This result was proved by K\"{o}hler \cite{Ko} for metrizable systems.
In \cite{Gl-tame} there is a simple proof of this which uses the
fact that for a tame metrizable system $X$ its enveloping semigroup
is a Fr\'{e}chet space.

Next we prove a representation theorem (Theorem \ref{t:Srepr}) according to which
the enveloping semigroup of a tame (respectively, HNS) system, admits an
admissible embedding into $\E(V)$, where $V$ runs over the class of Rosenthal (respectively, Asplund) Banach spaces.
These results extend the following well known theorem:
the class of reflexively representable compact right topological semigroups 
coincides with the class of compact semitopological semigroups (proved in \cite{Sh, Me-op}).
As an applications of Theorem \ref{t:Srepr}, using Theorem \ref{p:not-E-gr},
we obtain a \emph{generalized Ellis theorem}:
a tame compact right topological group is a topological group (Theorem \ref{gen-Ellis}).

Finally, a representation theorem for $S$-affine compactifications
(Theorem \ref{t:aff-rep}), shows that for tame (HNS, WAP)
compact metrizable $S$-systems, their $S$-affine compactifications can be
\emph{affinely $S$-represented} on Rosenthal (Asplund, reflexive) separable Banach spaces.

\section{Preliminaries}

Topological spaces are always assumed to be Hausdorff and completely
regular.
The closure of a subset $A \subset X$ 
is denoted by $\overline{A}$ or $cl(A)$.
Banach spaces and locally convex vector spaces are over the
field $\R$ of real numbers.
For a subset $A$ of a Banach space we denote by
$sp(A)$ and $\overline{sp}^{norm} (A)$
the linear span and the norm-closed linear span of $K$ respectively.
We denote by 
$co (A)$ and $\overline{co} (A)$ the convex hull and the closed convex hull of a set $A$, respectively.
If $A \subset V^*$ is a subset of the dual space $V^*$ of $V$ we mostly mean the weak$^*$ topology on $A$ and $\overline{co} (A)$
or $\overline{co} ^{w^*} (A)$ will denote the $w^*$-closure of $co (A)$ in $V^*$.
For a topological space $X$ we denote by $C(X)$ the Banach algebra of real
valued continuous and bounded functions equipped with the supremum norm. For a subset $A \subset C(X)$ we denote by
$\lan A \ran$ the smallest unital (i.e., containing the constants) closed subalgebra of $C(X)$ containing $A$.

\subsection{Semigroups and actions}

Let $S$ be a semigroup which is also a topological space.
By $\lambda_a: S \to S, x
\mapsto ax$ and $\rho_a: S \to S, x \mapsto xa$ we denote the left
and right $a$-transitions. 
The subset $\Lambda(S):=\{a \in S: \ \la_a \ \text{is continuous}\}$
is called the \emph{topological center} of $S$.

\begin{defin} \label{d:sem1} A semigroup $S$ as above is said to be:
\ben
\item a \emph{right topological semigroup} if every $\rho_a$ is continuous.
\item \emph{semitopological} if the multiplication $S \times S \to S$ is separately continuous.
\item \cite{Na72} \emph{admissible} if $S$ is right topological and $\Lambda(S)$ is dense in $S$.
\een
\end{defin}

Let $A$ be a subsemigroup of a right topological semigroup $S$.
If $A \subset \Lambda (S)$ then the closure $cl(A)$ is a right topological semigroup.
In general, $cl(A)$ is not necessarily a subsemigroup of $S$
(even if $S$ is compact right topological and $A$ is a left ideal). Also $\Lambda(S)$ may be empty for general compact right topological semigroup $S$. See \cite[p. 29]{BJM}.

\begin{defin} \label{d:con-act}
Let $S$ be a semitopological semigroup with a neutral element $e$.
Let $\pi: S \times X \to X$ be a {\it left action} of $S$ on a
topological space $X$. This means that
$ex=x$ and $s_1(s_2x)=(s_1s_2)x$ for all $s_1, s_2 \in S$ and $x \in X$, where
as usual, we write $sx$ instead of $\pi(s,x)=\lambda_s (x)= \rho_x (s)$.
Let $S \times X \to X$ and $S \times Y \to Y$ be two actions.
A map $f: X \to Y$ between $S$-spaces is an \emph{$S$-map} if $f(sx)=sf(x)$ for every $(s,x) \in S \times X$.

We say that $X$ is a \emph{dynamical} \emph{$S$-system} (or an
\emph{$S$-space} or an \emph{$S$-flow}) if the action $\pi$ is separately continuous
(that is, if all orbit maps $\rho_x: S \to X $ and all translations $\lambda_s:
X \to X$ are continuous).
We sometimes write it as a pair $(S,X)$.

\end{defin}

%

 A {\it right system} $(X,S)$ can be defined analogously. If
$S^{op}$ is the {\it opposite semigroup} of $S$ with the same
topology then $(X,S)$ can be treated as a left system $(S^{op},X)$
(and vice versa).

\begin{f} \label{f:sc=c} \cite{La-additional}  
Let $G$ be a \v{C}ech-complete (e.g., locally compact or completely
metrizable) semitopological group. Then every separately continuous action of $G$ on a compact space $X$ is continuous.
\end{f}

\nt \textbf{Notation:}
\emph{
All semigroups $S$ are assumed to be monoids,
i.e, semigroups with a neutral element which 
will be denoted by $e$.
Also actions are \emph{monoidal} (meaning $ex=x, \forall x \in X$) and separately continuous.
We reserve the symbol $G$ for the case when $S$ is a group.
All right topological semigroups below are assumed to be admissible.
}

Given $x \in X$, its {\em orbit} is the set $Sx=\{sx : s \in S\}$ and the closure of
this set, ${\cls}(Sx)$, is the {\em orbit closure} of $x$. A point $x$ with
${\cls}(Sx) = X$ is called a {\em transitive point},
and the set of transitive points is denoted by $X_{tr}$. We say that
the system is {\em point-transitive}
when $X_{tr} \ne\emptyset$. The system is called
{\em minimal} if $X_{tr} = X$.

\subsection{Representations of dynamical systems}
\label{s:ReprB}

A \emph{representation} of a semigroup $S$ on a normed space
$V$
is a co-homomorphism $h: S \to \Theta(V)$, where
$\Theta(V):=\{T \in L(V): \ ||T|| \leq 1\}$
and $h(e)=id_V$. Here $L(V)$ is the space of continuous
linear operators $V\to V$ and $id_V$ is the identity operator.
This is equivalent to the requirement that
$h: S \to \Theta(V)^{op}$ be a monoid homomorphism, where $\Theta(V)^{op}$
is the opposite semigroup of $\Theta(V)$.
If $S=G$, is a group then $h(G) \subset \Iso(V)$, where
$\Iso(V)$ is the group of all linear isometries from $V$ onto $V$.
The adjoint operator
$adj: L(V) \to L(V^*)$
 induces an injective co-homomorphism $adj: \Theta(V) \to \Theta(V^*), adj(s)=s^*$.
We will identify
$adj(L(V))$ and the opposite semigroup $L(V)^{op}$; as well as
$adj(\Theta(V)) \subset L(V^*)$ and its opposite semigroup
$\Theta(V)^{op}$. Mostly we use the same symbol $s$ instead of $s^*$. Since $\Theta(V)^{op}$ acts from the right on $V$
and from the left on $V^*$
we sometimes write $vs$ for $h(s)(v)$ and $s \psi$ for $h(s)^* (\psi)$.


A pair of vectors $(v,\psi) \in V \times V^*$ defines a function (called a {\it matrix coefficient} of $h$)
$$
m(v,\psi): S \to \R, \quad s \mapsto \psi(vs) = \lan vs, \psi \ran = \lan v, s \psi \ran.
$$


%

The weak operator topology on $\Theta(V)$ (similarly, on $\Theta(V)^{op}$) is the weak topology generated by all matrix coefficients.
So $h: S \to \Theta(V)^{op}$ is \emph{weakly continuous} iff $m(v,
\psi) \in C(S)$ for every $(v,\psi) \in V \times V^*$.
The strong operator topology on $\Theta(V)$ (and on $\Theta(V)^{op}$) is the pointwise
topology with respect to its left (respectively, right) action on the Banach space $V$.

\begin{lem} \label{l:introtype}
Let $h: S \to \Theta(V)$ be a weakly continuous co-homomorphism.
Then for every $v \in V$ the following map
$$
T_v: V^* \to C(S), \ \ T_v(\psi)=m(v,\psi)
$$
is a well defined linear bounded weak$^*$-pointwise continuous $S$-map between left $S$-actions.
\end{lem}

\begin{defin} \label{d:repr}
(See \cite{Me-nz,GM1}) Let $X$ be a dynamical $S$-system.
\ben
\item
A \emph{representation} of $(S,X)$ on a normed space $V$ is a pair
$$(h,\a): S \times X \rightrightarrows \Theta(V) \times V^*$$
where $h: S \to \Theta(V)$ is a co-homomorphism of semigroups and
$\a: X \to V^*$ is a weak$^*$ continuous bounded $S$-mapping with
respect to the dual action
$$S \times V^* \to V^*, \ (s \varphi)(v):=\varphi(h(s)(v)).$$
We say that the representation is weakly (strongly) continuous if
$h$ is weakly (strongly) continuous.
 A representation $(h,\a)$ is said to be \emph{faithful} if $\a$ is a topological embedding.

\item
If $\K$ is a subclass of the class of Banach spaces, we say
that a dynamical system $(S,X)$ is weakly (respectively, strongly) $\K$-\emph{representable}
if there exists a weakly (respectively, strongly) continuous
faithful representation of $(S,X)$ on a Banach space $V \in \K$.
\item
A subdirect product, i.e. an $S$-subspace of a direct product, of weakly (strongly) $\K$-representable
$S$-spaces is said to be weakly (strongly) \emph{$\K$-approximable}.
\een
\end{defin}

We consider in particular the following classes of Banach spaces: Reflexive, Asplund, and Rosenthal spaces.
 A reflexively (Asplund) representable
compact dynamical system is a dynamical version of
the purely topological notion of an \emph{Eberlein} (respectively,
a \emph{Radon-Nikodym}) compactum, in the sense of Amir and
Lindenstrauss (respectively, in the sense of Namioka).

\subsection{Background on Banach spaces and fragmentability}
\label{s:Ban}


\begin{defin} \label{def:fr}
 Let $(X,\tau)$ be a topological space and
$(Y,\mu)$ a uniform space.
\ben
\item \cite{JOPV}
$X$ is {\em $(\tau,
\mu)$-fragmented\/} by a
(typically, not continuous)
function $f: X \to Y$ if for every nonempty subset $A$ of $X$ and every $\eps
\in \mu$ there exists an open subset $O$ of $X$ such that $O \cap
A$ is nonempty and the set $f(O \cap A)$ is $\eps$-small in $Y$.
We also say in that case that the function $f$ is {\em
fragmented\/}. Notation: $f \in {\mathcal F}(X,Y)$, whenever the
uniformity $\mu$ is understood.
If $Y=\R$ then we write simply ${\mathcal F}(X)$.

\item \cite{GM1}
We say that a {\it family of functions} $F=\{f: (X,\tau) \to
(Y,\mu) \}$ is {\it fragmented}
if condition (1) holds simultaneously for all $f \in F$. That is,
$f(O \cap A)$ is $\eps$-small for every $f \in F$.
\item \cite{GM-tame}
We say that $F$ is an \emph{eventually fragmented family}
if every
infinite subfamily $C \subset F$ contains
an infinite fragmented subfamily $K \subset C$.
\een
\end{defin}

In Definition \ref{def:fr}.1 when $Y=X, f={id}_X$ and $\mu$ is a
metric uniform structure, we get the usual definition of
fragmentability (more precisely, $(\tau,\mu)$-fragmentability) in the sense of Jayne and Rogers \cite{JR}.
Implicitly it already appears in a paper of Namioka and Phelps \cite{NP}.

\begin{remark} \label{r:fr1} \cite{GM1, GM-rose}
\ben
\item It is enough to check the condition of Definition \ref{def:fr}
only for closed subsets $A \subset X$ and
for $\eps \in \mu$ from a {\it subbase} $\gamma$ of $\mu$ (that is,
the finite intersections of the elements of $\gamma$ form a base
of the uniform structure $\mu$).
\item
When $X$ and $Y$ are Polish spaces, $f: X \to Y$ is fragmented iff
$f$ is a Baire class 1 function.
\item
When $X$ is compact and $(Y,\rho)$ metrizable uniform space then $f: X \to Y$ is fragmented iff $f$ has a
\emph{point of continuity property} (i.e., for every closed
nonempty $A \subset X$ the restriction $f_{|A}: A \to Y$ has a continuity point).
\item When $Y$ is compact with its unique compatible uniformity $\mu$ then $p: X \to Y$ is fragmented if and only if
$f \circ p: X \to \R$ has a point of continuity property for every $f \in C(Y)$.
\een
\end{remark}

\begin{lem} \label{l:FrFa} \
\ben
\item
Suppose $F$ is a compact space, $X$ is \v{C}ech-complete,
$Y$ is a uniform space 
and we are given a separately continuous map $w: F \times X \to Y$.
Then the naturally associated family
$\tilde{F}:=\{\tilde{f}: X \to Y\}_{f \in F}$ is fragmented, where $\tilde{f}(x) = w(f,x)$.
\item
Suppose $F$ is a compact metrizable space, $X$ is hereditarily Baire 
 and $M$ is separable
and metrizable.
Assume we are given a map $w: F\times X \to M$ such that
every $\tilde{x}: F \to M, f \mapsto w(f,x)$ is continuous and $y: X \to M$ is continuous at
every $\tilde{y} \in Y$ for some dense subset $Y$ of $F$. Then the family $\tilde{F}$ is
fragmented.
\een
\end{lem}
\begin{proof}
(1): There exists a collection of uniform maps $\{\varphi_i: Y \to M_i\}_{i \in I}$
into metrizable uniform spaces $M_i$ which generates the uniformity on $Y$.
 Now for every closed subset $A \subset X$ apply Namioka's joint continuity theorem 
 to the separately continuous map
 $\varphi_i \circ w: F \times A \to M_i$ and take into account Remark \ref{r:fr1}.1.

(2): Since every $\tilde{x}: F \to M$ is continuous, the natural map $j: X
\to C(F, M), \ j(x)=\tilde{x}$ is well defined.
Every closed nonempty
subset $A \subset X$ is Baire. By
\cite[Proposition 2.4]{GMU}, 
$j |_{A}: A \to
C(F, M)$ has a point of continuity, where $C(F,M)$ carries the
sup-metric.
Hence, $\tilde{F}_{A} =\{\tilde{f}
\upharpoonright_{A}: A \to M \}_{f \in F}$ is equicontinuous at
some point $a\in A$. This implies 
that the family $\tilde{F}$ is fragmented.
\end{proof}

For other properties of fragmented maps and fragmented families
refer to \cite{Me-nz, GM1, GM-rose}.

Recall that a Banach space $V$ is an {\em Asplund\/} space if the
dual of every separable Banach subspace is separable. In the following result the equivalence (1) $\Leftrightarrow$ (2) is a
well known criterion \cite{N}, and (3) is a reformulation of (2) in terms of fragmented families.
When $V$ is a Banach space we denote by $B$, or $B_V$,
the closed unit ball of $V$. $B^*=B_{V^*}$ and $B^{**}:=B_{V^{**}}$ will denote the
weak$^*$ compact unit balls in the dual $V^*$ and second dual
$V^{**}$ of $V$ respectively.

\begin{f} \label{f:Asp} \emph{\cite{NP, N}} 
Let $V$ be a Banach space. The following conditions are
equivalent:
\ben
\item $V$ is an Asplund space.
\item
Every bounded subset $A$ of the dual $V^*$ is (weak${}^*$,norm)-fragmented.
\item $B$ 
is a fragmented family of real valued maps on the compactum $B^*$.
\een
\end{f}

Assertion (3) is a reformulation of (2).
Reflexive spaces and spaces of the type $c_0(\Gamma)$
are Asplund. For more details cf. \cite{Fa, N}.

We say that a Banach space $V$ is {\em Rosenthal} if it does
not contain an isomorphic copy of $l_1$. Clearly, every Asplund space is Rosenthal.

\begin{defin} \label{d:Ros-F}
\cite{GM-rose}
Let $X$ be a topological space. We say that a subset $F\subset
C(X)$ is a \emph{Rosenthal family} (for $X$) if $F$ is norm bounded and
the pointwise closure ${\cls_p}(F)$ of $F$ in $\R^X$ consists of fragmented maps,
that is,
${\cls_p}(F) \subset {\mathcal F}(X).$
\end{defin}


Let $f_n: X \to \R$ be a uniformly bounded sequence of functions on a \emph{set} $X$. Following  Rosenthal we say that
this sequence is an \emph{$l_1$-sequence} on $X$ if there exists a real constant $a >0$
such that for all $n \in \N$ and all choices of real scalars $c_1, \cdots, c_n$ we have
$$
a \cdot \sum_{i=1}^n |c_i| \leq ||\sum_{i=1}^n c_i f_i||.
$$
This is the same as requiring that the closed linear span in $l_{\infty}(X)$
of the sequence $f_n$ be linearly homeomorphic to the Banach space $l_1$.
In fact, in this case the map
$$l_1 \to l_{\infty}(X), \ \ (c_n) \to \sum_{n \in \N} c_nf_n$$
is a linear homeomorphic embedding.

The equivalence of (1) and (2) in the following fact is well known. See for example, \cite{Tal}.

\begin{f} \label{f:sub-fr}
\cite{Tal, GM-rose}
Let $X$ be a compact space and $F \subset C(X)$ a bounded subset.
The following conditions are equivalent:
\begin{enumerate}
\item
$F$ does not contain a subsequence equivalent to the unit basis of $l_1$.
\item
$F$ is a Rosenthal family for $X$. 
\item $F$ is an eventually fragmented family.

\end{enumerate}
\end{f}

We need some known characterizations of Rosenthal spaces.

\begin{f} \label{f:RosFr}
Let $V$ be a Banach space. The following conditions are
equivalent:
\begin{enumerate}
\item
$V$ is a Rosenthal Banach space.
\item \emph{(Rosenthal \cite{Ros0})}
Every bounded sequence in $V$ has a weak-Cauchy subsequence.
\item \emph{(E. Saab and P. Saab \cite{SS})}
Each $x^{**} \in V^{**}$ is a fragmented map when restricted to the
weak${}^*$ compact ball $B^*$. Equivalently, if $B^{**} \subset \F(B^*)$.

\item \emph{(Haydon \cite[Theorem 3.3]{Hay})} For every weak$^*$ compact subset $Y \subset V^*$ the weak$^*$
and norm closures of the convex hull $co(Y)$ in $V^*$ coincide:
$\overline{co}^{w^*} (Y)=\overline{co}^{norm} (Y)$.
\item $B$ is a \emph{Rosenthal family} for the weak$^*$ compact unit ball $B^{*}$.
\end{enumerate}
\end{f}

Condition (3) is a reformulation (in terms of fragmented maps) of a
criterion from \cite{SS} which was originally stated in terms of the point of continuity property.
(5) can be derived from (3).

\begin{f} \label{f:functionals} \emph{(Banach-Grothendieck theorem)}
\cite[Cor. 2.6]{AE}
If $V$ is a Banach space then for every continuous linear functional
$u: V^* \to \R$ on the dual space $V^*$ the following are equivalent:
\ben
\item $u$ is $w^*$-continuous.
\item The restriction $u|_{B^*}$ is $w^*$-continuous.
\item $u$ is the evaluation at some point of $V$. That is, $u \in i(V)$,
where $i: V \hookrightarrow V^{**}$ is the canonical embedding.
\een
\end{f}

Let $\{V_i\}_{i \in I}$ be a family of Banach spaces. The $l_2$-sum of this family, denoted by $V:=(\Sigma_{i \in I} V_i)_{l_2}$,
is defined as the space of all functions $(x_i)_{i \in I}$ on $I$ such that $x_i \in V_i$ and
$$
||x||:=(\sum_{i \in I} ||x_i||^2)^{\frac{1}{2}} < \infty.
$$


\begin{lem} \label{l:sum} \
\ben
\item $V^*=(\Sigma_{i \in I} V_i)_{l_2}^*=(\Sigma_{i \in I} V_i^*)_{l_2}$ and the pairing $V \times V^* \to \R$ is defined by
$\lan v, f \ran = \sum_{i \in I} f_i(v_i)$.
\item If every $V_i$ is reflexive (Asplund, Rosenthal) then $V$ is reflexive (respectively: Asplund, Rosenthal).
\item For every semitopological semigroup $S$ the classes of reflexively (Asplund, Rosenthal) representable compact $S$-spaces are
closed under countable products.
\een
\end{lem}
\begin{proof} (1) This is well known (see, for example, \cite{NP}).

(2) The reflexive case follows easily from (1). For the Asplund case see \cite{NP} (or \cite{Fa} for a simpler proof).
Now suppose that each $V_i$ is Rosenthal
and $l_1 \subset V=(\Sigma_{i \in I} V_i)_{l_2}$.
Since $l_1$ is separable one may easily reduce the question to the case of countably many Rosenthal spaces $V_i$.
So we can suppose that $V:=(\Sigma_{n \in \N} V_n)_{l_2}$.
In view of Fact \ref{f:RosFr} it suffices to show that every element
$u=(u_n)_{ \in V^{**}}$
is a fragmented map on the weak$^*$ compact unit ball
$B_{V^*}$. That is, we need to check that $u \in \F(B_{V^*})$.
The set $\F(X) \cap l_{\infty}(X)$ is
a Banach subspace of $l_{\infty}(X)$ for every topological space $X$.
So the proof can be reduced to the case of coordinate
functionals $u_{n_0}$. Also, $\lan u_{n_0}, (f_n)_{n \in \N} \ran= f_{n_0}(u_{n_0})$.
Now use the fact that $u_{n_0}$ is a fragmented map
on $B_{V^*_{n_0}}$ because $V_{n_0}$ is Rosenthal (Fact \ref{f:RosFr}).

(3) Similar to \cite[Lemma 3.3]{Me-hilb} (or \cite[Lemma 4.9]{Me-nz}) using (2) and the $l_2$-sum of representations
$(h_n,\a_n)$ of $(S,X_n)$ on $V_n$ where $||\a_n(x)|| \leq 2^{-n}$ for every $x \in X_n$ and $n \in \N$.
\end{proof}

 \subsection{$S$-Compactifications and functions} \label{s:fun}

A \emph{compactification} of $X$ is a pair $(\nu,Y)$ where $Y$ is a
compact (Hausdorff, by our assumptions) space and $\nu$ is a continuous map with a
dense range.

The Gelfand-Kolmogoroff theory \cite{GK} establishes an order
preserving bijective correspondence (up to equivalence of
compactifications) between Banach {\it unital}
subalgebras $\Acal \subset C(X)$ and compactifications $\nu: X
\to Y$ of $X$.
Every Banach unital $S$-subalgebra $\Acal$ induces
the {\it canonical $\Acal$-compactification} $\a_{\Acal}: X \to
X^{\Acal}$, where $X^{\Acal}$ is the spectrum
(or the Gelfand space --- the collection of continuous multiplicative functionals on $\Acal$).
The map $\a_{\Acal}: X \to X^{\Acal} \subset \A^*$ is
defined by the {\it Gelfand transform}, the evaluation at $x$
functional, $\a_{\Acal}(x)(f):=f(x)$. Conversely,
every compactification $\nu: X \to~Y$ is equivalent to the {\it
canonical $\Acal_{\nu}$-compactification} $\a_{\Acal_{\nu}}: X \to
X^{\Acal_{\nu}}$, where the algebra $\Acal_{\nu}$ is defined as
the image $j_{\nu}(C(Y))$ of the embedding $ j_{\nu}: C(Y)
\hookrightarrow C(X), \ \phi \mapsto \phi \circ \nu. $


\begin{defin} \label{d:comp}
Let $X$ be an $S$-system. An $S$-compactification of $X$ is a
continuous $S$-map $\alpha: X \to Y$,  with a dense range,
into a compact $S$-system $Y$. An $S$-compactification is said to be
\emph{jointly continuous}
(respectively, \emph{separately continuous}) if the action $S \times Y \to Y$ is jointly continuous (respectively, separately continuous).
\end{defin}


By $S_d$ we denote the discrete copy of $S$.

\begin{remark} \label{r:domin}
If $\nu_1: X \to Y_1$ and $\nu_2: X \to Y_2$ are two
compactifications, then $\nu_2$ {\it dominates} $\nu_1$, that is,
$\nu_1= q \circ \nu_2$ for some (uniquely defined) continuous map
$q: Y_2 \to Y_1$ iff $\Acal_{\nu_1} \subset \Acal_{\nu_2}$. If in
addition, $X$, $Y_1$ and $Y_2$ are $S_d$-systems (i.e., all the
$s$-translations on $X$, $Y_1$ and $Y_2$ are continuous)
and if $\nu_1$ and $\nu_2$ are $S$-maps, then $q$ is also an $S$-map. Furthermore, if
the action on $Y_1$ is (separately) continuous then the action on $Y_2$ is (respectively, separately) continuous.
If $\nu_1$ and $\nu_2$ are
homomorphisms of semigroups then $q$ is also a homomorphism. See \cite[App. D]{Vr-b}.
\end{remark}

\subsection{From representations to compactifications}
\label{s:canonical-r}

Representations of dynamical systems $(S,X)$ lead to $S$-compactifications of $X$.
Let $V$ be a normed space and let
$$
(h,\a): (S,X) \rightrightarrows (\Theta(V)^{op}, V^*)
$$
be a representation of $(S,X)$, where $\a$ is a
weak$^*$
continuous map.
Consider the induced compactification $\a: X \to Y:=\overline{\a(X)}$,
the weak$^*$ closure of $\a(X)$.
Clearly, the induced natural action $S \times Y \to Y$ is well defined and
every left translation is continuous. So, $Y$ is an $S_d$-system.

\begin{remark} \label{l:actions-g} \
\ben
\item
The induced action $S \times Y \to Y$ is separately continuous iff
the matrix coefficient $ m(v,y): S \to \R
$
is continuous $\forall$ $v \in V, \ y \in Y$.

\item
If $h$ is strongly (weakly) continuous then the induced dual action of $S$ on the weak$^*$ compact unit ball
$B^*$ and on $Y$ is jointly (respectively, separately) continuous.
\een
\end{remark}

To every
$S$-space $X$ we associate the \emph{regular representation} on
the Banach space $V:=C(X)$ defined by the pair $(h,\a)$ where $h:
S \to \Theta(V), s \mapsto L_s$
(with $L_sf(x) = f(sx)$) is the natural co-homomorphism and
$\a: X \to V^*, x \mapsto \delta_x$ is the evaluation map
$\delta_x(f)=f(x)$. Denote by ($\mathrm{WRUC}(X)$) \ $\mathrm{RUC}(X)$
the set of all (weakly) right uniformly continuous functions.
That is functions $f \in C(X)$ such that the orbit
map $\tilde{f}: S \to C(X), s \mapsto fs=L_s(f)$ is (weakly) norm continuous.
Then $\mathrm{RUC}(X)$ and $\mathrm{WRUC}(X)$ are norm closed
$S$-invariant unital linear subspaces of $C(X)$ and the restriction of the regular
representation is continuous on $\mathrm{RUC}(X)$ and weakly
continuous on $\mathrm{WRUC}(X)$.
Furthermore, $\mathrm{RUC}(X)$ is a Banach subalgebra of $C(X)$.
If $S \times X \to X$ is continuous and $X$ is compact then
$C(X)=\mathrm{RUC}(X)$.
In particular, for the left action of $S$ on itself $X:=S$ we write simply $\mathrm{RUC}(S)$ and $\mathrm{WRUC}(S)$.
If $X:=G$ is a topological group with the left action on itself then $\mathrm{RUC}(G)$
is the usual algebra of right uniformly continuous functions on $G$.
Note that $W\mathrm{RUC}(S)$ plays a major role in the theory of semigroups being the
largest left introverted linear subspace of $C(S)$ (Rao's theorem; see for example, \cite{BJMo}).

We say that a function $f\in C(X)$ on an $S$-space $X$ \emph{comes from} an 
$S$-compactification  $\nu: X \to Y$
(recall that we require only that the actions on $X,Y$ are separately continuous)
if there exists $\tilde{f} \in C(Y)$ such that $f=\tilde{f} \circ \nu$.
Denote by $\mathrm{RMC}(X)$ the set (in fact a unital Banach algebra) of all functions on $X$ which come from $S$-compactifications.
The algebra $\mathrm{RUC}(X)$ is the set of all functions which come from jointly continuous $S$-compactifications.

\begin{remark} \label{r:actions}
Let $X$ be an $S$-system.
\ben
\item
For every
$S$-invariant normed
subspace $V$ of $\mathrm{WRUC}(X)$ we have
the regular weakly continuous $V$-representation $(h,\a)$ of $(S,X)$ on $V$ defined by $\a(x)(f)=f(x), \ f \in V$
and the corresponding $S$-compactification $\a: X \to Y:=\overline{\a(X)}$.
The action of $S$ on $Y$ is continuous
iff $V \subset \mathrm{RUC}(X)$.
\item
Let $(h,\a)$ be a representation of the $S$-system $X$ on a Banach space $V$.
The inclusion $\a(X) \subset V^*$ induces a \emph{restriction operator}
$$r: V \to C(X), \ r(v)(x)=\lan v,\a(x) \ran.$$
Then $r$ is a linear $S$-operator (between right actions) with $||r|| \leq 1$.
If $h$ is weakly (strongly) continuous then $r(V) \subset \mathrm{WRUC}(X)$ (respectively, $r(V) \subset \mathrm{RUC}(X)$).

\item For every topological space $X$ the classical order preserving  Gelfand-Kolmogoroff correspondence
between compactifications 
of $X$ and unital subalgebras has a natural $S$-space generalization. More precisely, if $X$ is an $S$-space then
$S$-invariant unital Banach subalgebras $F$ of $\textrm{RUC}(X)$ (resp., $\textrm{RMC}(X)$) control
the $S$-compactifications $X \to Y$ with (resp., separately) continuous actions $S \times Y \to Y$.
\een
\end{remark}


The correspondence described in Remark \ref{r:actions}.3
 for Banach subalgebras $F$ of $\textrm{RUC}(X)$ is well known for topological group actions, \cite{vr-embed}.
One can easily extend it to the case of topological semigroup actions \cite{BH, Me-cs}.
Compare this also to the description of jointly continuous \emph{affine $S$-compactifications}
(Section \ref{s:Affine}) in terms of $S$-invariant closed linear unital subspaces of $\mathrm{RUC}(X)$.


 Regarding a description of separately continuous $S$-compactifications via subalgebras of $\mathrm{RMC}(X)$
and for more details about Remarks \ref{l:actions-g}, \ref{r:actions} see, for example, \cite{Me-nz, Me-cs} and,
also, Remark \ref{r:cycl-comes} below.

A word of caution
about our notation of $\mathrm{WRUC}(S), \mathrm{RUC}(S), \mathrm{RMC}(S)$. Note that in \cite{BJMo} the corresponding notation
is $\mathrm{WLUC}(S)$, $\mathrm{LUC}(S)$, $\mathrm{LMC}(S)$ (and sometimes $\mathrm{WLC}(S)$, $\mathrm{LC}(S)$, \cite{BJM}).

\begin{remark} \label{r:univ}
Let $\mathcal{P}$ be a class of compact separately continuous $S$-dynamical systems.
The subclass of $S$-systems with continuous actions will be denoted by
$\mathcal{P}_c$.
Assume that $\mathcal{P}$ 
is closed under products, closed subsystems and $S$-isomorphisms.
In such cases (following \cite[Ch. IV]{Vr-b}) we say that $\mathcal{P}$ is \emph{suppable}.
Let $X$ be a not necessarily compact $S$-space and let $\Pcal(X)$
be the collection of functions on $X$ coming from 
systems having property $\mathcal{P}$.
Then, as in the case of jointly continuous actions (see \cite[Prop. 2.9]{GM1}),  
there exists a universal $S$-compactification $X \to X^{\Pcal}$ of $X$ such that $(S,X) \in \Pcal$.
Moreover,
$j(C(X^{\Pcal}))=\Pcal(X)$.
In particular, $\Pcal(X)$ is a
uniformly closed, $S$-invariant subalgebra of $C(X)$.
Analogously, one defines $\Pcal_c(X)$. Again it is  a uniformly closed,
$S$-invariant subalgebra of $C(X)$, which is in fact a subalgebra of $\mathrm{RUC}(X)$.
For the corresponding $S$-compactification $X \to X^{\Pcal_c}$ the action of $S$ on $X^{\Pcal_c}$ is continuous.

In particular, for the left action of $S$ on itself we get the definitions of $\Pcal(S)$ and $\Pcal_c(S)$. As in
\cite[Prop. 2.9]{GM1} one may show that $\Pcal(S)$ and $\Pcal_c(S)$ are m-introverted Banach subalgebras of $C(S)$ and they
define the $\Pcal$-universal and $\Pcal_c$-universal semigroup compactifications $S \to S^{\Pcal}$ and $S \to S^{\Pcal_c}$.
\end{remark}

In the present paper we are especially interested in the following classes of compact $S$-systems:
a) Tame systems (Definition \ref{d:Tame});
b) Hereditarily Non-Sensitive, HNS in short (Definition \ref{d:HNS});
c) Weakly Almost Periodic, WAP in short (Section \ref{s:wap}).
See Sections \ref{s:wap}, \ref{s:HNS}, \ref{s:tame} and also \cite{GM1, GM-suc, GM-rose}.

For the corresponding
algebras, defined by Remark \ref{r:univ}, we use the following notation:
$\mathrm{Tame}(X)$, $\mathrm{Asp}(X)$, $\mathrm{WAP}(X)$. Note that
the Tame (respectively, HNS, WAP)
systems are exactly the compact systems which admit sufficiently many representations on
Rosenthal (respectively, Asplund, reflexive) Banach spaces (Section \ref{s:repr}). 


%

\begin{lem} \label{l:classes}  \
\ben
\item
For every $S$-space $X$ we have
$\Pcal_c(X) \subset \mathrm{RUC}(X) \subset \mathrm{WRUC}(X)
\allowbreak \subset \mathrm{RMC}(X)$  and
$\Pcal_c(X) \subset \Pcal(X) \cap \mathrm{RUC}(X)$.
If 
$\Pcal$ is preserved by factors then
$\Pcal_c(X) = \Pcal(X) \cap \mathrm{RUC}(X)$.

%
%

\item
If $X$ is a compact $S$-system with continuous action
then 
$\Pcal_c(X) = \Pcal(X)$, $\mathrm{RUC}(X)=\mathrm{WRUC}(X)=\mathrm{RMC}(X)=C(X)$.

\item
If $S=G$ is a \v{C}ech-complete semitopological group then for every $G$-space $X$ we have
$\Pcal_c(X) = \Pcal(X)$, $\mathrm{RUC}(X)=\mathrm{WRUC}(X)=\mathrm{RMC}(X)$;
in particular, $\mathrm{RUC}(G)=\mathrm{WRUC}(G)=\mathrm{RMC}(G)$.

\item 
$\mathrm{WAP}_c(G)=\mathrm{WAP}(G)$ remains true for every semitopological group $G$.
\item \cite[p. 173]{BJM} If $S$ is a k-space as a topological space then $\mathrm{WRUC}(X)=\mathrm{RMC}(X)$.
\een
\end{lem}
\begin{proof}
(1) is straightforward. In order to check the less obvious part $\Pcal_c(X) \supset \Pcal(X) \cap \mathrm{RUC}(X)$
we use a fundamental property of cyclic compactifications (see Remark \ref{r:cycl-comes}.1).

(2) easily follows from (1). (3) follows from Fact \ref{f:sc=c}, and (4) from Fact \ref{f:Lawson}.
(5) is a generalized version of \cite[Theorem 5.6]{BJM} and easily follows from Grothendieck's Lemma \cite[Cor. A6]{BJM}.
\end{proof}

\begin{defin} \label{d:admS}
Let $X$ be a compact space with a separately continuous action $\pi: S \times X \to X$.
We say that $X$ is WRUC-\emph{compatible}
(or that $X$ is WRUC) if $C(X)=\textrm{WRUC}(X)$.
An equivalent condition is that
the induced action $\pi_P: S \times P(X) \to P(X)$ be separately continuous (Lemma \ref{cont-act}).
\end{defin}

\begin{remark} \label{r:WRUC}
We mention three useful sufficient conditions for being WRUC-compatible (compare \cite[Def. 7.6]{Me-nz}
where this concept appears under the name \emph{w-admissible}):
a) the action $S \times X \to X$ is continuous; b) $S$, as a topological space, is a k-space (e.g., metrizable);
c) $(S,X)$ is WAP.
Below in Proposition \ref{p:tame-is-wruc} we show that $\mathrm{Tame}(X) \subset \mathrm{WRUC}(X)$ for every $S$-space $X$. In particular,
it follows that every \emph{compact} tame (hence, every WAP) $S$-system is WRUC-compatible.
\end{remark}

\subsection{Semigroup compactifications}


\begin{defin} \label{d:semcomp} Let $S$ be a semitopological semigroup.
\ben
\item \cite[p. 105]{BJM} A {\it right topological semigroup compactification} of $S$ is a pair $(\g,T)$ such
that $T$ is a compact right topological semigroup, and $\gamma$ is
a continuous semigroup homomorphism from $S$ into $T$, where $\gamma(S)$ is
dense in $T$ and the left translation $\lambda_s: T \to T, \
x\mapsto \gamma(s)x$ is continuous for every $s \in S$, that is,
$\gamma(S) \subset \Lambda(T)$.

 It follows that the {\it associated action}
$$\pi_{\gamma}: S \times T \to T, \ \ (s,x) \mapsto
\gamma(s) x=\lambda_s(x)$$ is separately continuous.

\item  \cite[p. 101]{Ru} A \emph{dynamical right topological
semigroup compactification} of $S$
is a right topological semigroup compactification $(\g,T)$ in the sense of (1)
such that, in addition, $\gamma$ is a jointly continuous $S$-compactification, i.e.,
the action $\pi_{\gamma}: S \times T \to T$ is jointly continuous.

\een
\end{defin}

If $S$ is a monoid (as we require in the present paper) with the neutral element $e$
 then it is easy to show that necessarily $T$ is a monoid with the neutral element $\g(e)$.
For
a discrete semigroup $S$, (1) and (2) are equivalent.
Directly from Lawson's theorem mentioned above (Fact \ref{f:sc=c}) we have:

\begin{f} \label{p:r=d}
Let $G$ be a \v{C}ech-complete (e.g., locally compact or completely
metrizable) semitopological group. Then $\gamma: G \to T$ is a right
topological semigroup compactification of $G$ if and only if
$\gamma$ is a dynamical right topological
semigroup compactification of $G$.
\end{f}

For every semitopological
semigroup $S$ there exists a maximal right topological
(dynamical) semigroup compactification. The corresponding algebra is
$\mathrm{RMC}(S)$ (respectively, $\mathrm{RUC}(S)$). 
If in the definition of a semigroup
compactification $(\g,T)$ we remove the 
condition
$\gamma(S) \subset \Lambda(T)$ then maximal compactifications (in
this setting) need not exist (See \cite[Example V.1.11]{BJMo}
which is due to J. Baker).

Let $\Acal$ be a closed unital subalgebra of $C(X)$ for some
topological space $X$.  We let $\nu_{\Acal}: X \to X^{\Acal}$ be the
associated compactification map (where, as before, $X^{\Acal}$ is
the maximal ideal space of $\Acal$). For instance, the
\emph{greatest ambit}
(see, for example, \cite{Usp-comp, Vr-b})
of a topological group $G$ is the compact $G$-space
$G^{\mathrm{RUC}}:=G^{\mathrm{RUC}(G)}$. It defines the
\emph{universal 
dynamical semigroup compactification} of $G$. For
$\Acal=\mathrm{WAP}(G)$ we get the \emph{universal semitopological
compactification} $G \to G^{WAP}$ of $G$, which is the
universal WAP compactification of $G$ (see \cite{deLG}).
Note that by \cite{MPU} the 
projection $q: G^{\mathrm{RUC}} \to G^{WAP}$ is a homeomorphism
iff $G$ is precompact.

\begin{remark} \label{r:ex} \
\ben
\item Recall that $\mathrm{RUC}(G)$ generates the topology of $G$ for
every topological group $G$. It follows that the corresponding
canonical representation
 ({\it Teleman's representation})
$$(h,\a_{\textrm{RUC}}): (G,G) \rightrightarrows (\Theta(V)^{op},
B^*)$$ on $V:=\textrm{RUC}(G)$ is faithful and $h$ induces a
topological group embedding of $G$ into $\Iso(V)$. See \cite{pest-wh} for details.

\item  There exists a nontrivial Polish
group $G$ whose universal semitopological compactification $G^{WAP}$ is trivial.
This is shown in \cite{Me-rup} for the Polish group $G:=H_+[0,1]$
of orientation preserving homeomorphisms of the unit interval.
Equivalently: every (weakly) continuous representation
$G \to \Iso(V)$ of $G$ on a reflexive Banach space $V$ is trivial.
\item
A stronger result is shown in \cite{GM-suc}: every continuous representation
$G \to \Iso(V)$ of $G$ on an Asplund space $V$ is trivial and every Asplund function on $G$  is constant
(note that $Asp_c(G)=\mathrm{Asp}(G)$ for Polish $G$ by Lemma \ref{l:classes}.3).
Every nontrivial right topological semigroup
compactification of the Polish topological group $G:=H_+[0,1]$ is
not metrizable \cite{GMU}. In contrast we show in Theorem \ref{H_+} that $G$ is Rosenthal representable.
  \een
\end{remark}

\subsection{Enveloping semigroups}

Let $X$ be a compact $S$-system with a separately continuous action.
Consider the natural map $j: S \to C(X,X), s \mapsto \lambda_s$.
As usual denote by $E(X)=cl_p(j(S)) \subset X^X$ the enveloping (Ellis)
semigroup of $(S,X)$. The associated homomorphism
$j: S \to E(X)$ is
a right topological semigroup compactification
(say, \emph{Ellis compactification})
of $S$, $j(e)=id_X$ and the associated action $\pi_j: S \times E(X) \to E(X)$ is separately
continuous. Furthermore, if the $S$-action on $X$ is continuous then $\pi_j$ is continuous,
 i.e., $S \to E(X)$ is a dynamical 
 semigroup compactification.

%
%
%

%

\begin{lem}  \label{l:Env-all}  \
\ben
\item
Let $X$ be a compact semitopological $S$-space and $L$ a subset
of $C(X)$ such that $L$ separates points of $X$.
Then the Ellis compactification
$j: S \to E(X)$ is equivalent to the
compactification of $S$ which corresponds to the subalgebra
$\A_{L}:=\lan m(L,X) \ran$, the smallest norm closed $S$-invariant unital subalgebra of
$C(S)$ which contains the family
$$
\{m(f,x): S \to \R, \ s \mapsto f(sx)\}_{f \in L, \ x \in X}.
$$
\item
Let $q: X_1 \to X_2$ be a continuous onto $S$-map between compact
$S$-spaces.
There exists a (unique) continuous onto semigroup homomorphism
$Q: E(X_1) \to E(X_2)$ with $j_{X_1} \circ Q=j_{X_2}$.
\item Let $Y$ be a closed $S$-subspace of a compact 
$S$-system $X$.
The 
map $r_X: E(X) \to E(Y), \ p \mapsto p|_Y$
is the unique continuous onto semigroup homomorphism such that $r_X \circ j_X
=j_Y$.
%
\item
Let $\a: S \to P$ be a right topological compactification of
a semigroup $S$. Then the enveloping semigroup $E(S,P)$ of the semitopological system $(S,P)$
is naturally isomorphic to $P$.
\item If $X$ is metrizable then $E(X)$ is separable.
 Moreover, $j(S) \subset E(X)$ is separable. 
\een
\end{lem}
\begin{proof} (1) The proof is straightforward using the Stone-Weierstrass theorem.

(2) By Remark \ref{r:domin} it suffices to show that the
compactification $j_{X_1}: S \to E(X_1)$ dominates the compactification
$j_{X_2}: S \to E(X_2)$. Equivalently we have to verify the
inclusion of the corresponding algebras.
 Let $q(x)=y, f_0 \in C(X_2)$ and $f= f_0 \circ q$. Observe
that $m(f_0,y) = m(f,x)$ and use (1).

(3) Is similar to (2).

(4) Since $E(S,P) \to P, a \mapsto a(e)$ is a natural
homomorphism, $j_P:S \to E(S,P)$ dominates the compactification $S
\to P$. So it is enough to show that, conversely, $\a: S \to P$ dominates $j_P:S \to
E(S,P)$. By (1) the family of functions
$$
\{m(f,x): S \to \R\}_{f \in C(P), \ x
\in P}
$$
generates the Ellis compactification $j_P:S \to E(S,P)$. Now observe
that each $m(f,x): S \to \R$ can be extended naturally to the
function $P \to \R, \ p \mapsto f(px)$ which is \emph{continuous}.


(5) Since $X$ is a metrizable compactum, $C(X,X)$ is separable and metrizable in the compact open topology. Then $j(S) \subset C(X,X)$
is separable (and metrizable) in the same topology. Hence, the dense subset $j(S) \subset E(X)$ is separable in the pointwise topology.
This implies that $E(X)$ is separable.
\end{proof}

\begin{remark} \label{r:env=crtadm}
Every enveloping semigroup $E(S,X)$ is an example of a compact right topological admissible semigroup.
Conversely, every compact right topological \emph{admissible} semigroup $P$ is an enveloping semigroup
(of $(\Lambda(P),P)$, as it follows from Lemma \ref{l:Env-all}.4).
\end{remark}


\section{Operator compactifications}
\label{s:OCandW}

Operator compactifications
provide an important tool for constructing and studying
semigroup compactifications via representations of
semigroups on Banach spaces (or, more generally, on locally convex vector spaces).
In classical works by Eberlein, de Leeuw and Glicksberg,
it was shown that weakly almost periodic Banach representations of a semigroup
$S$ induce semitopological compactifications of $S$.
In general the situation is more complicated and we have
to deal with right topological semigroup compactifications of
a semigroup $S$.
We refer to the papers of Witz \cite{Wi} and Junghenn \cite{Ju}. 
We note also that in his book \cite{Ellis} R. Ellis builds his entire theory of
abstract topological dynamics using the language of operator representations.

First we reproduce the construction of Witz with some minor changes.
Let $h: S \to L(V)$ be a
weakly continuous, equicontinuous representation (co-homomorphism) of a
semitopological semigroup $S$
into the space $L(V)$ of continuous linear operators on
a locally convex vector space $V$.
``Equicontinuous" here means that the subset $h(S) \subset L(V)$ is an equicontinuous family of linear operators.
Then the weak$^*$ operator closure $\overline{h(S)^{op}}$ of the adjoint semigroup
$h(S)^{op} \subset L(V)^{op}=adj(L(V))$ in $L(V^*)$
is a right topological semigroup compactification of $S$. We obtain the compactification
$$
S \to P:=\overline{h(S)^{op}} \subset L(V^*)
$$
which, following Junghenn, we call an \emph{operator
compactification} of $S$ (induced by the representation $h$).
%
The weak$^*$ operator topology on $L(V^*)$ is the weakest topology generated by the system
$$\widetilde{m}(v, \psi): L(V^*) \to \R, \ p \mapsto \lan v, p \psi \ran  = \lan vp, \psi \ran$$
of maps, where $v \in V, \psi \in V^*$, $vp=p^*(v) \in V^{**}$ and $p^*: V^{**} \to V^{**}$ is the adjoint of $p$.

In fact, the semigroup $P$ can be treated also as the weak$^*$ operator closure of $h(S)$ in $L(V,V^{**})$.
The latter version is found mainly in \cite{Wi} and \cite{Ju}.

The \emph{coefficient algebra} $A_h$ (respectively, \emph{coefficient space} $M_h$)
of the representation $h: S \to \Theta(V)$ is the smallest norm closed,
unital subalgebra (respect., subspace)  of $C(S)$ containing all
the matrix coefficients of $h$
$$m(V,V^*)=\{m(v,\psi):
S \to \R, \ \ s \mapsto  \lan vh(s),\psi \ran | \ v \in V, \ \psi \in V^*  \}.$$
That is, according to our notation $A_h=\lan m(V,V^*) \ran$ and $M_h=\overline{sp}^{norm} (m(V,V^*) \cup \{\textbf{1}\})$.

\begin{lem} \label{l:by-coef-alg}
Let $S \to P:=\overline{h(S)^{op}} \subset L(V^*)$ be the operator
compactification induced by a weakly continuous equicontinuous
representation $h: S \to L(V)$ on a locally convex space $V$.
The algebra of this compactification is just the coefficient algebra $A_h$.
\end{lem}


\begin{proof} For every $(v, \psi) \in V \times V^*$ the function $m(v,\psi): S \to \R$ is a restriction of the
continuous map
$\widetilde{m}(v, \psi)|_P: P \to \R, \ \ p \mapsto \lan v,p\psi \ran=\psi(vp)$. Such maps separate points of $P$.
Now use the Stone-Weierstrass theorem.
\end{proof}


\subsection{The enveloping semigroup of a Banach space}
\label{s:WS}

Let $V$ be a Banach space and $\Theta(V)$ the semigroup of all
non-expanding operators from $V$ to itself. As in Section \ref{s:ReprB}
consider the natural left action of $\Theta(V)^{op}$ on the weak$^*$ compact unit ball $B^*$.
This action is separately continuous
when $\Theta(V)^{op}$ carries the weak operator topology.

\begin{defin} \label{d:envBan}
Given a Banach space $V$ we denote by $\E(V)$ the enveloping semigroup of the dynamical system $(\Theta(V)^{op}, B^*)$.
We say that $\E(V)$ is the \emph{enveloping semigroup of $V$}. 
\end{defin}

 Always, $\E(V)$ is a compact
right topological admissible affine semigroup.
The corresponding Ellis compactification $j: \Theta(V)^{op} \to \E(V)$ is a topological embedding.
Alternatively, $\E(V)$ can be defined as the weak$^*$ operator closure of the adjoint monoid
$\Theta(V)^{op}$ in $L(V^*)$ (Lemma \ref{l:W=E}.2). So it is the operator compactification of the semigroup $\Theta(V)$.

If $V$ is separable then $\E(V)$ is separable by Lemma \ref{l:Env-all}.5 because $B^*$ is metrizable.
$\E(V)$ is metrizable iff $V$ is separable Asplund, Theorem \ref{t:metr-W}.

Every weakly continuous representation $h: S \to \Theta (V)$ of a semitopological semigroup 
on a Banach space $V$ (by non-expanding operators) gives rise to a right topological semigroup compactification
$$
h: S \to \overline{h(S)^{op}} \subset \E(V)
$$
where $\overline{h(S)^{op}}$ is the closure in $\E(V)$.
We sometimes call it the \emph{standard operator compactification}  of $S$ (generated by the representation $h$).

%

\begin{defin} \label{d:reprCRTS}
Let $\a: P \to K$ be a continuous (not necessarily onto) homomorphism
between compact right topological admissible semigroups.
Suppose that $S$ is a dense subsemigroup of $\Lambda(P)$. We say that:
\ben
\item
$\a$ is $S$-\emph{admissible} if $\a(S) \subset \Lambda(K)$.
\item
$\a$ is \emph{admissible} if it is $S$-admissible with respect to
some dense subsemigroup $S \subset \Lambda(P)$.
\item
$P$ is \emph{representable} on a Banach space $V$ if there exists an admissible embedding $\a$ of $P$ into $\E(V)$.
If $V$ is Rosenthal (Asplund, reflexive) then we say that $P$ is Rosenthal (Asplund, reflexively) representable.
\een
\end{defin}

Every standard operator compactification generated by a representation $h$ of $S$ on $V$
induces an admissible embedding of $\overline{h(S)^{op}}$ into $\E(V)$ because,
$h(S) \subset \Theta(V)^{op}$ (and $\Theta(V)^{op}=\Lambda(\E(V))$, Lemma \ref{l:transit}.5).
In the next lemma,
as before,
given a subset $A \subset C(S)$, we let $\langle A \rangle$ denote the closed unital
subalgebra of $C(S)$ generated by $A$.

\begin{lem} \label{l:W=E}  \
\ben
\item
Every standard operator compactification
$h: S \to \overline{h(S)^{op}} \subset \E(V)$ is equivalent to
the Ellis compactification $j: S \to E=E(S,B^*)$.
The algebra of these compactifications is the coefficient algebra $A_h=\lan m(V,V^*) \ran$.
\item  $\overline{\Theta (V)^{op}}$ is isomorphic to $\E(V)$ 
and the algebra of the compactification $j: \Theta(V)^{op} \hookrightarrow \E(V)$ is the coefficient algebra $A_h$ for
$h: \Theta(V) \hookrightarrow \overline{\Theta (V)^{op}} = \E(V)$.
\item
The natural inclusion $\a: E(S,B^*) \hookrightarrow \E(V)$ is $j(S)$-admissible.
\een
\end{lem}
\begin{proof} (1)
Both of these compactifications have the same algebra $A_h$.
Indeed Lemma \ref{l:by-coef-alg} implies this for the compactification $h: S \to \overline{h(S)^{op}} \subset \E(V)$.
For $j: S \to E(S,B^*)$ use Lemma \ref{l:Env-all}.1.

Note that (2) is a particular case of (1) for $S=\Theta (V)$.

(3) is trivial because $j(S)$ is dense in $\Lambda(E(S,B^*))$ and $\a(j(S))=h(S) \subset \Theta(V)^{op}$.
\end{proof}

\begin{prop} \label{p:factorOC}
Every semigroup compactification 
is a factor of an operator semigroup compactification.
\end{prop}
\begin{proof}
Let $(\g,P)$ be a semigroup compactification of 
$S$. Take a faithful Banach representation of the $S$-flow $P$ on $V$. For example, one can take
the regular representation of $P$ on $V:=C(P)$.
Now the enveloping semigroup $E(S,P)$ is a factor of $E(S,B^*)$ which is an
operator semigroup compactification of $S$ (Lemma \ref{l:W=E}) and
$E(S,P)$ is naturally isomorphic to $P$ (Lemma \ref{l:Env-all}.4).
\end{proof}

In Example \ref{ex} we show that there exists a right topological semigroup compactifications of the group $\Z$,
which is not an operator compactifications.
It
follows that compact right topological \emph{operator}
semigroups are not closed under factors. Indeed the compactification
$\Z \to \beta \Z=\Z^{\mathrm{RUC}}$ is an operator compactification
(by Remarks \ref{r:intro}.1 below) and it is the universal $\Z$-ambit.

Not every admissible compact right topological (semi)group admits a representation
on a Banach space (see Theorem \ref{e:dist-not}).
On the other hand we will later investigate the question when a ``good" semigroup compactification can be realized as
a standard operator compactification 
on ``good" Banach spaces (see Section \ref{s:appl}).


\sk

In the sequel whenever $V$ is understood we use the following simple notations $\E:=\E (V), \ \Theta:=\Theta(V), \ \Theta^{op}:=\Theta(V)^{op}$. By $S_V$ we denote the unit sphere of $V$.

\begin{lem} \label{l:transit}
For every Banach space $V$, 
every $v \in S_V$ and $\psi \in S_{V^*}$ we have

\ben \item
 $\Theta v = B.$
 \item
$v \E= B^{**}.$
 \item
$ cl_{w^*} (\Theta^{op} \psi) = B^*$.
\item
$\E \psi=B^*$.
\item $\Lambda(\E)=\Theta^{op}$.
 \een
\end{lem}
\begin{proof} (1)
Take $f \in S_{V^*}$ such that $f(v)=1$. For every $z \in B$
define the rank 1 operator $$A(f,z): V \to V, \ \ x \mapsto f(x)z.$$

Then $A(f,z)(v)=z$ and $A(f,z) \in \Theta$ since $||A(f,z)||=||f|| \cdot ||z||=||z|| \leq 1$.

(2) By (1), $v \Theta^{op} = \Theta v = B$ which is pointwise dense in $B^{**}$ by Goldstine theorem. So, $v \E= B^{**}$
because $\E \to (V^{**},w^*), p \mapsto vp$ is continuous and $\E=\overline{\Theta^{op}}$.

(3) We can suppose that $V$ is infinite-dimensional (use (1) for the finite-dimensional case).
Then the unit sphere $S_{V^{\ast}}$ is
weak (hence, weak$^*$) dense in $B^{*}$. So
it is enough to prove that the weak$^*$ closure of $\Theta^{op} \psi$ contains $S_{V^*}$.
Let $\phi \in S_{V^*}$.  We have to show that
for every $\eps >0$ and $v_1, v_2, \cdots , v_n
\in V$ there exists $s \in \Theta$ such that $||s^*\psi(v_i) - \phi(v_i)|| < \eps$ for every $i=1,2, \cdots,n$,
where $s^* \in \Theta^{op}$
is the adjoint of the operator $s$.
Since $\psi \in V^*$ and $||\psi||=1$ one may choose $z \in B_V$ such that
$$
|\phi(v_i)(\psi(z)-1)| < \eps
$$
for every $i=1,2, \cdots,n$. Define $s:=A(\phi,z)$.
Then $$|(s^*\psi)(v_i)- \phi(v_i)|=|\psi(sv_i) - \phi(v_i)|=|\psi(\phi(v_i)z)-\phi(v_i)|=|\phi(v_i)(\psi(z)-1)| < \eps$$ for every $i$.

(4) Follows from (3) because $\E$ is the weak$^*$ operator closure of $\Theta^{op}$.

(5) Trivially, $\Lambda(\E) \supseteq \Theta^{op}$.
Conversely, let $\s \in \Lambda(\E).$ Then $\s \in L(V^*)$ with $||\s|| \leq 1$.
Consider the adjoint operator $\s^*: V^{**} \to V^{**}$. We have to show that
$\s^*(v) \in V \subset V^{**}$, for every $v \in \Lambda$,
where we treat $V$ as a Banach subspace of $V^{**}$.
By Fact \ref{f:functionals}
it is enough to show that $\s^*(v)|_{B^*}: B^* \to \R$ is w$^*$-continuous.
By our assumption, $\s \in \Lambda(\E).$ That is, the left translation $l_{\s}: \E \to \E$ is continuous.
Choose a point $z \in S_{V^*}$ and consider the orbit map $\tilde{z}: \E \to B^*, p \mapsto pz$.
Then, $\tilde{z} \circ l_{\s}=\s|_{B^*} \circ \tilde{z}$. By (4) we have $\E z=B^*$, hence, $\tilde{z}: \E \to B^*$ is onto.
Since $\E$ is compact, it follows that the map $\s|_{B^*}: B^* \to B^*$ is continuous. This implies that
 $\s^*(v)|_{B^*}: B^* \to \R$ is w$^*$-continuous (for any $v \in V$), as desired.
\end{proof}

\section{Affine compactifications of dynamical systems and introversion}
\label{s:Affine}

\subsection{Affine compactifications in terms of state spaces}
\label{s:aff-def}

Let $S$ be a semitopological semigroup.
An $S$-system $Q$ is an \emph{affine $S$-system} if $Q$ is a convex subset of a locally convex
vector space and each $\lambda_s: Q \to Q$ is affine. If in
addition $S=Q$
acts
on itself by left translations and if right
translations are also affine maps then $S$ is said to be an \emph{affine semigroup}.
For every compact affine $S$-system $Q$ each element of its enveloping
semigroup is a (not necessarily, continuous) affine self-map of $Q$.

\begin{defin} \label{d:affinization} \
\ben
\item
\cite[p. 123]{BJM} An \emph{affine semigroup compactification}
of a semitopological semigroup $S$ is a pair $(\psi,Q)$,
where $Q$ is a compact right topological affine semigroup and
$\psi: S \to Q$ is a continuous 
homomorphism such that $co (\psi(S))$ is dense in $Q$ and 
$\psi(S) \subset \Lambda(Q)$.
\item
By an \emph{affine $S$-compactification}
of an $S$-space $X$ 
we mean a pair
$(\a,Q)$, where $\a: X \to Q$ is a continuous $S$-map and $Q$ is a
convex compact affine $S$-flow
such that $\a(X)$ \emph{affinely generates} $Q$, that is \
 $\overline{co} (\a(X)) =Q$ (see \cite{Gl-af}).
 \item
In particular, for a trivial action (or for the trivial semigroup $S$) we
retrieve in (2) the notion of an \emph{affine compactification} of a topological space $X$.
\een
\end{defin}

An affine $S$-compactification $\a: X \to Q$ induces the
$S$-compactification $\a: X \to Y:=\overline{\a(X)}  \subset Q$ of $X$.
Of course we have $Y=\overline{\a(X)}=\a(X)$ when $X$ is compact.
Definition \ref{d:affinization}.2 is a natural extension of Definition \ref{d:affinization}.1.

\begin{remarks} \label{r:aff-list}  \
\ben
\item
For any Banach space $V$,
$\Theta$ is an affine semitopological semigroup, $(\Theta^{op}, B^*)$ is an affine system and
the inclusion $\Theta^{op} \hookrightarrow \E$ is
an affine semigroup compactification.
\item
Not every semigroup compactification (in contrast to affine semigroup compactifications)
comes as an operator compactification. See Theorem \ref{e:dist-not} (and Proposition \ref{p:dominates}) below.

\item
For every continuous 
compact $S$-system $X$, the weak$^*$ compact unit ball
$B^* \subset C(X)^*$ and its closed subset $P(X)$ of all
probability measures, are continuous affine $S$-systems (Proposition \ref{bar}.2).
\item
Every Banach representation $(h,\a)$ of an $S$-flow $X$ naturally induces an $S$-affine compactification $X \to Q:= \overline{co} \ \a(X)$ (Section \ref{s:Op-env}).
Conversely, every affine compactification of an $S$-space $X$ comes from a Banach representation of the $S$-space $X$
on the Banach space $V \subset C(S)$ which is
just the \emph{affine compactification space}
(see Lemmas \ref{aff-descr} and \ref{cont-act}).
\een
\end{remarks}

As in the case of compactifications of flows one defines notions of
preorder, factors and isomorphisms of affine compactifications.
More explicitly, we say that for two affine compactifications,
$\a_1: X \to Q_1$ \emph{dominates} $\a_2: X \to Q_2$ if there exists a
continuous affine map (a {\emph{morphism}) $q: Q_1 \to Q_2$
such that $q \circ \a_1=\a_2$. Notation: $\a_1 \succeq \a_2$.
If one may choose $q$ to be a homeomorphism then we say that
$q$ is an \emph{isomorphism} of affine compactifications.
Notation: $\a_2 \cong \a_1$. It is easy to see that $\a_2 \cong \a_1$ iff $\a_2 \succeq \a_1$ and $\a_1 \succeq \a_2$.

\begin{lem} \label{l:morph}
If $q: Q_1 \to Q_2$ is a morphism
between two $S_d$-affine compactifications $\a_1: X \to Q_1$ and $\a_2: X \to Q_1$
then $q$ is an $S$-map.
\end{lem}
\begin{proof} Since the $s$-translations in $Q_1$ and $Q_2$ are affine
it easily follows that the inclusion maps $co(\a_1(X)) \hookrightarrow Q_1$, $co(\a_2(X)) \hookrightarrow Q_2$
are $S_d$-compactifications and also
the restriction map $q: co(\a_1(X)) \to co(\a_2(X))$ is
an onto $S$-map. The induced map $co(\a_1(X)) \to Q_2$ defines an $S_d$-compactification.
Now Remark \ref{r:domin} yields that $q: Q_1 \to Q_2$ is an $S$-map.
\end{proof}

Recall that for a normed unital subspace $F$ of $C(X)$ the
\emph{state space} 
of $F$ is the $w^*$-compact subset  
$$M(F):=\{\mu \in F^*: \ \|\mu\|=\mu(\textbf{1})=1\}$$ of all \emph{means}
on $F$. If in addition $F \subset C(X)$ is a subalgebra, we denote
by $MM(F)$ the compact set of all multiplicative means on $F$. 
For a compact space $X$ and for $F=C(X)$ the 
state space $M(C(X))$ is the space of all probability measures on $X$ which we denote as usual by $P(X)$.


\begin{lem} \label{l:state}  \emph{(\cite{Ph, BJMo,BJM}} For every topological space $X$ 
we have:
\ben
\item State space $M(F)$ is convex and weak$^*$ compact in the dual $F^*$ of $F$.
\item The map
$\delta: X \to M(F), \ \delta(x)(f)=f(x),$ is
affine and weak$^*$-continuous, and its image $\delta(X)$ affinely generates
$M(F)$ (i.e. $\overline{co}^{w^*}(\delta(X)) = M(F)$).
\item
Every $\mu \in F^*$ is a finite linear combination of members of $M(F)$.
\item
If $F \subset C(X)$ is a subalgebra then $\delta(X)$ is dense in $MM(F)$.
\een
\end{lem}

Thus $\delta: X \to M(F)$ is an affine compactification of $X$. We call it the
\emph{canonical $F$-affine compactification} of $X$. The induced
compactification $\delta: X \to \overline{\delta(X)}=Y$
is said to be the \emph{canonical $F$-compactification} of $X$.  By
Stone-Weierstrass theorem it follows that 
$C(Y)$ is naturally isometrically isomorphic to $\Acal_{\delta}:=\lan F \ran$,
the closed unital subalgebra of $C(X)$ generated by $F$. 

For every compact convex subset $Q$ of a locally convex vector space $V$ we denote by
$\mathrm{A}(Q)$ the Banach unital subspace of $C(Q)$ consisting of
the affine continuous functions on $Q$.
Of course $f|_Q: Q \to \R$ is affine and continuous for every $f \in V^*$.
So by the Hahn-Banach theorem $\mathrm{A}(Q)$ always separates points of $Q$.
 It is well known
(see \cite[Cor. 4.8]{AE}) that the subspace $$\mathrm{A}_0(Q):=\{f|_Q +c : \ \  \ f \in V^*, \ c \in
\R \}$$ is uniformly dense in $\mathrm{A}(Q)$. If $V$ is a Banach space then by Fact \ref{f:functionals}
every $w^*$-continuous functional on $V^*$ is the evaluation at some point $v \in V$.
This implies the following useful lemma.
\begin{lem} \label{l:dense}
For every Banach space $V$ and a weak$^*$ compact convex set $Q \subset V^*$
the subspace $$\mathrm{A}_0(Q):=\{\tilde{v}|_Q +c : \ \   v \in V, \ c \in
\R \} = r_Q(V)+\R \cdot \textbf{1}$$ is uniformly dense in $\mathrm{A}(Q)$,
where $\tilde{v}(\varphi)=\lan v,\varphi \ran$ and $r_Q:V \to C(Q)$ is the restriction operator.
\end{lem}

Next we classify the affine compactifications of a topological space
$X$ in terms of unital closed subspaces of $C(X)$,
in the spirit of the Gelfand-Kolmogoroff theorem
(compare Remark \ref{r:actions}.3).
At least for compact spaces $X$ and point-separating subspaces
$F \subset C(X)$ versions of Lemma \ref{aff-descr} below
can be found in several classical sources.
See for example \cite[Ch. 6]{Ph}, \cite[Ch. 6, \S 29]{Choquet}, \cite[Theorem II.2.1]{Alf},
 \cite[Ch. 6, \S  23]{Sem}, \cite[Ch. 1, \S 4]{AE}.
For affine bi-compactifications of transformation semigroups it remains true in a suitable setting, \cite[Remark 3.2]{JP}.


\begin{lem} \label{aff-descr}
Let $X$ be a topological space. The assignment $\Upsilon: F
\mapsto \delta_F$, where $\delta_F: X \to M(F))$ is the canonical
$F$-affine compactification, defines an order preserving  bijective correspondence between the collection of
unital Banach subspaces $F$ of $C(X)$ and the collection of
affine compactifications of $X$ (up to equivalence).
In the converse direction, to every
affine compactification $\a: X \to Q$ corresponds the unital Banach subspace
$F:=A(Q)|_X \subset C(X)$ (called the \emph{affine
compactification space}). Then the canonical affine
compactification $\delta_F: X \to M(F)$ is affinely equivalent to
$\a: X \to Q$.
\end{lem}
\begin{proof}
For a Banach unital subspace $F$ of $C(X)$ define $\Upsilon(F)=(\delta_F,M(F))$
as the canonical $F$-affine compactification $\delta_F: X \to M(F)$.

\sk

\nt \emph{Surjectivity} of $\Upsilon$:

\nt
Every affine compactification $\a: X \to Q$, up to equivalence, is a
canonical $F$-affine compactification. In order to show this consider the set
$A(Q)$ of all continuous affine functions on $Q$, viewed as a (Banach unital)
subspace of $C(Q)$. Let $A(Q)|_X$ be the set of all functions on
$X$ which are $\a$-extendable to a continuous affine function on $Q$.
Thus $A(Q)|_X:=\a_{\sharp}(A(Q)) \subset C(X)$, where $\a_{\sharp}: C(Q)
\to C(X), f \mapsto f \circ \a$ is the natural linear operator induced by
$\a: X \to Q$. Every such operator has norm 1.
Moreover, since $\a(X)$ affinely generates $Q$ and the functions in $A(Q)$ are affine, it follows that
 $\a_{\sharp}: A(Q) \to C(X)$ is a linear isometric embedding.
 Denote by $F$ the Banach unital subspace $\a_{\sharp}(A(Q))=A(Q)|_X$ of $C(X)$. We are going to show that
 the affine compactifications $\delta_F: X \to M(F)$ and $\a: X \to Q$ are isomorphic.
Define the evaluation map
$$e: Q \to M(F) \subset F^*, \ \ e(q)(f)= \tilde{f}(q),$$
where $\tilde{f}:=\a_{\sharp}^{-1}(f) \in A(Q)$ is the uniquely defined extension of $f \in F:=A(Q)|_X$.
Since $\a_{\sharp}^{-1}: F \to A(Q)$ is a linear isometry we easily obtain that $e(q) \in F^*$. Clearly, $||e(q)||=e(q)(\textbf{1})=1$
for every $q \in Q$. Hence, indeed $e(q) \in M(F)$ and the map $e: Q \to M(F)$ is well-defined.
Since, $\tilde{f}: Q \to \R$ is an affine map for every $f \in F$, it easily follows that $e: Q \to M(F)$ is an affine map.
For every $x \in X$ we have $e(\a(x))(q)=\tilde{f}(\a(x))=f(x)$. So,  $\delta_F=e \circ \a$.
It is also clear that $e$ is $w^*$-continuous. Since $\delta_F (X)$ affinely generates $M(F)$ (Lemma \ref{l:state}), it follows that
$e(Q)=M(F)$. Always, $\mathrm{A}(Q)$ separates points of $Q$. This implies that $e: Q \to M(F)$ is injective, hence a homeomorphism.

 \sk

\nt The \emph{Injectivity} and \emph{order-preserving} properties of $\Upsilon$:

These properties follow from the next claim.
\sk

\nt \textbf{Claim:} \emph{If $\a_1: X \to Q_1$ and $\a_2: X \to Q_2$ are two affine compactifications then $\a_2$ dominates $\a_1$
if and only if $F_2 \supset F_1$,
where $F_1$ and $F_2$ are the corresponding affine compactification spaces.}

\sk

Suppose $F_2 \supset F_1$ and let $j: F_1 \hookrightarrow F_2$
be the inclusion map.
Then the restricted adjoint map $j^*: M(F_2) \to M(F_1)$ is a weak$^*$
continuous affine map and the following diagram commutes
\begin{equation*}
\xymatrix { X \ar[dr]_{\delta_1} \ar[r]^{\delta_2} & M(F_2)
\ar[d]^{j^*} \\
  & M(F_1) }
\end{equation*}
Moreover $j^*$ is onto (use for example Lemma \ref{l:state}.2). The second
direction is trivial.
\end{proof}

\begin{cor} \label{c:bar}
For every compact space $X$ the Banach space $C(X)$
 determines the universal (greatest) affine compactification $\a_{b}: X \to P(X)=M(C(X))$.
 For any other affine compactification $X \to Q$ we have a uniquely determined
 natural affine continuous onto map, called the \emph{barycenter map},
 $b: P(X) \to Q$, such that $\a_b=b \circ \delta$.
\end{cor}

%

\sk
Next we deal with affine compactifications of $S$-systems.

\begin{lem} \label{cont-act}
Assume that $X$ is endowed with a semigroup action $S \times X \to X$ with continuous translations, i.e.,
$X$ is an $S_d$-space.
Let $\a: X \to Q$ be an affine compactification of
the $S_d$-space $X$.
Denote by $F$ the corresponding affine compactification space.
\ben
\item
$S \times Q \to Q$ is naturally topologically $S$-isomorphic to the 
action $S \times M(F) \to M(F)$.
\item $S \times Q \to Q$ is separately continuous iff $F \subset \mathrm{WRUC}(X)$. 
\item $S \times Q \to Q$ is continuous iff $F \subset \mathrm{RUC}(X)$.
\een
\end{lem}
\begin{proof} (1) All the translations $\lambda_s: X \to X$ are continuous and
$F=\mathrm{A}(Q)|_X$ is an $S$-invariant subset of $C(X)$.
So it is clear that the natural dual action $S \times F^* \to F^*$ is well defined and that every translation
$\lambda_s: F^* \to F^*$ is weak$^*$ continuous.
Now observe that $S \times M(F) \to M(F)$ is a restriction of the action $S \times F^* \to F^*$.
Since $X \to Q$ and $X \to M(F)$ are $S_d$-affine compactifications and
the evaluation map $e: Q \to M(F)$ from the proof of Lemma \ref{aff-descr} is an isomorphism
of affine compactifications, we obtain by Lemma \ref{l:morph} that $e: Q \to M(F)$ is an $S$-map.

(2) Use Lemma \ref{l:state}.3 and the restriction operator $r: F \to C(X)$ (Remark \ref{r:actions}).

(3) Use  Remark \ref{r:actions}.1.
\end{proof}

\begin{prop} \label{bar} \
\ben
\item
If $X$ is an $S$-space then the same
assignment $\Upsilon$, as in Lemma \ref{aff-descr}, establishes an order preserving bijection between the collection of
$S$-invariant unital Banach subspaces $F$ of $\textrm{WRUC}(X)$ and
(equivalence classes of) $S$-affine compactifications of the $S$-system $X$.
Furthermore, the subspaces of $\textrm{RUC}(X)$ correspond exactly to the $S$-affine
compactifications $X \to Q$ with continuous actions $S \times Q \to Q$.
\item
Let $X$ be a compact $S$-system and $\a_{b}: X \to P(X)=M(C(X))$ be the
universal affine compactification of the space $X$.
 \begin{itemize}
  \item [(a)] For every $S_d$-affine compactification $\a: X \to Q$ the barycenter map $b: P(X) \to Q$ is an $S$-map.
  \item [(b)] $S \times P(X) \to P(X)$ is separately continuous iff $C(X)=\mathrm{WRUC}(X)$ (iff $X$ is $\mathrm{WRUC}$).
\item [(c)] $S \times X \to X$ is continuous iff $S \times P(X) \to P(X)$ is continuous iff $C(X)=\mathrm{RUC}(X)$.
 \end{itemize}
 \een
\end{prop}
\begin{proof}
(1) Use Lemma \ref{cont-act}.

(2) (a) Apply Lemmas \ref{l:morph} and Corollary \ref{c:bar}. For (b) and (c) use Lemma \ref{cont-act}.
\end{proof}

\begin{remark} \label{r:affAnd} \
\ben
\item
As we already mentioned every $S$-affine compactification $\a: X \to Q$ induces the
$S$-compactification $\a_0: X \to Y:=\overline{\a(X)}$ of $X$. Always
the affine compactification space $F:=A(Q)|_X \subset C(X)$
generates the induced compactification algebra $\A$ of $\a_0$. That is, $\langle F \rangle=\A$.
Indeed, affine continuous functions on $Q$ separate the points.
Hence, by Stone-Weierstrass theorem $A(Q)|_Y$ generates $C(Y)$.
It follows that $\a_0^*(A(Q)|_Y)=F$ generates $\a_0^*(C(Y))=\A$, where $\a_0^*: C(Y) \to C(X)$ is the induced map.
\item
For every weakly continuous representation
$h: S \to \Theta(V)$ on a Banach space $V$ we have the associated
affine semigroup compactification $h: S \to Q:=\overline{co}(h(S)^{op}) \subset \E$.
Observe that since ${co}(h(S)^{op}) \subset \Theta^{op} = \Lambda(\E)$,
the closure $Q:=\overline{co}(h(S)^{op})$ in $\E$ is a semigroup.
In this case we say that $S \to Q$ is a \emph{standard affine semigroup compactification}.
Every affine semigroup compactification can be obtained in this way
(see Proposition \ref{l:W=Univ}.2).
\item
More generally, every operator compactification (Section \ref{s:OCandW})
$h: S \to P:=\overline{h(S)^{op}} \subset L(V^*)$
of $S$ on a locally convex vector space $V$
induces an affine semigroup compactification $\a_h: S \to Q:=\overline{co}(h(S)^{op} \subset L(V^*)$.
The affine compactification space $\A(Q)|_S$ coincides with the coefficient space $M_h$.
Indeed, $M_h \subset \A(Q)|_S$ because every matrix coefficient
$m(v, \psi): S \to \R$ is a restriction of the map $\widetilde{m}(v,\psi): Q \to \R$ which is continuous and affine.
On the other hand the collection
$\{\widetilde{m}(v,\psi)\}_{v \in V, \psi \in V^*}$
separates the points of $Q$. It follows that
the Banach unital subspaces $M_h$ and $\A(Q)|_S$ of $C(S)$ induce the isomorphic affine compactifications of $S$.
By Lemma \ref{aff-descr} we obtain that $M_h=\A(Q)|_S$.
\een
\end{remark}


\subsection{Cyclic affine $S$-compactifications}
 \label{s:cyclAff}

Let $X$ be a (not necessarily compact) $S$-system.
For every $f \in \mathrm{WRUC}(X)$ denote by $\A_f:=\lan fS \cup \{\textbf{1}\} \ran \subset C(X)$
the smallest $S$-invariant unital Banach subalgebra which contains $f$. The corresponding Gelfand compactification is an
$S$-compactification $\a_f: X \to |\A_f| \subset \A_f^*$.
We call it the \emph{cyclic compactification} of $X$ (induced by $f$).
Now consider $V_f:=\overline{sp}^{norm} (fS \cup \{\textbf{1}\})$ ---
the smallest closed linear unital $S$-subspace of $\mathrm{WRUC}(X)$ generated by $f$.
By Proposition \ref{bar}.1 we have the affine $S$-compactification
$\delta_f: X \to M(V_f) \subset V_f^*$
(where, $\delta_f(x)(\varphi)=\varphi(x)$ for every $\varphi \in V_f$)
which we call the \emph{cyclic affine $S$-compactification} of $X$. 
A natural idea is to reconstruct
$\a_f$ from $\delta_f$ restricting the codomain 
of $\delta_f$.
In the following technical lemma
we also give a useful realization (up to isomorphisms)
of these two compactifications in $C(S)$ with the pointwise convergence topology.
Note that we have also a left action of $S$ on $C(S)$ defined by
$S \times C(S) \to C(S), (sf)(t)=(ts)=R_sf$.


\begin{lem} \label{l:f-introtype} \
\ben
\item The following map
$$
T_f: V_f^* \to C(S), \ \ T_f(\psi)=m(f,\psi).
$$
is a well defined linear bounded weak$^*$-pointwise continuous $S$-map between left $S$-actions.
\item
The restriction $T_f: M(V_f) \to Q_f$ is an isomorphism of the affine $S$-compactifications $\delta_f: X \to M(V_f)$ and
$\pi_f: X \to Q_f$, where $\pi_f:=T_f \circ \delta_f$ and $Q_f:=T_f(M(V_f)) \subset C(S)$.
\item
Consider $$X_f:=\overline{\pi_f(X)}^p= cl_p(\{m(f,\delta_f(x)\}_{x \in X})  \subset Q_f.$$
The restriction of the codomain
leads to the $S$-compactification $\pi_f: X \to X_f$ which
is isomorphic to the cyclic compactification $\a_f: X \to |\A_f|$.
\een
\end{lem}
\begin{proof} (1) $m(f,\psi) \in C(S)$ \ $\forall$ $\psi \in V_f^*$ because $f \in \mathrm{WRUC}(X)$.
Other conditions are also easy.

(2)
$T_f: M(V_f) \to Q_f$ is a morphism of the affine $S$-compactifications $\delta_f: X \to M(V_f)$ and
$\pi_f: X \to Q_f$, where $\pi_f:=T_f \circ \delta_f$. So, $\delta_f \succeq \pi_f$.
In order to establish that $\pi_f \succeq \delta_f$ it is enough to show that
our original function $f: X \to \R$ belongs to the affine compactification
space of $\pi_f$. This follows from the observation that the evaluation at
$e$ functional, $\widehat{e}: C(S) \to \R$,
restricted to $\overline{co} (X_f) \subset C(S)$, is an affine function such that $f=\widehat{e} \circ \pi_f$.

(3) By Remark \ref{r:affAnd}.1 the algebra of the cyclic compactification $\pi_f: X \to X_f$
is just $\lan V_f \ran$, but this is exactly $\A_f$,
the algebra of the compactification $\a_f: X \to |\A_f|$.
\end{proof}


\begin{remark} \label{r:cycl-comes}  \
\ben
\item
Note that $f=F_e \circ \pi_f$, where $F_e:= \widehat{e}|_{X_f}$.
So $f$ comes from the $S$-system $X_f$. Moreover, if $f$
comes from an $S$-system $Y$ and an $S$-compactification $\nu: X \to Y$,
then there exists a continuous onto $S$-map $\a: Y \to X_f$ such that $\pi_f=\a \circ \nu$.
The action of $S$ on $X_f$ is continuous iff $f \in \mathrm{RUC}(X)$
(see Remark \ref{r:actions}.3).
\item  Moreover, $\widehat{e}S =\{\widehat{s}\}_{s \in S}$ separates points of $Q_f$
(where $\widehat{s}$ is the evaluation at $s$ functional) and $(X_f)_{F_e} = X_f$.
\item
The action of $S$ on $X_f$ is separately continuous iff $f \in \mathrm{RMC}(X)$
(use again Remark \ref{r:actions}.3).
By definition, $Q_f = \overline{co}^p (X_f)$ in $C(S)$.
At the same time the extended action of $S$ on $Q_f$ need not be separately continuous for $f \in \mathrm{RMC}(X)$.
This is a reflection of the fact that in general $\mathrm{RMC}(X)$ is strictly larger than $\mathrm{WRUC}(X)$ (see \cite[p. 219]{BJM}).
However by Lemma \ref{l:classes} we know that $\mathrm{WRUC}(X)=\mathrm{RMC}(X)$ in many natural cases.
Also, $\mathrm{Tame}(X) \subset \mathrm{WRUC}(X)$ for every $S$-system $X$ by Proposition \ref{p:tame-is-wruc}.
\een
\end{remark}

\begin{lem} \label{l:prop}
Let $V$ be a Banach $S$-invariant unital subspace of $\mathrm{WRUC}(X)$ and $f \in V$. Then
\ben
\item $X_f:=\overline{\pi_f(X)}^p=cl_p(m(f,\delta_f(X)))$.
\item $Q_f=m(f,M(V_f))=m(f,M(V))=\overline{co}^p (X_f)$.
\item In the particular case of $X:=S$, with the left action of $S$ on itself, we have $X_f=\overline{Sf}^p$.
\een
\end{lem}
\begin{proof}
Straightforward, using Lemma \ref{l:f-introtype}.
\end{proof}


Some other useful properties of cyclic compactifications can be found in \cite{BJM, GM1, GM-suc}.

\subsection{Introversion and semigroups of means}

In this section we assume that $F$ is a
normed
unital subspace of $C(S)$, where $S$, as before, is a semitopological
monoid. Suppose also that $F$ is
left translation invariant, that is, the function
$(L_sf)(x)=f(sx)$
belongs to $F$ for every $(f,s) \in F \times S$. Then the dual
action $S \times M(F) \to M(F)$ is well defined and each
$s$-translation, being the restriction of the adjoint operator
$L_s^*$, is continuous on $M(F)$.

We recall the fundamental definition of introverted subspaces which
was introduced by M.M. Day. We follow \cite{BJMo} and \cite
{BJM}.

\begin{defin}\label{def-day} (M.M. Day)
\ben
\item $F$ is \emph{left introverted}
if $m(F,F^*) \subset F$ (equivalently (Lemma \ref{l:state}.3),
$m(F,M(F)) \allowbreak \subset F$).
\item
When $F$ is an algebra then $F$ is said to be \emph{left
m-introverted} if $m(F,MM(F)) \allowbreak \subset F$.

\nt \textbf{Causion:} \emph{We usually say simply that $F$ is \emph{introverted} (rather than \emph{left introverted}).}
\een
\end{defin}



\begin{f} \label{f:product} (Evolution product (in the sense of J.S. Pym) \cite{BJMo}, \cite[Ch.2.2]{BJM}).
\ben
\item
If $F$ is an 
introverted closed subspace of $C(S)$ then $F^*$ is a
Banach algebra under the dual space norm and multiplication
 $(\mu, \varphi) \mapsto \mu \odot \varphi$, where
$$
(\mu \odot \varphi) (f): = \mu (m(f,\varphi))
\ \ (f \in F).
$$
Furthermore, with respect to the weak$^*$ topology, $F^*$ is a right topological
affine semigroup, $(M(F), \odot)$ is a compact right topological
affine subsemigroup, \allowbreak $co(\delta(S)) \subset \Lambda(M(F))$
and $\delta: S \to M(F)$ is an affine semigroup compactification.
\item
If $F$ is an m-introverted closed subalgebra of $C(S)$
then $(MM(F), \odot)$ is a compact right topological subsemigroup
of $(M(F), \odot)$. Furthermore, $\delta: S \to MM(F)$ is a right
topological semigroup compactification.
See also subsection \ref{s:p-univ} below for another view of m-introverted algebras as the algebras corresponding to enveloping semigroups.
 \een
\end{f}

The following result shows that 
the m-introverted algebras and 
introverted subspaces of $C(S)$
correspond to
the semigroup compactifications, and affine
semigroup compactifications of $S$
respectively (compare Proposition \ref{bar}).

\begin{f} \label{f:adm} Let $S$ be a semitopological semigroup.
 \ben
\item \cite[p. 108]{BJM}
 If $(\psi, K)$ is a semigroup compactification of $S$
 then $\psi^* C(K)$
is an 
m-introverted closed subalgebra of $C(S)$. Conversely, if $F$ is an 
m-introverted closed subalgebra of $C(S)$ then there exists a
unique
semigroup compactification $(\psi,K)$ of $S$ such that $\psi^*
C(K)=F$. Namely
the canonical semigroup compactification $\delta: S \to MM(F)$.
\item \cite[p. 123]{BJM}
If $(\psi,K)$ is an affine semigroup compactification
of $S$ then $\psi^* \mathrm{A}(K)$ is an 
introverted closed subspace
of $\mathrm{WRUC}(S)$. Conversely, if $F$ is an 
introverted closed subspace
of $\mathrm{WRUC}(S)$ then there exists a unique
affine semigroup compactification $(\psi,K)$ of $S$ such that
$\psi^* \mathrm{A}(K)=F$. Namely
the canonical affine semigroup compactification $\delta: S \to M(F)$.
\een
\end{f}

Propositions \ref{l:W=Univ}.1 and  \ref{p:dominates}.1 cover Fact \ref{f:adm}.2.

\begin{remark} \label{r:introInWRUC}
\cite[p. 123 and p. 172]{BJM} $\mathrm{WRUC}(S)$ is the largest introverted subspace of $C(S)$.  
\end{remark}

The next proposition demonstrates the universality of the standard operator affine semigroup compactifications.

\begin{prop} \label{l:W=Univ} \
\ben
\item
For every introverted closed unital subspace $F$ of $C(S)$ there exists a natural $co (\delta(S))$-admissible
affine embedding $M(F) \hookrightarrow \E(F)$ of right topological compact affine semigroups.
\item
Every affine semigroup compactification $\a: S \to Q$
is equivalent to a standard operator affine semigroup compactification $\a': S \to Q' \subset \E(V)$ for some Banach space $V$.
\een
\end{prop}
\begin{proof}
(1) The following natural map
$$
i: M(F) \hookrightarrow \E(F), \ \ \ i(m)(\varphi)=m \odot \varphi \ \ \forall \varphi \in F^*
$$
is the desired embedding, where $m \odot \varphi$ is the evolution product (Fact \ref{f:product}.1).
The continuity is easy to verify,
and the injectivity follows from the fact that if $e$ is the neutral element of $S$
then $\delta(e)$ is the neutral element of $M(F)$.

(2) is a conclusion from (1) and Fact \ref{f:adm}.2.
\end{proof}


\subsection{Point-universality of systems}
\label{s:p-univ}

For a point-transitive compact separately continuous $S$-system $(X,x_0)$ consider the natural $S$-compactification
map $j_{x_0}: S \to X, s \mapsto sx_0$ and the corresponding
Banach algebra embedding $j_{x_0}^*: C(X) \hookrightarrow C(S)$. Denote $\mathcal{A}(X,x_0)=j_{x_0}^*(C(X))$.
The enveloping semigroup $(E(X),e)$
is always a point-transitive (separately continuous) $S$-space.
Hence, $\mathcal{A}(E(X),e)$ is well defined and $E(X) \to X, p \mapsto px_0$ is the
natural surjective $S$-map.
Clearly $\mathcal{A}(X,x_0) \subset \mathcal{A}(E(X),e)$.

Recall that a point-transitive $S$-flow $(X,x_0)$ is said to be \emph{point-universal} \cite{GM1, GM-suc} if
it has the property that for every $x\in X$ there is a homomorphism $\pi_x:(X,x_0)\to ({\cls}(Sx),x)$.
This is the case iff
$S \to X, \ g \mapsto gx_0$ is a right topological semigroup compactification of
$S$;
iff $(X,x_0)$ is isomorphic to its own enveloping semigroup, $(X,x_0) \cong (E(X),e)$;
iff the algebra $\Acal(X,x_0)$ is m-introverted.


%
%



%

\section{Operator enveloping semigroups}
\label{s:Op-env}


\subsection{The notion of E-compatibility}
\label{s:motivation}

In a review article \cite[p. 212]{Pym90} J. Pym  asks the general
question: ``how affine flows might be obtained?"
and then singles out the canonical construction where,
with a given compact $S$-flow $X$ one associates the induced
affine flow on $P(X)$, the compact
convex space of probability measures on $X$, and where $X$
is naturally embedded into $P(X)$ by identifying the points
of $X$ with the corresponding dirac measures. Then $P(X)$ is at least $S_d$-space with respect to
the induced affine action $S \times P(X) \to P(X)$. Recall that by Proposition \ref{bar}.2,
$P(X)$ is an $S$-space (i.e., the action is separately continuous)
iff $X$ is WRUC.

In turn, $P(X)$ can be viewed (via Riesz'
representation theorem) as a part of the weak$^*$ compact unit
ball $B^*$ in the dual space $C(X)^*$. So we have the embeddings of $S_d$-systems
$$B^*\supset P(X)\supset X.$$
These embeddings induce the continuous onto homomorphisms of the 
enveloping semigroups
$$
E(B^*) \to E(P(X)) \to E(X).
$$
The first homomorphism $E(B^*) \to E(P(X))$ is always an isomorphism (Lemma \ref{l:state}.3).
Pym
asks when the second homomorphism $\phi: E(P(X)) \to
E(X)$ is an isomorphism.
The first systematic study of this question is to be found
in a paper of K\"{o}hler \cite{Ko}.
Since $\phi$ is an isomorphism iff it is injective, following
\cite{Gl-tame}, we say that an $S$-system $X$ (with continuous action) is \emph{injective}
when $\phi$ is an isomorphism.
See Definition \ref{d:inj} for a more general version.

For cascades (dynamical $\Z$-systems) the first non-injective example was constructed by Glasner \cite{Gl-tame},
answering a question of K\"{o}hler \cite{Ko}.
Earlier Immervoll \cite{Imm} gave an example of a non-injective system $(S,X)$ where $S$ is a some special
semigroup $S$.

Now we turn to a more general question. In the construction above
instead of the Banach space $C(X)$ and the natural embeddings
$X \subset P(X) \subset B^*$ one may consider
representations on general Banach spaces $V$.

\begin{question}
When is the enveloping semigroup of an affine compactification of
a compact system $X$, arising from a representation on a Banach space $V$,
isomorphic to the enveloping semigroup of the system itself ?
\end{question}

More precisely, let
$$
(h,\a): (S,X) \rightrightarrows (\Theta(V)^{op}, V^*)
$$
be a weakly continuous representation 
of a (not necessarily compact) $S$-system $X$ on a Banach space $V$.
It induces an $S$-compactification $X \to Y$, where $Y:=\overline{\a(X)}$
and an $S$-affine compactification $X \to Q$,
where $Q:= \overline{co} \ \a(X)$.
Since $h$ is weakly continuous it follows that the action of $S$ on
the weak$^*$ compact unit ball $B^*$ (hence also on $Y$ and $Q$) is
separately continuous.

By our definitions, $\a(X)$ and hence, $Y$ and $Q$ are norm bounded. So for some $r>0$ we have the embeddings of $S$-systems:
$$rB^* \supset Q \supset Y=\overline{\a(X)}.$$
By Lemma \ref{l:Env-all}, we get the induced continuous surjective homomorphisms of the 
enveloping semigroups (of course, the $S$-spaces $B^*$ and $rB^*$ are isomorphic):
$$
\psi: E(B^*) \to E(Q), \ \ \Phi: E(Q) \to E(Y).
$$

\begin{lem} \label{l:algebras} \
\ben
\item
$m(V, \overline{sp}^{norm}(A)) \subset \overline{sp}^{norm} (m(V,A))$ for every subset
$A \subset V^*$. 
\item The algebra of the compactification $S \to E(Y)$ is $\A(E(Y),e)=\lan m(V,Y) \ran$.
\item The algebra of the compactification $S \to E(Q)$ is $\A(E(Q),e)=\lan m(V,Q) \ran$.
\item The algebra of the compactification $S \to E(B^*)$ is
$\A(E(B^*),e)= \lan m(V,V^*) \ran=A_h,$ where $A_h$ is the coefficient algebra.
\een
\end{lem}
\begin{proof} (1) is straightforward. For other assertions
use (1) and Lemma \ref{l:Env-all}.1 taking into account the definitions of Section \ref{s:p-univ}.
\end{proof}

 \begin{defin} \label{d:E-compat}
Let $X$ be an 
$S$-flow.
\ben
\item We say that an \emph{$S$-affine compactification} (Definition \ref{d:affinization}.2)
$\a: X \to Q$ is \emph{E-compatible} if the map
$\Phi: E(Q) \to E(Y)$ is an isomorphism
(equivalently, is injective), where $Y:=\overline{\a(X)}$.
By Lemma \ref{l:algebras} it is equivalent
to saying that if $\A(E(Q),e) \subset \A(E(Y),e)$ or if $m(V,Q) \subset \lan m(V,Y) \ran$.
\item
We say that a weakly continuous Banach representation $(h,\a): (S,X) \rightrightarrows
(\Theta(V)^{op},V^*)$ of an
$S$-flow $X$ on a Banach space $V$ is:
\begin{itemize}
\item [(a)]
$E$-\emph{compatible} if the map $\Phi: E(Q) \to E(Y)$ is an isomorphism where
$Q:= \overline{co} \ \a(X)$. 
That is, if the induced affine compactification of the representation $(h,\a)$ is E-compatible.
\item [(b)]
\emph{Strongly $E$-compatible} if the map $\Phi \circ \Psi: E(B^*) \to E(Y)$ is an isomorphism.
It is equivalent to saying that
$m(V,V^*) \subset \lan m(V,Y) \ran$ (equivalently, $m(V,V^*) \subset \Acal(E(Y),e)$).
\end{itemize}
\een
\end{defin}


%


We say that $K \subset V^*$ is a $w^*$-\emph{generating} subset of $V^*$ if
$sp(\overline{co}^{w^*}(K))$ is norm dense in $V^*$.
A representation $(h,\a)$ of a system $(S,X)$ on $V$ is
$w^*$-\emph{generating} (or simply \emph{generating})
if $\a(X)$ is a $w^*$-generating subset of $V^*$.
Later on (in the proof of Theorem \ref{t:tame=R-repr}) we will have the occasion
to use the versatile construction of Davis-Figiel-Johnson-Pelczy\'nski \cite{DFJP}. The second item of the next lemma refers to this construction.

\begin{lem} \label{l:generating} \
\ben
\item
For every space $X$ and
a closed unital linear subspace $V \subset C(X)$ the regular
$V$-representation $\a: X \to V^*$ is generating.
\item
Let $E$ be a Banach space and let $\| \ \|_n, n \in \N$, be a sequence
of norms on $E$ where each of the norms is equivalent to the given norm
of $E$. For $v\in E,$ let
$$
N(v):=\left(\sum^\infty_{n=1} \| v \|^2_n\right)^{1/2} \hskip
0.1cm \text{and} \hskip 0.1cm \hskip 0.1cm V: = \{ v \in E
\bigm| N(v) < \infty \}.
$$
Denote by $j: V \hookrightarrow E$ the inclusion map. Then
$(V,N)$ is a Banach space, $j: V \to E$ is a continuous linear
injection such that $j^*: E^* \to V^*$ is norm dense. If $E=C(X)$ and $\a=j^* \circ \delta: X \to V^*$ is the induced map
then $\a(X)$ is a $w^*$-generating subset of $V^*$.
\een
\end{lem}
\begin{proof} (1)
Indeed by Lemma \ref{l:state} every
$\mu \in V^*$ is a finite linear combination of members of $M(V)=\overline{co}^{w^*} (\a(X))$.
Hence, $sp(\overline{co}^{w^*} (\a(X)) = V^*$.

(2) For the proof that $j^*: E^* \to V^*$ is norm-dense see Fabian \cite[Lemma 1.2.2]{Fa}.
Since $\delta(X)$ affinely generates $M(C(X)^*)$ and $M(C(X)^*)$ linearly spans $C(X)^*$
(Lemma \ref{l:state}) it follows that $sp (\overline{co}^{w^*} (\a(X)))$
is norm dense in $V^*$, where $\a (X)=j^* (\delta (X))$.
So, $\a(X)$ is $w^*$-generating in $V^*$.
\end{proof}

\begin{lem} \label{norm-cl}  \
\ben
\item
Suppose $\overline{co}^{w^*} (Y)=\overline{co}^{norm} (Y)$ holds for a weakly continuous representation $(h,\a)$ on a Banach space $V$,
where $Y:=\overline{\a(X)}^{w^*}$.
Then the representation is E-compatible.
\item
For $w^*$-generating representations, E-compatibility implies strong E-compatibility.
\item
For every regular representation
of an $S$-space $X$ on a closed unital $S$-invariant linear subspace
$V \subset \mathrm{WRUC}(X)$,
E-compatibility implies strong E-compatibility.
\item
(Monotonicity)
Let $\a_1$ and $\a_2$ be two \underline{faithful} $S$-affine compactifications of a compact $S$-space $X$ such that $\a_1 \succeq \a_2$,
then $E$-compatibility of $\a_1$ implies $E$-compatibility of $\a_2$.
\een
\end{lem}
\begin{proof} (1) By Lemma \ref{l:algebras}.2 we have $\overline{sp}^{norm} (m(V,Y)) \subset \A(E(Y),e)$.
By our assumption on $Q:=\overline{co}^{w^*} (Y)=\overline{co}^{norm} (Y)$ and using Lemma \ref{l:algebras}.1 we get
$$m(V, Q)=m(V, \overline{co}^{w^*} (Y))=m(V, \overline{co}^{norm} (Y)) \subset \overline{sp}^{norm} (m(V,Y)) \subset \A(E(Y),e).$$
Since $\lan m(V, Q) \ran=\A(E(Q),e)$ (Lemma \ref{l:algebras}.2), we obtain $\A(E(Q),e) \subset \A(E(Y),e)$.

(2)
$Y$ is a $w^*$-generating subset in $V^*$. Therefore
by assertions (1) and (3) of Lemma \ref{l:algebras} we get that $\psi: E(B^*) \to E(Q)$ is an isomorphism. So
$\Phi \circ \Psi: E(B^*) \to E(Y)$ is an isomorphism iff $\Phi: E(Q) \to E(Y)$ is an isomorphism.

(3) Use (2) and Lemma \ref{l:generating}.1.

(4)
Since the affine compactifications are faithful and $X$ is compact, we may identify $E(\a_1(X))$ and $E(\a_2(X))$ with $E(X)$.
Since $\a_1 \succeq \a_2$ we have the induced homomorphism $E(Q_1) \to E(Q_2)$.
The injectivity of the homomorphism $E(Q_1) \to E(X)$
implies the injectivity of $E(Q_2) \to E(X)$.
\end{proof}

Note that (4) need not remain true if we drop the faithfulness of the affine compactifications.

\begin{lem} \label{l:E-compatible}
Let $X$ be an $S$-space and
$(h,\a): (S,X) \rightrightarrows (\Theta(V)^{op},V^*)$ a weakly continuous
representation of $(S,X)$ on a Banach space $V$.
Let $j: S \to E(Y)$ be the Ellis compactification for the $S$-system $Y:=\overline{\a(X)}$.
\TFAE
\ben
\item
The representation $(h,\a)$ is strongly E-compatible.
\item
There exists a $j(S)$-admissible embedding $h': E(Y) \hookrightarrow \E(V)$ such that $h'\circ j=h$.
\item
The semigroup compactifications \ $j: S \to E(Y)$ and $h:S \to \overline{h(S)}$,
where $\overline{h(S)}$ is the closure in $\E(V)$, are naturally isomorphic.
\item
$m(V,V^*) \subset \lan m(V,Y) \ran$ (equivalently, $m(V,V^*) \subset \Acal(E(Y),e)$).
\een
\end{lem}

\begin{proof}
By Lemma \ref{l:W=E}, $E(S,B^*)$ is naturally embedded into $\E(V)$ and
 the semigroup compactifications $S \to
E(S,B^*)$ and $S \to \overline{h(S)}$ are isomorphic. So (1), (2) and (3) are equivalent.

(1) $\Leftrightarrow$ (4): Use Lemma \ref{l:algebras}.
\end{proof}


\begin{prop} \label{p:E-aff+}
\TFAE
\ben
\item
An affine $S$-compactification $\a: X \to Q$ of an $S$-space $X$ is
E-compatible.
\item
The induced representation of $(S,X)$ on the Banach space
$V:= A(Q)|_X \subset \mathrm{WRUC}(X)$, the affine compactification space of
$\a$, is E-compatible (equivalently, strongly E-compatible).
\item
$m(V,V^*) \subset \lan m(V,Y) \ran$, (equivalently,  $m(V,M(V)) \subset \lan m(V,Y) \ran$), where $Y:=\overline{\a(X)}$.
\item
For every $f \in V$ we have
$$
\overline{co (X_f)}^p \subset \lan m(V, Y) \ran.
$$
\een
\end{prop}
\begin{proof}
  (1) $\Leftrightarrow$ (2)
By Definition \ref{d:E-compat} and the description of affine compactifications (Lemma \ref{aff-descr})
taking into account Lemma \ref{norm-cl}.3.

(2) $\Leftrightarrow$ (3):
Use Lemma \ref{l:E-compatible} (taking into account that by Lemma \ref{l:state}.3
every $\mu \in V^*$ is a finite linear combination of members of $M(V)$).

(3) $\Leftrightarrow$ (4): Clearly, $m(V,M(V))=\bigcup_{f \in V} m(f,M(V)).$
Recall that, by Lemma \ref{l:prop}, for each $f \in V$ we have
$\overline{co (X_f)}^p=m(f,M(V)).$
\end{proof}


\subsection{Affine semigroup compactifications}

The second assertion of the next result shows that every affine semigroup compactification is E-compatible. 

\begin{prop} \label{p:dominates}
Let $\nu: S \to Q$ be an affine semigroup compactification
and let $P = \overline{\nu(S)}$.
Then
\ben
\item \cite[p. 123]{BJM}
The space $V=\mathrm{A}(Q)|_S$
is introverted.
\item
The affine compactification $\nu: S \to Q$ is E-compatible,
that is, the restriction map $E(Q) \to E(P)=P$ is an isomorphism.
\een
\end{prop}
\begin{proof}
(1)
%
We have to show that $m(V,M(V)) \subset V$.
As in the proof of Lemma \ref{aff-descr} consider the Banach space
$V=\mathrm{A}(Q)|_S$ of the affine compactification $S \to Q$.
For every $f \in V$ and $\mu \in M(V)=Q$ the corresponding matrix coefficient
$m(f,\mu)$ is again in $V$ because
$m(f,\mu)$ is a restriction to $S$ of the affine continuous map
$Q \to \R, \ q \mapsto \tilde{f}(q \odot \mu)$,
where $\tilde{f} \in \mathrm{A}(Q)$ with $f=\tilde{f}|_S$ and $q \odot \mu$ is the evolution product
(see Fact \ref{f:product}.1) in the semigroup $M(V)=Q$.

(2)
By (1) we have $m(V,V^*) \subset V$. Observe that
$V=\mathrm{A}(Q)|_S \subset \Acal(P,e)$. Since $\Acal(P,e) \subset \Acal(E(P),e)$, we get
 $m(V,V^*) \subset \Acal(E(P),e)=\lan m(V,P) \ran$ (Lemma \ref{l:algebras}.2). So, $V$ is E-compatible by Proposition \ref{p:E-aff+}.
%
%
\end{proof}

\begin{defin} \label{d:intro-gen}
We say that a
subalgebra
$\A \subset C(S)$ is \emph{intro-generated} 
if there exists an introverted subspace $V \subset \A$ such that
$\langle V \rangle=\A$.
\end{defin}

Below, in Theorem \ref{e:dist-not}, we show that the m-introverted algebra
of all distal functions $D(\Z)$ is not intro-generated.
In Example \ref{e:z2} we present an m-introverted intro-generated subalgebra $\A$ of $l_{\infty}(\Z^2)$ which is not introverted.


\begin{prop} \label{t:conditions-new}
Let $\nu:S \to P$ be a right topological semigroup compactification
and let $\A =\Acal(P,e)$ be the corresponding m-introverted subalgebra of $C(S)$.
\TFAE
\ben
\item
The compactification $\nu:S \to P$ is equivalent to a standard operator
compactification on a Banach space.
\item
The compactification $\nu:S \to P$ is equivalent to an operator
compactification on a locally convex vector space.
\item
There exists an affine (equivalently, \emph{standard affine}) semigroup compactification $\psi: S \to Q$
such that
the compactification $\psi : S \to  \overline{\psi(S)}$ is equivalent to
$\nu:S \to P$.
\item
The algebra $\A$
of the compactification $\nu:S \to P$ is intro-generated.
\item
There exists a Banach unital $S$-subspace $V \subset \mathrm{WRUC}(S)$
such that
$\langle V \rangle=\Acal$ and
for every $f \in V$ we have
$$
\overline{co (Sf)}^p \subset \Acal.
$$


 \een
\end{prop}
\begin{proof}
(1) $\Rightarrow$ (2): Trivial.

(2) $\Rightarrow$ (3): By our assumption $\nu: S \to P$ is equivalent to an operator
compactification. Therefore there exists
a weakly continuous equicontinuous representation $h: S \to L(V)$ of a
semitopological semigroup $S$ on a locally convex vector
space $V$ such that $\nu$ can be identified with $h: S \to P$ where $P$ is the
weak$^*$ operator closure $\overline{h(S)^{op}}$ of the adjoint semigroup $h(S)^{op}
\subset L(V)^{op}$ in $L(V^*)$.
Consider the compact subsemigroup $Q:=\overline{co} \ \nu (S)=\overline{co} (P)$ of $L(V^*)$
(in weak$^*$ operator topology).
 Then the map $\psi: S \to Q, \
\psi(s)=\nu(s)$ is an affine compactification of $S$. Indeed,
$\psi(S) \subset \Lambda (Q)$ because $\nu(S) \subset
L(V)^{op}$. Observe that $\psi$ induces $\nu$ because $
\overline{\psi(S)} =P$ and $\psi(s)=\nu(s)$ for every $s \in S$.
(By Proposition \ref{l:W=Univ} we can assume that $\psi: S \to Q$ is a standard affine semigroup compactification.)

(3) $\Rightarrow$ (4): By Proposition \ref{p:dominates}.1
the space $V:=A(Q)|_S$ of the affine compactification $\psi$ is introverted. 
Always, the affine compactification subspace $V$ generates the induced compactification algebra $\A$ (Remark \ref{r:affAnd}.1).
That is, $\langle V \rangle=\A$.
 Hence $\A$ is intro-generated.

(4) $\Rightarrow$ (5): If $V$ is an introverted subspace of $\A$ then $m(f,M(V)) \subset V$ for every $f \in V$.
Also, $V \subset \mathrm{WRUC}(S)$ by Remark \ref{r:introInWRUC}.
So Lemma \ref{l:prop} implies $m(f,M(V))=\overline{co (Sf)}^p$. Hence, $\overline{co (Sf)}^p \subset V \subset \langle V \rangle=\A$.

(5) $\Rightarrow$ (1):
Consider the regular representation $(h,\a)$
of the $S$-space $X:=S$ on the closed unital $S$-invariant linear subspace
$V \subset \mathrm{WRUC}(S)$. By the Stone-Weierstrass theorem it follows that the algebra of
the corresponding $S$-compactification $\a: S \to Y:=\overline{\a(X)}$ is $\lan V \ran$. By our assumption $\lan V \ran = \A$.
So we obtain that
$\a: S \to Y$ can be identified with the original $S$-compactification $\nu:S \to P$. Recall that
$\overline{co (X_f)}^p=\overline{co (Sf)}^p$ (Lemma \ref{l:prop})
 and $\Acal(E(P),e)=\Acal(P,e)=\A$ (Lemma \ref{l:Env-all}.4).
 Applying the equivalence (4) $\Leftrightarrow$ (1)
of Proposition \ref{p:E-aff+} we get that the regular $V$-representation $(h,\a)$ of $(S,S)$ is E-compatible.
In fact, strongly E-compatible by Lemma \ref{norm-cl}.3.
 We obtain that $E(S,B^*) \to E(S,P)$ is an isomorphism.
Now observe that $E(S,B^*) \subset \E(V)$ (Lemma \ref{l:W=E}.2) and $E(S,P)=P$  (Lemma \ref{l:Env-all}.4). 
\end{proof}

\subsection{Injectivity of compact dynamical systems}
\label{s:inj}

Every continuous action is WRUC (Lemma \ref{l:classes}.2). So the following 
naturally extends the definition from
\cite{Gl-tame}, mentioned in Section \ref{s:motivation}.

\begin{defin} \label{d:inj}
We say that a compact WRUC $S$-system $X$ is \emph{injective} if one
(hence all) of the following equivalent conditions
are
satisfied:

\ben
\item The greatest affine $S$-compactification $X \to P(X)$ is E-compatible.
\item
The regular representation of $(S,X)$ on the Banach space $C(X)$ is (strongly) E-compatible.
\item
$m(C(X),C(X)^*) \subset \lan m(C(X),X) \ran$ (equivalently, $m(C(X),C(X)^*) \subset \Acal(E(X),e)$).
\item Every faithful affine $S$-compactification $X \hookrightarrow Q$ is E-compatible.
\een
\end{defin}
\begin{proof} Here we prove that these conditions are equivalent.
First of all since $X$ is WRUC the action $S \times P(X) \to P(X)$ is separately continuous by Proposition \ref{bar}.2.
Regarding (2) note that
by Lemma \ref{norm-cl}.3 every regular E-compatible representation is strongly E-compatible.

(1) $\Leftrightarrow$ (2) and (4) $\Rightarrow$ (1): \ Are trivial.

(2) $\Leftrightarrow$ (3): Use Lemma \ref{l:E-compatible}.

(2) $\Rightarrow$ (4): \
Since $X \hookrightarrow P(X)$ is the greatest $S$-affine compactification of $X$
we can apply the monotonicity of $E$-compatibility (Lemma \ref{norm-cl}.4).
\end{proof}


\begin{prop} \label{p:from-inj}
Let $X$ be an injective 
$S$-system.
Then the enveloping semigroup compactification $S \to E(X)$
is equivalent to a (standard) operator compactification and hence
all of the equivalent conditions of Proposition \ref{t:conditions-new} are satisfied.
\end{prop}
\begin{proof} Since $X$ is injective the regular representation of $(S,X)$ on $C(X)$ is weakly continuous and E-compatible.
Now apply Lemma \ref{l:E-compatible}. 
\end{proof}



\begin{thm} \label{t:introv<->inj}
Let $V \subset C(S)$ be an m-introverted 
Banach $S$-subalgebra and $\a: S \to P$ be the corresponding right topological semigroup compactification.
Then the $S$-system $P$ is injective if and only if $V$ is introverted.
\end{thm}
\begin{proof} By Definition \ref{d:inj} the
$S$-flow $P$ is injective iff $m(C(P),C(P)) \subset \Acal(E(S,P),e)$. By Lemma
\ref{l:Env-all}.4,  $E(S,P)$ is naturally isomorphic to the
semigroup $P$. 
Hence,
$\Acal(E(P),e)=\Acal(P.e)=j^*(C(P))=V$. Observe also that the canonical
representations of $(S,P)$ on $C(P)$ and on $V$ are naturally
isomorphic. In particular, $m(C(P),C(P)^*)=m(V,V^*)$. Summing up we
get: the $S$-flow $P$ is injective iff $m(V,V^*) \subset V$ iff $V$ is
introverted.
\end{proof}

\begin{lem} \label{l:E -> X}
Let $X$ be a compact
point-transitive
$S$-system.
If the enveloping semigroup $E(X)$,
as an $S$-flow, is injective then $X$ is also injective.
\end{lem}
\begin{proof}
Let $z \in X$ be a transitive point and let $q: E(X) \to X, q(p)=pz$ be the corresponding onto continuous $S$-map. It induces
the surjective homomorphism $q_E: E(B_{C(E(X))^*}) \to E(B_{C(X)^*})$. By the injectivity of $(S,E(X))$ the $S$-compactifications
$S \to E(E(X))$ and $S \to E(B_{C(E(X))^*})$ are equivalent. On the other hand, $S \to E(E(X))$ is equivalent to $S \to E(X)$.
It follows that $S \to E(X)$ dominates $S \to E(B_{C(X)^*})$. Conversely, $S \to E(B_{C(X)^*})$ clearly dominates $S \to E(X)$.
Therefore these compactifications are equivalent.
\end{proof}

\subsection{Some examples of injective dynamical systems}

\label{r:inj-ex}


Let $G$ be a topological group.
As in Remark \ref{r:univ} a property $\Pcal$ of continuous compact
$G$-systems is said to be {\it suppable} if it is preserved
by the operations of taking products and subsystems.
To every suppable property $\Pcal$ corresponds a universal point-transitive $G$-system
$(G,G^\Pcal)$ (see e.g. \cite[Proposition 2.9]{GM1}).

\begin{remark} \label{r:injEx} \
\ben
\item
Suppose $\mathcal{P}$ is a suppable property of dynamical systems
such that whenever $(G,X)$ has $\mathcal{P}$ then so does $(G,P(X))$.
It then follows immediately that the corresponding $\mathcal{P}$-universal
point-transitive system $Y=G^\Pcal$ is injective. Indeed, then the $G$-systems $P(Y)$ and
$E(P(Y))$ have $\mathcal{P}$. By the universality of $Y$
it is easy to see that $Y$ and $E(P(Y))$ are naturally isomorphic. On the other hand,
$Y$ and $E(Y)$ are naturally isomorphic (Lemma \ref{l:Env-all}.4). Hence, also $E(P(Y))$ and $E(Y)$ are isomorphic.
This way we see that, for example,
the universal point-transitive (i) equicontinuous, (ii) WAP, (iii) HNS and (iv) tame dynamical systems, are all injective.
\item
Another application of this principle is obtained by
regarding the class of $\Z$-flows having zero
topological entropy. It is easy to check that this property is suppable
and the fact that it is preserved under the functor $X \mapsto P(X)$
follows from a theorem
of Glasner and Weiss \cite{GW}.
\item
Let $\Omega = \{0,1\}^\Z$ be the $\{0,1\}$-Bernoulli system on $\Z$.
It is well known that the enveloping semigroup of $(\Z,\Omega)$ is
$\beta \Z$, the \v{C}ech-Stone compactification of $\Z$ (see
\cite[Exercise 1.25]{Gl}).
Since $\beta \Z$ is the universal enveloping semigroup, it follows that
$E(X) = E(\Omega) = \beta \Z$
for every point-transitive system $(\Z,X)$
which admits $(\Z,\Omega)$ as a factor
(e.g. every mixing $\Z$-subshift of finite type will have
$\beta \Z$ as its enveloping semigroup since it has a Cartesian power which
admits $\Omega$ as a factor \cite{B}).
Moreover we necessarily have in that case that also $E(P(X)) = \beta \Z$.
Thus every such $X$ is injective. (See \cite{Ko} and \cite{BGH}).
\item
Every tame $S$-system is injective (Theorem \ref{t:tame is inj}).
This result for metrizable systems was obtained by K\"{o}hler \cite{Ko}.
In \cite{Gl-tame} there is a simple proof which uses the fact that the
enveloping semigroup $E(X)$ of a tame metric system $X$ is a Fr\'{e}chet space.

\item Every transitive continuous $G$-system $X$ is a factor of the injective $G$-system $G^{\mathrm{RUC}}$ (see (1)).
Thus injectivity is not preserved by factors. The same
assertion holds
for subsystems (Remark \ref{r:iteration}).
\een
\end{remark}

\begin{remark} \label{r:intro} \
\ben
\item Theorem \ref{t:introv<->inj} shows that injectivity can serve as a key property
in providing introverted subalgebras of $C(S)$.
In particular, the algebras in Remark \ref{r:injEx}.1 are introverted.

It is well known (see \cite[Ch. III, Lemma 8.8]{BJMo}) that every $S$-invariant unital closed
subspace of $\mathrm{WAP}(S)$ is introverted. In particular, $\mathrm{Hilb}(S)$ and $\mathrm{AP}(S)$ are introverted.
It is also well known that  
$\mathrm{RUC}(S)$ is an introverted algebra (see \cite[p. 163]{BJM}).
Hence, the corresponding semigroup
compactifications, the greatest ambit $S \to S^\mathrm{{RUC}}$, the Bohr compactification $S \to S^\mathrm{{AP}}$ and the
universal semitopological semigroup compactification $S \to S^\mathrm{{WAP}}$, respectively,
are operator compactifications (Proposition \ref{t:conditions-new})
and can be extended to affine semigroup compactifications.




\item
In Theorem \ref{m-intro to intro}, we show that in fact every m-introverted
Banach $S$-subspace of $\mathrm{Tame}(S)$ (hence also of
$\mathrm{Asp}(S)$ and $\mathrm{WAP}(S)$) is introverted.
\item
The Roelcke algebra
$\mathrm{UC}(G)=\mathrm{LUC}(G) \cap \mathrm{RUC}(G)$
is not even m-introverted in general, \cite{GM-suc}
for Polish topological groups $G$.
\een
\end{remark}

\subsection{The iteration process}
Starting with an arbitrary topological group $G$ and a compact dynamical system
$(G,X)$
(with continuous action)
 we define inductively a sequence of new systems by iterating the operation
of passing to the space of probability measures. Explicitly
we let $P^{(1)}(X)=P(X)$ and for $n \ge 1$ we let
$P^{(n+1)}(X) = P(P^{(n)}(X))$. Each $ P^{(n)}(X)$ is an affine
dynamical system and thus the barycenter map
$b : P^{(n+1)}(X) \to P^{(n)}(X)$ is a well defined continuous affine
homomorphism. Moreover, identifying a measure $\mu \in
P^{(n)}(X)$ with the dirac measure $\del_\mu \in P^{(n+1)}(X)$
we can consider $b : P^{(n+1)}(X) \to P^{(n)}(X)$ as a retract.
Next observe that the induced map
$b_* : P^{(n+1)}(X) \to P^{(n)}(X)$ coincides with
the map $b : P^{(n+1)}(X) \to P^{(n)}(X)$.
For convenience we write $Z_n = P^{(n)}(X)$
and we now let $Z = P^{(\infty)}(X) = \underset{\leftarrow}{\lim}\  Z_n$, the inverse limit.
We denote by $\pi_n : Z \to Z_n$ the natural projection.

\begin{prop}
There is a natural bijection $\al: P(Z) \to Z$. In particular
$(G,Z)$
 is injective.
\end{prop}

\begin{proof}
Given $\mu \in P(Z)$ we consider, for each $n$, its image
$z_{n+1}=(\pi_n)_*(\mu) =\mu_n \in P(Z_n) = P^{(n+1)}(X)$.
Then,
\begin{align*}
z_{n+2} & = \mu_{n+1} = (\pi_{n+1})_*(\mu)\\
&=
(b \circ \pi_n)_*(\mu)
=  b_* ((\pi_{n+1})_*(\mu))\\
& = b_*(z_{n+1}) = b(z_{n+1}).
\end{align*}
Therefore, the sequence $(z_n)$ defines a unique point $z=\al(\mu)$
in $Z$. Clearly $\al : P(Z) \to Z$ is a continuous
$G$-map
and it is easy to check that it is one-to-one. Finally for $z \in Z$
we have the sequence of measures $z_{n+1} = \pi_{n+1}(z) \in
Z_{n+1} = P^{(n+1)}(X) = P(Z_n)$.
Because, as maps from $P^{(n+1)}(X)$ to $P^{(n)}(X)$,
the maps $b$ and $b_*$ coincide, this choice of measures on the
various $Z_n$ is consistent and defines a measure $\beta(z) \in P(Z)$.
One can check now that $\beta\circ \al$ is the identity map on $P(Z)$
and our proof is complete.
\end{proof}

\begin{remark} \label{r:iteration}
We note that in the above construction the compact affine space
$Z$ is metrizable when $X$ is metrizable. Moreover, via the maps
$z \mapsto \del_z$ we have a natural embedding of $X$  into the space $Z$.
This iterated construction can serve now as a source of examples.
Beginning with an arbitrary compact metric system $X$ with a property,
say, $\Rcal$
--- which is preserved under inverse limits and such that $Y \in \Rcal \imp P(Y) \in \Rcal$ --- we obtain in
$(Z,T)$, with $Z = P^{(\infty)}(X)$, a metrizable injective system which
contains $X$ as a subsystem and has property $\Rcal$. Some properties $\Rcal$ as above
are e.g. ``weak-mixing" (\cite{BS}), ``zero topological entropy" (\cite{GW}) and ``uniform rigidity" (with respect to a given sequence) (\cite{GMa}).
\end{remark}



\section{Examples of non-injective systems}
\subsection{Minimal distal non-equicontinuos systems are not injective}

In this section $G$ denotes a semitopological group.
Let $Q$ be an affine compact $G$-system and let $ext\ Q$ denote the set
of extreme points of $Q$. The system $Q$ is said to be \emph{minimally generated} \cite{Gl-af} if
the $G$-subsystem $\overline{ext \ Q}$ is minimal.
We recall the following theorem.


\begin{f} \label{t:Glasner}
(\cite[Theorem 1.1]{Gl-af})
Let $Q$ be a minimally generated metric distal affine compact $G$-system. Then $Q$ is  equicontinuous.
\end{f}

\begin{lem} \label{l:min-gen}
Let $Q$ be a compact convex affine $G$-flow and let $X$ be a compact minimal $G$-subflow.
\TFAE
\ben
\item $X=\overline{ext \ Q}$.
\item $\overline{co} (X) = Q$.
\een
\end{lem}
\begin{proof}
(1) $\Rightarrow$ (2):   By Krein-Milman theorem (see for example 
\cite[p. 659]{Vr-b}), $Q=\overline{co} (ext \ Q)$.
Hence, we get $Q=\overline{co} (\overline{ext \ Q})=\overline{co} X$.

(2) $\Rightarrow$ (1): By Krein-Milman theorem $ext \ Q \neq \emptyset$. Choose $z \in ext \ Q$.
The orbit $Gz$ is dense in $X$ by minimality of $X$.
By Milman theorem \cite[p. 659]{Vr-b}, $ext \ Q\subset X$. Since $G$ is a \emph{group}
of affine transformations, $gz \in ext \ Q$ for every $g \in G$.
So $ext \ Q$ is dense in $X$.
\end{proof}

\begin{remark} \label{r:Glasnerthm}
Theorem \ref{t:Glasner} remains true for non-metrizable $Q$. Indeed, by \cite{Pen}
the general case can be reduced to the metrizable case using Ellis' construction.

At least for every separable $G$ a more direct argument is as follows.
We treat $Q$ as an affine compactification of the minimal system $X:=\overline{ext \ Q}$ ((1) $\Rightarrow$ (2) of Lemma \ref{l:min-gen}) and
observe that one may $G$-approximate the (faithful) affine compactification $X
\hookrightarrow Q$ by metric (not necessarily faithful) affine compactifications $X \to Q_i$.
More precisely, let $V \subset C(X)$ be the Banach unital $G$-invariant space of the affine compactification $X \hookrightarrow Q$.
Since $G$ is separable there are sufficiently many \emph{separable} Banach
unital $G$-invariant subspaces $V_i$  to separate points of $Q$.  The corresponding affine distal
$G$-factors $r_i: Q \to Q_i=\overline{co} (X_i), \ X_i:=r_i(X)$ are again minimally generated
((2) $\Rightarrow$ (1) of Lemma \ref{l:min-gen}).
Since each $Q_i$ is metrizable, it is equicontinuous by
Theorem \ref{t:Glasner}, and the same is then true for $Q$.
\end{remark}

\begin{prop} \label{p:not-E-fl}
Let $X$ be a compact distal minimal $G$-flow.
\TFAE
\ben
\item
The $G$-flow $X$ admits an E-compatible faithful affine compactification $\a: X \hookrightarrow Q$.
\item
The $G$-flow $X$ admits an E-compatible faithful representation of
$(S,X)$ on a Banach space.
\item
$X$ is an equicontinuous $G$-flow.
\item
$X$ is an injective $G$-flow.
\een
\end{prop}
\begin{proof}
(1) $\Leftrightarrow$ (2) Follows directly from Proposition \ref{p:E-aff+} and Definition \ref{d:E-compat}.

(1) $\Rightarrow$ (3) The distality of the $G$-flow $X$ means,
by Ellis result, that the enveloping semigroup $E(X)$ is a group. Since by (1), $E(Q) \to E(X)$ is an isomorphism
we obtain that $E(Q)$ is also a group, hence $Q$ is a distal $G$-flow. By Theorem \ref{t:Glasner},
taking into account Lemma \ref{l:min-gen} and Remark \ref{r:Glasnerthm} we conclude that $Q$, hence also, $X$ are equicontinuous flows.

(3) $\Rightarrow$ (4) This follows, for example, from  
Theorem \ref{t:tame is inj} below.

(4) $\Rightarrow$ (1) Apply Definition \ref{d:inj}.
\end{proof}

\begin{thm} \label{p:not-E-gr}
Let $P$ be a compact right topological \underline{group} and let
$\nu: G \to P$ be a right topological semigroup compactification of a group $G$.
\TFAE
\ben
\item
The compactification $\nu: G \to P$ 
is equivalent to an operator compactification.
\item There exists a $\nu(G)$-admissible embedding of $P$ into
$\E(V)$ for some Banach space $V$.
\item
$P$ is a topological group. \een
\end{thm}

\begin{proof} 
(3) $\Rightarrow$ (1): Every compact
topological group is embedded into the unitary group $U(H) \subset \Theta(H)$ for
some Hilbert space $H$ and in this case
 $\Theta(H) = \E(H^*)$.

(1) $\Rightarrow$ (3):
By Proposition \ref{t:conditions-new}, $G \to P$
can be embedded into an affine
semigroup compactification $G \to Q$, so that $\overline{co} (P) = Q$.
by Proposition \ref{p:dominates}, $G \to Q$ is $E$-compatible.
The system $(G,P)$ is distal because $E(G,P)=P$ is a group. Moreover, $(G,P)$ is minimal being distal and point-transitive.
 Now by Proposition \ref{p:not-E-fl} we conclude that $(G,P)$ is equicontinuous and hence $E(G,P)=P$ is a topological group.
%

(1) $\Leftrightarrow$ (2): Use Proposition \ref{t:conditions-new} and Lemma \ref{l:W=E}.
\end{proof}


Let $\mathrm{D}(\Z)$ be the algebra of all distal functions on $\Z$ and $\Z \to
\Z^{\mathrm{D}(\Z)}$ the corresponding semigroup compactification (see \cite[p. 178]{BJM}).
The right topological group $\Z^{\mathrm{D}(\Z)}$ is not a topological group so
by Proposition \ref{p:not-E-fl} and Theorem \ref{p:not-E-gr} we get

\begin{thm} \label{ex} \label{e:dist-not} \
\ben
\item
The semigroup compactification $\Z \to \Z^{\mathrm{D}(\Z)}$ is not an operator compactification.
\item
The algebra of all distal functions $\mathrm{D}(\Z)$ is not
intro-generated.
\item
$(\Z,\Z^{\mathrm{D}(\Z)})$ is a minimal 
distal cascade which
does not admit faithful E-compatible affine compactifications.
\item
The compact right topological group $\Z^{D(\Z)}$ does not admit faithful admissible representations on Banach spaces.


\een
\end{thm}

The fact that $\mathrm{D}(\Z)$ is not introverted
(weaker than Theorem \ref{e:dist-not}.2)
 was shown in  \cite[p. 179]{BJM}.


\subsection{A Toeplitz non-injective system}

We give here an example of a non-injective metric minimal $\Z$-system
which is a transitive almost 1-1 extension of an adding machine.
The latter property is actually a characterization of being a {\em Toepliz dynamical system}.
For more details on Toepliz systems we refer to \cite{W}.

\begin{thm} \label{e: Toeplic}
There is a Toeplitz non-injective system.
\end{thm}

\begin{proof}
Let $y_0 \in \{0,1\}^\Z$ be the ``Heaviside" sequence defined by the rule
$$
y_0(n) =
\begin{cases}
0 & {\text{for}} \ n < 0\\
1 & {\text{for}} \ n \ge 0,
\end{cases}
$$
and let $Y=\OC_S(y_0)$ be its orbit closure in $\{0,1\}^\Z$
under the shift transformation: $S \om (n) = \om(n+1),
\om \in \{0,1\}^\Z, n \in \Z$. Thus $Y =
\{S^n y_0: n \in \Z\} \cup \{\bf{0}, \bf{1}\}$ is isomorphic to the
$2$-point compactification $\Z \cup \{\pm \infty\}$ of the integers.

Let $(X,S)$
be the (necessarily minimal and
non-regular)
Toeplitz system corresponding to the subshift $Y$ and
a suitable sequence of periods $(p_i)_{i \in \N}$ as described by
Williams in Section 4 of \cite{W}.
Here, as usual, we write $(X,S)$
for the $\Z$-action $(\Z,X)$ generated by $S$.

By  \cite[Theorem 4.5]{W} the
system $(X,S)$ admits exactly two $S$-invariant ergodic
probability measures (corresponding to the dirac measures
$\del_{\bf{0}}$ and $\del_{\bf{1}}$ on $Y$), which we will denote
by $\mu_{0}$ and $\mu_1$, respectively.

We will show the existence of a probability measure $\nu$ on
$X$ for which, in the weak$^*$ topology on the compact space $P(X)$
of probability measures on $X$ and the action induced by $S$,
we have
\begin{equation}\label{pm}
\lim_{n \to -\infty} S^n \nu =\mu_{0} \quad \text{and}\quad
\lim_{n \to +\infty} S^n \nu =\mu_{1}.
\end{equation}

The dynamical system $(X,S)$ has a structure of an almost one-to-one extension of an adding machine, say,
$$
\pi : X \to G = \underset{\leftarrow}{\lim} \ \Z/p_i\Z.
$$
It follows that the proximal relation $Prox(X)$ on $X$ coincides with the
$\pi$-relation:
$$
R_\pi=\{(x,x'): \pi(x)=\pi(x')\}.
$$
In particular this implies that $Prox(X)$ is an equivalence relation. Now
this latter condition is equivalent to the fact that $E(X)$, the enveloping
semigroup of $(X,S)$ has a unique minimal ideal.
Since $E(E(X))=E(X)$ we conclude that in any dynamical system
$(X',T)$ whose enveloping semigroup is isomorphic to $E(X)$
the proximal relation $Prox(X')$ is again an equivalence relation.

However the equations (\ref{pm}) clearly show that in the dynamical system
$(P(X),S)$
the proximal relation is no longer an equivalence
relation. It therefore follows that the natural restriction map $r: E(P(X)) \to E(X)$
is not an isomorphism; i.e.
$(X,S)$
is not injective.

It thus remains to construct a measure $\nu$ as above. For that purpose
let us recall the following objects constructed by Williams.
Using the notations of \cite{W} we let $C =
\{x \in X : 0 \not\in Aper(x)\}$, and let $D \subset X$ be the set
of $x \in X$ with $Aper(x)$ a $2$-sided infinite sequence.
We have $C = \pi^{-1}(\pi(C))$ so that the subset $Aper(x)\subset \Z$
is well defined on $G$, with $Aper(x) = Aper(\pi(x))$ for every $x \in X$.
By definition the fact that $X$ is a non-regular Toeplitz system
means that we have $0 < m(\pi(C)) = d < 1$,  with $m$ denoting the Haar
measure on $G$.
The (Borel) dynamical system $(G \times Y,T)$ is given by the
Borel map $T :  G \times Y \to G \times Y$ defined as
$T(g,y) = (g +\ch, S^{\theta(g)}y)$, where
$\theta: G \to \{0,1\}$ is the function $\ch_{\pi(C)}$.
Williams shows that $\phi(G \times Y) = X$, that
$S \circ \phi = \phi \circ T$, and that the restriction of $\phi$
to the subset $\pi(D) \times Y$ is a Borel isomorphism from
$\pi(D) \times Y$ into $X$.
In our case we have
$$
\pi(D) \times Y = (\pi(D) \times \{{\bf{0}}\}) \cup
\left(\bigcup_{n \in \Z} \pi(D) \times \{S^n y_0\}\right)
\cup (\pi(D) \times \{{\bf{1}}\}).
$$
We also have
$m(\pi(D))=1$ and $\phi(m \times \del_{{\bf{0}}}) =\mu_0,\
\phi(m \times \del_{{\bf{1}}}) =\mu_1$.

Now let
$$
\nu = \phi(m \times \del_{y_0}).
$$

Iterating the map $T$ we see that for $n \in \Z$
$$
T^n(g,y) = (g + n\ch, S^{\theta_n(g)}y),
$$
where for every $g \in G$, $\theta_0(g) =  0$ and
$$
\theta_n(g) =
\begin{cases}
\theta(g) + \theta(g +\ch) + \cdots + \theta(g + (n-1) \ch) &
{\text{for}} \ n \ge 1\\
-\theta(g + n\ch) - \theta(g + (n + 1)) \ch - \cdots - \theta(g - \ch)
& {\text{for}} \ n \le -1.
\end{cases}
$$
Note that by the ergodic theorem we have
\begin{equation}\label{et}
\lim_{n \to \pm \infty} \frac 1 n \theta_n(g) = m(\pi(C)), \qquad m {\text{-a.e.}}.
\end{equation}

Next let us consider the integrals $\int f \, dS^n \nu$ for a fixed
continuous real valued function $f$ on $X$.
We have
\begin{align*}
\int_X f(x) \, dS^n \nu(x) & = \int_X f(S^nx) \, d\nu(x)=
\int_{G \times Y} f \circ S^n \circ \phi (g,y) \, d(m \times \del_{y_0})\\
& = \int_{G \times Y} f \circ \phi (T^n(g,y)) \, d(m \times \del_{y_0})\\
& = \int_{G } f \circ \phi ((g + n\ch,S^{\theta_n(g)}y_0)) \, dm(g)\\
& = \sum_{j \in \Z} \int_{\{g: \theta_n(g) = j\}}
f \circ \phi ((g + n\ch,S^j y_0)) \,dm(g) .
\end{align*}
If we now further assume that the function $f$ depends only on
coordinates $i$ with $| i | \le N$ for a fixed $N$ then,
taking into account the way the map $\phi$ is defined
on $\pi(D) \times Y$ as well as (\ref{et}), we see that indeed
$$
\lim_{n\to -\infty} \int f \, dS^n \nu = \int f \,d\mu_0 \quad{\text{and}}
\quad \lim_{n\to +\infty} \int f \, dS^n \nu = \int f \,d\mu_1.
$$
Since the collection of functions $f$ depending on finitely many coordinates
is uniformly dense in $C(X)$ this proves (\ref{pm}) and our proof is complete.
\end{proof}

%
%
%

In view of the last two sections one would like to
know how minimal weakly-mixing systems behave with respect
to injectivity.

\begin{problem}
Construct examples of minimal weakly-mixing $\Z$-flows which are
injective (not injective).
\end{problem}

\subsection{A non-injective $\Z^2$-dynamical system which admits an $E$-compatible faithful affine compactification}





Let $X = \{0,1\}^\Z$ and let $\sigma$ denote the shift transformation on $X$.
Define two $\Z^2$-actions on $X$ by
$$
\Phi_{mn}x = \sigma^m x \quad {\text{and}} \quad
\Psi_{mn}x  = \sigma^n x.
$$
Since $E(X,\sigma)$ is canonically isomorphic to $\beta\Z$
we clearly have $E(X,\Phi) \cong E(X,\Psi) \cong \beta\Z$.
In particular it follows that the two $\Z^2$-systems $(X,\Phi)$ and
$(X,\Psi)$ are injective.
Next consider the product $\Z^2$-system $(X \times X, \Xi)$,
where the action is diagonal; i.e.
$$
\Xi_{mn}(x,y) = (\Phi_{mn}x,\Psi_{mn}y)=(\sigma^mx,\sigma^ny).
$$
The proof of the next claim is straightforward.

\begin{claim}
$E(X \times X,\Xi) \cong E(X,\Phi) \times E(X,\Psi)
\cong \beta \Z \times \beta \Z$.
Moreover, identifying an element $p$ of $E(X \times X,\Xi)$ with
a pair $p=(p_\Phi,p_\Psi)$, we have $p=(p_\Phi,p_\Psi)=
(p_\Phi,\id) \circ (\id,p_\Psi)$ and if $\Xi_{m_i n_i} \to p$ then,
$(\Phi_{m_i 0},\id)=(\Phi_{m_i 0},\Psi_{m_i 0}) = \Xi_{m_i 0} \to (p_\Phi,\id)$
and
$(\id,\Psi_{0 n_i}) =(\Phi_{0 n_i},\Psi_{0 n_i}) =\Xi_{0 n_i}\to (\id,p_\Psi)$.
\end{claim}

\begin{prop}\label{not-inj}
The product dynamical system $(\Z^2,X \times X)$ (of two different $\Z^2$-flows) is not injective.
\end{prop}

\begin{proof}
Suppose to the contrary that it is injective, i.e. that we have
$E(P(X \times X),\Xi) \cong E(X \times X,\Xi)
\cong E(X,\Phi) \times E(X,\Psi)$.
Let $\mu$ be the Bernoulli measure $(1/2(\del_0 + \del_1))^\Z$
on $X$, an element of $P(X)$.
We let $\Del_\mu$, an element of $P(X \times X)$, be the corresponding
graph measure on $X \times X$ defined as the push-forward of $\mu$ via
the diagonal map $x \mapsto (x,x)$.
Let $A = \{\Xi_{mm}: m \in \Z\}$ and let $p \in \overline{A}
\subset E(P(X \times X),\Xi)$ be any element which is not in $A$
so that $p = \lim \Xi_{m_i m_i}$ with $m_i\nearrow \infty$ a net in $\Z$ .
We let $p_\Phi$ denote its projection in $E(P(X),\Phi)$ and
$p_\Psi$ its projection in $E(P(X),\Psi)$.

\begin{claim}
\begin{enumerate}
\item
$p\Del_\mu=\Del_\mu$.
\item
$(p_\Phi,\id)\Del_\mu = \mu \times \mu$.
\item
$(\id, p_\Psi)\Del_\mu = \mu \times \mu$.
\end{enumerate}
\end{claim}
\begin{proof}
The first equality holds trivially as $\Del_\mu$ is $A$-invariant.
The second and third equalities follow from the fact that $(X,\mu,\sig)$
is mixing as a measure dynamical system.
\end{proof}

We complete the proof of the proposition by pointing out the
following absurd:
$$
\Del_\mu = p\Del_\mu = (p_\Phi,\id) \circ (\id, p_\Psi) \Del_\mu =
(p_\Phi,\id) \mu \times \mu = \mu \times \mu.
$$
\end{proof}

%
%


\begin{ex} \label{e:z2} \
\begin{enumerate}
\item
The semigroup $\Z^2$-compactification defined naturally by the
embedding
$$\nu: \Z^2 \to Y=\beta\Z \times \beta\Z$$
is not injective but admits an $E$-compatible faithful affine compactification.
\item
There exists an intro-generated m-introverted Banach subalgebra of $l_{\infty}(\Z^2)$ which is not
introverted.
\end{enumerate}
\end{ex}
\begin{proof}
Let $V$ be the Banach subspace of $l_{\infty}(\Z^2)$
consisting of functions of the form $f(x_1,x_2) = f_1(x_1) + f_2(x_2)$ with
$f_1, f_2 \in l_{\infty}(\Z)$. It is easy to see that $V$ is an introverted
$\Z^2$-subspace of $C(Y)$.
Furthermore, by the Stone-Weierstrass theorem, the closed algebra
$\langle V \rangle$ generated by $V$
is the algebra $\A\subset l_{\infty}(\Z^2)$ which corresponds to the compactification $\nu: \Z^2 \to Y=\beta\Z \times \beta\Z$.
In particular, by Propositions \ref{p:E-aff+} and  \ref{t:conditions-new},
$(\Z^2,Y)$ admits an $E$-compatible faithful affine compactification.
However, the $\Z^2$-flow $Y$ is not injective.
Indeed, $Y$ is just the enveloping semigroup $E(\Z^2,X \times X)$ of the $\Z^2$-flow
$X\times X$ from Proposition \ref{not-inj}. Therefore,
by Lemma \ref{l:E -> X} the injectivity of $Y$ will imply
(observe that $Y$ is transitive)
that $X \times X$ is injective contradicting Proposition \ref{not-inj}.
Finally, Theorem \ref{t:introv<->inj} shows that $\A$ is not introverted.
\end{proof}

The nonmetrizability of $Y$ in Example \ref{e:z2} and of
$\Z^{D(\Z)}$ in Theorem \ref{e:dist-not} is unavoidable by Theorem
\ref{t:tame-sc-case}.

\begin{problem} \label{q:Z?}
Are there examples as above with $\Z$ as the acting group, rather than
$\Z^2$ ?
\end{problem}


\section{Tame and HNS systems and related classes of right topological semigroups}

\subsection{Some classes of right topological semigroups}
\label{s:semigroups}

To the basic classes of right topological semigroups listed in \ref{d:sem1} above, we add the following two which have
naturally arisen in the study of tame and HNS dynamical systems.


\begin{defin} \label{d:sem2} \cite{GM1, GM-rose}
A compact admissible right topological semigroup $P$ is said to be:
\ben
\item \cite{GM-rose} \emph{tame} if the left translation $\lambda_a: P \to P$ is a fragmented map for every $a \in P$.
\item  \emph{HNS-semigroup} (\emph{$\F$-semigroup} in \cite{GM1}) if 
$\{\lambda_a : P \to P\}_{a \in P}$ is a fragmented family of maps.
\een
\end{defin}

These classes are closed under factors.
We have the inclusions:  
$$ \{\text{compact semitopological semigroups}\} \subset \{{\text{HNS-semigroups}}\} \subset
\text{\{Tame semigroups\}} 
$$


\begin{lem} \label{t:F} \
 \ben
\item
Every compact semitopological semigroup $P$ is a HNS-semigroup.
\item
Every HNS-semigroup is tame.
\item
If $P$ is a \emph{metrizable} compact right topological
admissible semigroup
then $P$ is a HNS-semigroup.


 \een
\end{lem}
\begin{proof}
(1) Apply Lemma \ref{l:FrFa}.1 to $P \times P \to P$.

(2) is trivial.

(3) Apply Lemma \ref{l:FrFa}.2 to $P \times P \to P$. 
\end{proof}



If $P$ is Fr\'{e}chet, as a topological space, then $P$ is a tame semigroup by Corollary \ref{c:FrechetSisTame} below.


\subsection{Compact semitopological semigroups and WAP systems}
\label{s:wap}

As usual, a continuous function on a (not necessarily compact) $S$-space $X$
is said to be \emph{weakly almost periodic} (WAP) if the weakly closure of the orbit $fS$ is weakly compact in
$C(X)$. Notation: $f \in \mathrm{WAP}(X)$. It is equivalent that $f=\tilde{f} \circ \a$ comes from an $S$-compactification
$\a: X \to Y$ such that $\tilde{f} \in \mathrm{WAP}(Y)$. In fact, one may choose the cyclic $S$-compactification $Y=X_f$.
A compact dynamical $S$-system $X$ is said to be WAP if
$C(X)=\mathrm{WAP}(X)$. The latter happens iff every element $p \in E(X)$ is
 a continuous selfmap of $X$ (Ellis and Nerurkar).

\begin{prop} \label{p:refl}
Let $V$ be a Banach space.
\TFAE
\ben
\item
$V$ is reflexive.
\item
The compact semigroup $\E$ is semitopological.
\item
$\E=\Theta^{op}$.
\item
$\Theta$ is compact with respect to the weak operator topology.
\item $(\Theta^{op}, B^*)$ is a WAP system.
\een
\end{prop}
\begin{proof}
Use Lemma \ref{l:transit}.1 and the standard characterizations of
reflexive Banach spaces. 
\end{proof}

\begin{f} \label{t:wap} \cite[Section 4]{Me-nz}
Let $S$ be a semitopological semigroup.
\ben
\item
A compact (continuous) $S$-space $X$ is WAP
if and only if $(S,X)$ is weakly (respectively, strongly) reflexively approximable.
\item
A compact (continuous) \emph{metric} $S$-space $X$ is WAP
if and only if $(S,X)$ is weakly (respectively, strongly) reflexively representable.
\een
\end{f}

We next recall a version of Lawson's theorem \cite{Lawson} and
its soft geometrical proof using representations of dynamical systems on reflexive spaces.

\begin{f} \label{f:Lawson}
\emph{(Ellis-Lawson's Joint Continuity Theorem)} Let $G$ be a subgroup
of a compact semitopological monoid $S$. Suppose that $S \times X
\to X$ is a separately continuous action with compact $X$. Then
the action $G\times X \to X$ is jointly continuous and $G$ is a
topological group.
\end{f}
\begin{proof} \nt \emph{A sketch of the proof from \cite{Me-nz}}: It is
easy to see by Grothendieck's Lemma that $C(X)=\mathrm{WAP}(X)$. Hence $(S,X)$ is a weakly almost periodic system.
By Theorem
\ref{t:wap} the proof can be reduced to the particular case where
$(S,X)=(\Theta(V)^{op},B_{V^*})$ for some reflexive Banach space $V$ with
$G:=\Iso(V)$, where $\Theta(V)^{op}$ 
is endowed with the weak operator topology.
By \cite{Me-op} 
the weak and strong operator topologies coincide on $\Iso(V)$ for reflexive $V$.
In particular, $G$ is
a topological group and it acts continuously on $B_{V^*}$.
\end{proof}

As a corollary one gets the classical result of Ellis. See also a generalization in Theorem \ref{gen-Ellis}.

\begin{f} \label{t:Ellis}
\emph{(Ellis Theorem)} Every compact semitopological group is a
topological group.
\end{f}


Another consequence of Theorem \ref{t:wap} (taking into account Proposition \ref{p:refl}) is

\begin{f} \label{f:sem-in-ref}
\emph{(\cite{Sh} and \cite{Me-op})}
Every compact semitopological semigroup $S$ is embedded into
$\Theta(V)=\E(V^*)$ for some reflexive $ V$.
\end{f}


Thus, compact semitopological semigroups $S$ can be
characterized as closed subsemigroups of $\E(V)$ for reflexive Banach spaces $V$.
We will show below, in Theorem \ref{t:Srepr}, that analogous statements
(for admissible embeddings)
hold for HNS and tame semigroups, where the corresponding classes of Banach spaces are Asplund and Rosenthal spaces respectively.

\subsection{HNS-semigroups and dynamical systems}
\label{s:HNS}

The following definition
(for continuous group actions) originated in \cite{GM1}.
One may extend it to separately continuous semigroup actions.

\begin{defin} \label{d:HNS}
We say that a compact $S$-system $X$ is \emph{hereditarily non-sensitive} (HNS, in short)
if one of the following equivalent conditions are satisfied:
\ben
\item
For every closed nonempty subset $A \subset X$ and for every entourage $\eps$ from the unique compatible uniformity on $X$ there exists
an open subset $O$ of $X$ such that $A \cap O$ is nonempty and $s (A \cap O)$ is $\eps$-small for every $s \in S$.
\item
The family of translations $\widetilde{S}:=\{\tilde{s}: X \to X\}_{s \in S}$ is a fragmented family of maps.
 \item
 $E(S,X)$ 
 is a fragmented family of maps from $X$ into itself.
 \een
\end{defin}

The equivalence of (1) and (2) is evident from the definitions.
Clearly, (3) implies (2). As to the implication (2) $\Rightarrow$ (3),
observe that the pointwise closure of a fragmented family is again
a fragmented family, \cite[Lemma 2.8]{GM-rose}.

Note that if $S=G$ is a group then in Definition \ref{d:HNS}.1 one may consider only closed subsets $A$
which are $G$-invariant (see the proof of \cite[Lemma 9.4]{GM1}).

\begin{lem} \label{l:HNS-prop} \
\ben
\item For every $S$
the class of HNS compact $S$-systems is closed under subsystems,
arbitrary products and factors.
\item
For every HNS compact $S$-system $X$
the corresponding enveloping semigroup $E(X)$ is
HNS both as an $S$-system and as a semigroup.

\item Let $P$ be a 
HNS-semigroup.
Assume that $j: S \to P$ be a continuous homomorphism from a semitopological semigroup $S$ into $P$
such that $j(S) \subset \Lambda(P)$. Then the 
$S$-system $P$ is HNS.
\item \{HNS-semigroups\}=\{enveloping semigroups of HNS systems\}.
\een
\end{lem}
\begin{proof}
(1) As in \cite{GM-rose} using the stability properties of fragmented families.

(2)
$(S,E)$ is a HNS system because HNS is preserved by subdirect products.
So, by Definition \ref{d:HNS}, $\{\lambda_a : E \to E\}_{a \in j(S)}$ is a fragmented family of maps.
Then its pointwise closure $\{\lambda_a : E \to E\}_{a \in E}$ is also a fragmented family.

(3)
Since $j(S) \subset \Lambda(P)$ the closure $\overline{j(S)}$ is a subsemigroup of $P$.
We can assume that $\overline{j(S)}=P$.
By Lemma \ref{l:Env-all}.4, the enveloping semigroup
$E(S,P) \subset P^P$ can be naturally
identified with $P$ so that every $a \in E(S,P)$ is identified with
the corresponding left translation $\lambda_a: P \to P$.
Since $P$ is a HNS-semigroup the set of all
left translations $\{\lambda_a: P \to P\}_{a \in E}$ is a fragmented family.
Hence, $(S,P)$ is a HNS system (Definition \ref{d:HNS}).

(4) Combine (2) and (3) taking into account Lemma \ref{l:Env-all}.4.
\end{proof}

\begin{thm} \label{t:F-semig} Let $V$ be a Banach space.
\TFAE \ben
\item
$V$ is an Asplund Banach space.
\item
$(\Theta^{op}, B^*)$ is a HNS system.
\item $\E$ is a HNS-semigroup.
\een
\end{thm}
\begin{proof}
(1) $\Rightarrow$ (2): Use Definition \ref{d:HNS}.2 and
the following well known characterization of Asplund spaces: $V$ is Asplund iff
$B^*$ is $(w^*,norm)$-fragmented (Fact \ref{f:Asp}).

(2) $\Rightarrow$ (1)
By Fact \ref{f:Asp} we have to show that $B$ is a fragmented family for $B^*$.
Choose a vector $v \in S_V$.
Since $\Theta^{op}$ is a fragmented family of self-maps
on $B^*$ and as $v: B^* \to \R$ is uniformly continuous we get that
the system $v \Theta^{op}=\Theta v$
of maps from $B^*$ to $\R$ is also
fragmented. Now recall that $\Theta v=B$ by Lemma \ref{l:transit}.1.

(2) $\Rightarrow$ (3): Follows from Lemma \ref{l:HNS-prop}.2
and the fact that $\E$ is the enveloping semigroup $E(\Theta^{op},B^*)$.

(3) $\Rightarrow$ (2):
$\Lambda(\E)=\Theta^{op}$ (Lemma \ref{l:transit}.5) and $\E$ is a HNS-semigroup.
So, $(S,\E)$ is HNS by Lemma \ref{l:HNS-prop}.3
 with $S=\Theta^{op}$.
Take $\psi \in B^*$ with $||\psi||=1$.
The map $q: \E \to B^*, \ p \mapsto p \psi$ defines
a continuous homomorphism of $\Theta^{op}$-systems. By Lemma
\ref{l:transit}.4, we have $\E \psi=B^*$. So $q$ is onto. Now observe that
the HNS property is preserved by factors of $S$-systems
(Lemma \ref{l:HNS-prop}.1).
\end{proof}

Our next theorem is based on ideas from \cite{GMU}.
\begin{thm} \label{t:metr-W}
Let $V$ be a Banach space. \TFAE
\ben
\item
$V$ is a separable Asplund space.
\item
$\E$ is homeomorphic to the Hilbert cube $[-1,1]^{\N}$ (for infinite-dimensional $V$).
\item
$\E$ is metrizable.
\een
\end{thm}
\begin{proof} 

(1) $\Rightarrow$ (2) Since $\E$ is a compact affine subset in the Fr\'{e}chet space $\R^{\N}$ we can use Keller's Theorem \cite[p. 100]{BP}.

(2) $\Rightarrow$ (3) Is trivial.

(3) $\Rightarrow$ (1) $\E$ is a HNS-semigroup by Lemma
\ref{t:F}.3. Now Theorem \ref{t:F-semig} implies that $V$ is
Asplund. It is also separable; indeed, by Lemma
\ref{l:transit}.4, $B^*$ is a continuous image of $\E$, so that
$B^*$ is also $w^*$-metrizable, which in turn yields the separability of $V$.
\end{proof}

Now in Theorem \ref{t:GMU}.2 we obtain a short proof of one of the main results of
\cite{GMU} (stated there for continuous group actions).


\begin{thm} \label{t:GMU}
Let $X$ be a compact 
$S$-system.  Consider the following assertions:
\bit
\item [(a)] $E(X)$ is metrizable.
\item [(b)]  $(S,X)$ is HNS.
\eit
Then:
\ben
\item \emph{(a) $\Rightarrow$ (b)}.
\item
If $X$, in addition, is metrizable then \emph{(a) $\Leftrightarrow$ (b)}.
\een
\end{thm}
\begin{proof}
(1) By Definition \ref{d:HNS} we have to show that
$E(X)$ is a fragmented family of maps from $X$ into itself.
The unique compatible uniformity on the compactum $X$ is the weakest uniformity on $X$ generated by $C(X)$.
Using Remark \ref{r:fr1}.1 one may reduce the proof to the verification of the following claim: $E^f:=\{f \circ p:  p \in E(X)\}$ is a
fragmented family for every $f \in C(X)$. In order to prove this claim apply Lemma \ref{l:FrFa}.2
to the induced mapping $E(X) \times X \to \R, (p,x) \mapsto f(px)$ (using our assumption
that $E(X)$ is metrizable).

(2) If $X$ is a metrizable HNS $S$-system then
by Theorem \ref{t:tame=R-repr} below, $(S,X)$ is representable on a separable Asplund space $V$.
We can assume that $X$ is $S$-embedded into $B^*$.
The enveloping semigroup $E(S,B^*)$ is embedded into
$\E$ (Lemma \ref{l:W=E}). The latter is metrizable by virtue of Theorem \ref{t:metr-W}.
Hence $E(S,X)$ is also metrizable, being a continuous image of $E(S,B^*)$.
\end{proof}

\begin{prop} \label{p:metrAlg}
Let $S$ be a semitopological semigroup and $\a: S \to P$ be a right topological semigroup compactification.
\ben
\item If $P$ is metrizable then $P$ is a HNS-semigroup and the system $(S,P)$ is HNS.
\item Let
$V \subset C(S)$ be an 
m-introverted closed subalgebra of $C(S)$. If $V$ is separable then necessarily $V \subset \mathrm{Asp}(S)$.
\een
\end{prop}
\begin{proof} (1)  By Lemma \ref{t:F}.3, $P$ is a HNS-semigroup. By Lemma \ref{l:HNS-prop}.3, the system $(S,P)$ is HNS.

(2)
By Fact \ref{f:adm}.1, the algebra $V$ induces a semigroup compactification $S \to P$. Since $V$ is separable, $P$ is metrizable.
So by (1), $(S,P)$ is HNS. Therefore, $V \subset \mathrm{Asp}(S)$. 
\end{proof}

\subsection{Tame semigroups and tame systems}
\label{s:tame}

\begin{defin} \label{d:Tame}
A compact separately continuous $S$-system $X$ is said to be
\emph{tame} if the translation $\lambda_a: X \to X,
\ x \mapsto a x$ is a fragmented map for every element $a \in E(X)$ of the enveloping semigroup.
\end{defin}

This definition is formulated in \cite{GM-rose} for continuous group actions.

According to Remark \ref{r:univ} we define, for every $S$-space $X$,
the $S$-subalgebras $\mathrm{Tame}(X)$ and $\mathrm{Tame_c}(X)$
 of $C(X)$.
Recall that in several natural cases we have $\Pcal_c(X) = \Pcal(X)$ (see Lemma \ref{l:classes}).

%

\begin{lem} \label{l:ourclasses}
Every WAP system is HNS and every HNS is tame. Therefore, for every semitopological semigroup $S$ and every $S$-space $X$
(in particular, for $X:=S$) we have
$$
\mathrm{WAP}(X) \subset \mathrm{Asp}(X) \subset \mathrm{Tame}(X) \ \ \ \
\mathrm{WAP_c}(X) \subset \mathrm{Asp_c}(X) \subset \mathrm{Tame_c}(X).
$$
\end{lem}
\begin{proof}
We can suppose that $X$ is compact.
If $(S,X)$ is WAP then $E(X) \times X \to X$ is separately continuous.
By Lemma \ref{l:FrFa}.1 we obtain that $E$ is a fragmented family of maps from $X$ to $X$. In particular,
its subfamily $\{\tilde{s}: X \to X\}_{s \in S}$ of all translations is fragmented. Hence, $(S,X)$ is HNS.

Directly from the definitions we conclude that every HNS is tame.
\end{proof}

Another proof of Lemma \ref{l:ourclasses} comes also from Banach representations
theory for dynamical systems because every reflexive space is Asplund and every Asplund is Rosenthal.

By \cite{GM-tame}, a compact metrizable $S$-system $X$ is tame
iff $S$ is \emph{eventually fragmented} on $X$, that is,
for every infinite (countable) subset $C \subset G$ there exists an infinite subset $K \subset C$
such that $K$ is a fragmented family of maps $X \to X$.

\begin{lem} \label{l:both} \
\ben
\item
For every $S$
the class of tame $S$-systems is closed under closed subsystems,
arbitrary products and factors.
\item
For every tame compact $S$-system $X$
the corresponding enveloping semigroup $E(X)$ is
tame both as an $S$-system and as a semigroup.
\item
Let $P$ be a tame right topological compact semigroup and
let $\nu: S \to P$ be a continuous homomorphism from a semitopological
semigroup $S$ into $P$ such that $\nu(S) \subset \Lambda(P)$.
Then the $S$-system $P$ is tame.
\item \{tame semigroups\}=\{enveloping semigroups of tame systems\}.
\een
\end{lem}
\begin{proof}
(1) As in \cite{GM-rose} using the stability properties of fragmented maps.

(2)
$(S,E)$ is a tame system because by (1) tameness is preserved by
subdirect products. Its enveloping semigroup can be identified
with $E$ itself (Lemma \ref{l:Env-all}.4), so that $\lambda_p: E \to E$
is fragmented for every $p \in E$.

(3) $\nu(S) \subset \Lambda(P)$, so $\overline{\nu(S)}$ is a semigroup. We can assume that $\overline{\nu(S)}=P$.
By Lemma \ref{l:Env-all}.4, the enveloping semigroup $E(S,P) \subset P^P$
can be naturally identified with $P$
in such a way that every $a \in E(S,P)$
is identified  with the corresponding left translation $\lambda_a: P \to P$ for some $a \in P$.
Since $P$ is a tame semigroup
every left translation $\lambda_a: P \to P$ is fragmented. Hence, $(S,P)$ is a tame system.
%
%

(4) Combine (2) and (3) taking into account Lemma \ref{l:Env-all}.4.
\end{proof}

\begin{prop} \label{p:tame-f}
Let $X$ be a compact $S$-space 
and $f \in C(X)$.
The following conditions are equivalent:
\begin{enumerate}
\item $f \in \mathrm{Tame}(X)$.
\item
${\cls}_p(fS) \subset {\mathcal F}(X)$ (i.e. the orbit $fS$ is a
Rosenthal family for $X$).
\end{enumerate}
\end{prop}
\begin{proof}
See \cite[Prop. 5.6]{GM-rose}.
%
\end{proof}

\begin{thm} \label{t:tame2} Let $V$ be a Banach space.
\TFAE
\ben
\item
$V$ is a Rosenthal Banach space.
\item $(\Theta^{op},B^*)$ is a tame system.
\item $p: B ^* \to B^*$ is a fragmented map for each $p \in \E$. 
\item
$\E$ is a tame 
semigroup.

\een
\end{thm}
\begin{proof}
(2) $\Leftrightarrow$ (3): Follows from the definition
of tame flows because $\E=E(\Theta^{op},B^*)$. 

(2) $\Rightarrow$ (4): Since $\E=E(\Theta^{op},B^*)$, Lemma \ref{l:both}.2 applies.

(4) $\Rightarrow$ (2): 
By our assumption, $\E$ is a tame semigroup.
Then by Lemma \ref{l:both}.3 the system $(\Theta^{op}, \E)$ is tame.
Its factor (Lemma \ref{l:transit}.4) $(\Theta^{op},B^*)$ is tame, too.

(2) $\Rightarrow$ (1): By a characterization of Rosenthal spaces
\cite[Prop. 4.12]{GM-rose} (see also Fact \ref{f:RosFr})
it suffices to show that $B^{**} \subset \F(B^*)$.
Since $(\Theta^{op},B^*)$ is tame, $p: B^* \to B^*$ is fragmented for every $p \in E(\Theta^{op},B^*)=\E$.
Pick an arbitrary $v \in B_V$ with $\|v\|=1$.
Then $v \E$ is exactly $B^{**}$ by Lemma \ref{l:transit}.2. So every $\phi \in B^{**}$ is a composition $v \circ p$, where
$p$ is a fragmented map. Since $v: B^* \to \R$ is weak$^*$ continuous we
conclude that $\phi: B^* \to B^*$ is fragmented.

(1) $\Rightarrow$ (3): We have to show that $\E \subset
\F(B^*,B^*)$ for every Rosenthal space $V$.
Let $p \in \E$. Then $p \in \Theta(V^*)$. That is, $p$ is a linear map $p: V^* \to
V^*$ with norm $\leq 1$. Then, for every vector $f \in V$, the
composition $f \circ p: V^* \to \R$ is a linear bounded 
functional on $V^*$. That is, $f \circ p  \in
V^{**}$ belongs to the second dual. Again, by the above mentioned characterization of Rosenthal spaces,
the corresponding restriction $f \circ p |_{B^*}: B^* \to \R$ is
a fragmented function for every $f \in V$. 
Since $V$ separates points of $B^*$ we can
apply \cite[Lemma 2.3.3]{GM-rose}. It follows that $p: B^* \to B^*$ is
fragmented for every $p \in \E$.
\end{proof}

\subsection{A dynamical BFT dichotomy}
Recall that a topological space $K$ is a {\it Rosenthal compactum}
\cite{Godefroy} if it is homeomorphic to a pointwise compact
subset of the space $\calB_1(X)$ of functions of the first Baire
class on a Polish space $X$. All metric compact spaces are
Rosenthal. An example of a separable non-metrizable Rosenthal
compactum is the {\it Helly compact}
of all
nondecreasing
selfmaps of $[0,1]$ in the pointwise topology.
Recall that a topological space $K$ is {\em Fr\'echet}
(or, {\em Fr\'echet-Urysohn}) 
 if for every $A\subset K$ and every $x\in cl(A)$ there exists a
sequence of elements of $A$ which converges to $x$.
Every Rosenthal compact space $K$ is Fr\'{e}chet by a result of Bourgain-Fremlin-Talagrand
\cite[Theorem 3F]{BFT}, generalizing a result of Rosenthal.

\begin{thm} \label{t:FrechetE}
If the enveloping semigroup $E(X)$ 
is a Fr\'{e}chet (e.g., Rosenthal) space, as a topological space,
then $(S,X)$ is a tame system (and $E(X)$ is a tame semigroup).
\end{thm}
\begin{proof}
Let $p \in E(X)$. We have to show that $p: X \to X$ is fragmented.
By properties of fragmented maps \cite[Lemma 2.3.3]{GM-rose}
it is enough to show that $f \circ p: X \to \R$ is fragmented for every $f \in C(X)$.
By the Fr\'echet property of $E(X)$ we may choose a sequence $s_n$ in $S$
such that the sequence $j(s_n)$ converges to $p$ in E(X).
Hence the sequence of continuous functions $f \circ s_n=f \circ j(s_n)$
converges pointwise  to $f \circ p$ in $\R^X$.
Apply Lemma \ref{l:FrFa}.2 to the evaluation map $F \times X \to \R$,
where $F:=\{f \circ p\} \cup \{f \circ j(s_n)\}_{n \in \N} \subset \R^X$ carries the pointwise topology.
 We conclude that $F$ is a fragmented family. In particular, $f \circ p$ is a fragmented map.
 ($E(X)$ is a tame semigroup by Lemma \ref{l:both}.2.)
\end{proof}

\begin{cor} \label{c:FrechetSisTame}
Let $P$ be a compact right topological admissible semigroup.
If $P$ is Fr\'{e}chet (e.g., when it is Rosenthal), as a topological space,
then $P$ is a tame semigroup.
\end{cor}
\begin{proof}
Applying Theorem \ref{t:FrechetE} to the system $(S,P)$, with $S:=\Lambda(P)$ we obtain that $E(S,P)=P$ is a tame semigroup.
\end{proof}

The following result was proved in \cite[Theorem 3.2]{GM1}
using the Bourgain-Fremlin-Talagrand  (BFT) dichotomy
in the setting of continuous group actions.
The same arguments work also for separately continuous semigroup actions.
For the sake of completeness we include a simplified proof.


\begin{f} [A dynamical BFT dichotomy] \label{D-BFT}
Let $X$ be a compact metric dynamical $S$-system and let $E=E(X)$ be its
enveloping semigroup. We have the following alternative. Either
\begin{enumerate}
\item
$E$ is a separable Rosenthal compact, hence ${card} \; {E} \leq
2^{\aleph_0}$; or
\item
the compact space $E$ contains a homeomorphic
copy of $\beta\N$, hence ${card} \; {E} = 2^{2^{\aleph_0}}$.
\end{enumerate}
The first possibility 
holds iff $X$ is a tame $S$-system.
\end{f}
\begin{proof}
 For every $f \in C(X)$ define $E^f:=\{f\circ p: p\in E\}$.
Then $E^f$ is a pointwise compact subset of $\R^X$, being a
continuous image of $E$ under the map $q_f: E \to E^f, \hskip
0.2cm p \mapsto f\circ p$.
Since $X$ is metrizable
by Lemma \ref{l:Env-all}.5 there exists a sequence $\{s_m\}_{m=1}^\infty$ in $S$ such that
$\{j(s_m)\}_{m=1}^\infty$ is dense in $E(X)$. In particular, the sequence of real valued functions
$\{f \circ s_m\}_{m=1}^\infty$ is pointwise dense in $E^f$. 

Choose a sequence $\{f_n \}_{n \in \N}$ in $C(X)$ which separates
the points of $X$. For every pair
$s, t$ of distinct elements of $E$ there exist a point $x_0 \in X$
and a function $f_{n_0}$ 
such that
$f_{n_0}(sx_0) \neq f_{n_0}(tx_0)$. It follows that the continuous
diagonal map
$$
\Phi: E \to \prod_{n\in \N} E^{f_n},\qquad
p\mapsto (f_1\circ p, f_2\circ p,
\dots )
$$
separates the points of $E$ and hence is a topological embedding.
Now if for each $n$ the space $E^{f_n}$ is a Rosenthal compactum
then so is $E\cong \Phi(E)\subset \prod_{n=1}^\infty E^{f_n}$,
because the class of Rosenthal compacta is closed under countable
products and closed subspaces. On the other hand
if at least one $E^{f_n}=cl_p(\{f_n \circ s_m\}_{m=1}^\infty)$
is not Rosenthal then, by a version of the BFT-dichotomy (Todor\u{c}evi\'{c} \cite[Section 13]{To-b})
it contains a homeomorphic copy of $\beta\N$
and it is easy to see that so does its preimage $E$. In fact if
$\beta\N \cong Z\subset E^{f_n}$ then any closed subset $Y$ of $E$
which projects onto $Z$ and is minimal with respect to these
properties is also homeomorphic to $\beta\N$.

Now we show the last assertion. If
$X$ is tame then every $p \in E(X)$ is a fragmented self-map of $X$.
Hence every $f \circ p \in E^f$ is fragmented.
By Remark \ref{r:fr1}.2 this is equivalent to saying that every $f \circ p$ is Baire 1. So $E^f \subset \mathcal{B}_1(X)$ is a Rosenthal compactum.
Therefore, $E \cong \Phi(E) \subset \prod_{n\in \N} E^{f_n}$ is also Rosenthal. 
Conversely, if $E$ is a Rosenthal compactum
then $(S,X)$ is tame by Theorem \ref{t:FrechetE}.
\end{proof}



\begin{thm} [BFT dichotomy for Banach spaces] \label{t:W=Ros-c}
Let $V$ be a separable Banach space and let $\E=\E(V)$ be its (separable) enveloping semigroup.
We have the following alternative. Either
\begin{enumerate}
\item
$\E$ is a Rosenthal compactum, hence ${card} \; {\E} \leq
2^{\aleph_0}$; or
\item
the compact space $\E$ contains a homeomorphic
copy of $\beta\N$, hence ${card} \; {\E} = 2^{2^{\aleph_0}}$.
\end{enumerate}
The first possibility holds iff $V$ is a Rosenthal Banach space.
\end{thm}
\begin{proof}
Recall that $\E=E(\Theta^{op},B^*)$.
By Theorem \ref{t:tame2}, $V$ is Rosenthal iff $(\Theta^{op},B^*)$ is tame.
Since $V$ is separable, $B^*$ is metrizable. So we can apply Fact \ref{D-BFT}.
\end{proof}

\subsection{Amenable affine compactifications}

Let $G$ be a topological group and $X$ a $G$-space.
Let us say that an affine $S$-compactification $\a: X \to Y$ is \emph{amenable} if
$Y$ has a $G$-fixed point.
We say that a closed unital linear subspace $\A \subset \mathrm{WRUC}(X)$ is
(left) \emph{amenable} if the corresponding affine $G$-compactification is amenable.
By Ryll-Nardzewski's classical theorem
$\mathrm{WAP}(G)$ is 
amenable.
Let $f \in \mathrm{RUC}(G)$ and let $\pi_f: X \to Q_f$ be the
corresponding cyclic affine $G$-compactification
(Section \ref{s:cyclAff}).
In our recent work \cite{GM-fp} we show that $\mathrm{Asp_c}(G)$ is 
amenable and that for every $f \in \mathrm{Asp_c}(G)$
there exists a $G$-fixed point (a \emph{$G$-average} of $f$) in $Q_f$.
The first result together with Proposition \ref{p:metrAlg} yield the following:

\begin{cor} \label{c:MetrSemCompIsAmen}
Let $G$ be a topological group and
$\A$ a (left) m-introverted closed subalgebra of $\mathrm{RUC}(G)$. If $\A$ is separable then $\A$ is amenable.
\end{cor}

A topological group $G$ is said to be amenable if $\mathrm{RUC}(G)$ is amenable. 
By a classical result of von Neumann, the free discrete group
$\mathbb{F}_2$ on two symbols is not amenable. So, $\mathrm{RUC}(\mathbb{F}_2)=l_{\infty}(\mathbb{F}_2)$ is not amenable.
By \cite{GM-fp},  $\mathrm{Tame}(\mathbb{F}_2)$ is not amenable.
It would be interesting to study for which non-amenable groups $G$ the algebra $\mathrm{Tame_c}(G)$ is amenable and
for which $f \in \mathrm{Tame_c}(G)$ there exists a $G$-fixed point of $Q_f$.

\begin{ex} \
\ben
\item
Results of \cite{GM-tame} show that
$\varphi_D(n)=sgn \cos (2 \pi n \a)$ is a tame function on $\Z$ which is not Asplund.
\item
As a simple illustration of Proposition \ref{p:metrAlg}
note that the two-point semigroup compactifications of $\Z$ and $\R$ are obviously metrizable.
So the characteristic function $\xi_{\N}: \Z \to \R$ and $arctg: \R \to \R$ are both Asplund. Grothendieck's double limit criterion
 shows that these functions are not WAP.
\een
\end{ex}

\section{Representations of semigroup actions on Banach spaces}
\label{s:repr}

As was shown in several of our earlier works some
properties of dynamical systems are clearly reflected in
analogous properties of their enveloping semigroups on the one hand,
and in their representations on Banach spaces on the other.
Our results from \cite{GM1, GMU, GM-rose} are formulated for group actions.
However the main results in these papers remain true for semigroup actions. 

For continuous group actions the results (1), (2) of the following theorem
were  proved respectively in
\cite{GM-rose} and \cite{GM1} (compare also with Theorem \ref{t:wap}).
We will show next how the proofs of (1), (2) can be modified to suit the more general case of semigroup
actions, obtaining, in fact, also some new results.

\begin{thm} \label{t:tame=R-repr}
Let $S$ be a semitopological semigroup and $X$ a compact $S$-system with a separately continuous action.
\ben
\item $(S,X)$ is a tame (continuous) system
if and only if $(S,X)$ is weakly (respectively, strongly) Rosenthal-approximable.
\item $(S,X)$ is a HNS (continuous) system
if and only if $(S,X)$ is weakly (respectively, strongly) Asplund-approximable.
\een
If $X$ is metrizable then in (1) and (2) ``approximable" can be replaced by ``representable".
Moreover, the corresponding Banach space can be assumed to be separable.
\end{thm}
\begin{proof}
The proof for continuous actions is the same as in \cite{GM-rose}.
So below we show only how the proof can be adopted
for separately continuous actions and weakly continuous representations.

\sk

\nt \emph{The ``only if" part:}
For (1) use the fact that $(\Theta^{op},B^*)$ is a tame system (Theorem \ref{t:tame2}) for every Rosenthal $V$ and for (2), the fact that
$(\Theta^{op}, B^*)$ is HNS (Theorem \ref{t:F-semig}) for Asplund $V$.

\sk

\nt\emph{ The ``if" part:} (1)
For every $f \in C(X)=\mathrm{Tame}(X)$ the orbit $fS$ is a Rosenthal family for $X$ (Proposition \ref{p:tame-f}).
Applying Theorem \ref{t:general} below we
conclude that every $f \in C(X)=\mathrm{Tame}(X)$ on a compact
$S$-space
$X$ comes from a Rosenthal representation. Since continuous functions
separate points of $X$, this implies that Rosenthal representations of $(S,X)$ separate points of $X$.
So, for (1) it is enough to prove the following result.

\begin{thm} \label{t:general} Let $X$ be a compact $S$-space and
let $F \subset C(X)$ be a Rosenthal family 
for $X$ such that
$F$ is $S$-invariant; that is, $fS \subset F \ \ \forall f \in F$. Then
there exist: a Rosenthal 
Banach space $V$, an injective mapping
$\nu: F \to B_V$ 
and a representation
$$
h: S \to \Theta(V), \ \ \ \a: X \to V^*
$$
of $(S,X)$ on $V$ such that $h$ is weakly continuous, $\a$ is
a weak$^*$ continuous map 
and
$$
f(x)= \langle \nu(f), \a(x)
\rangle \ \ \ \forall \ f \in F \ \ \forall \ x \in X.
$$
Thus the following diagram commutes

\begin{equation}  \label{diag1}
\xymatrix{ F \ar@<-2ex>[d]_{\nu} \times X
\ar@<2ex>[d]^{\a} \ar[r]  & \R \ar[d]^{id_{\R}} \\
V \times V^* \ar[r]  &  \R }
\end{equation}

If $X$ is metrizable then in addition we can suppose that $V$ is separable.

If the action $S \times X \to X$ is continuous we may assume that $h$ is strongly continuous.
\end{thm}
%
%

\sk
\begin{proof}

\nt \textbf{Step 1:} The construction of $V$.

\br
For brevity of notation let $\Acal := C(X)$ denote the Banach
space $C(X)$, $B$ will denote its unit ball,
and $B^*$ will denote
the weak$^*$ compact unit ball of the dual space $\Acal^*
=C(X)^*$.
Let $W$ be the symmetrized convex hull of $F$; that is, $W:={\co}(F \cup -F).$
Consider the sequence of sets
\begin{equation} \label{bLNO}
 M_n:=2^n W + 2^{-n} B.
\end{equation}
Then $W$ is convex and symmetric. We apply the construction of Davis-Figiel-Johnson-Pelczy\'nski \cite{DFJP}
as follows.
Let $\| \ \|_n$ be the Minkowski
functional of the set $M_n$, that is,
$$
\| v\|_n = \inf\ \{\lambda
> 0 \bigm| v\in \lambda M_n\}.
$$
Then $\| \ \|_n$ is a norm on $\Acal$ equivalent to the given norm
of $\Acal$. For $v\in \Acal,$ set
$$
N(v):=\left(\sum^\infty_{n=1} \| v \|^2_n\right)^{1/2} \hskip
0.1cm \text{and let} \hskip 0.1cm \hskip 0.1cm V: = \{ v \in \Acal
\bigm| N(v) < \infty \}.
$$
Denote by $j: V \hookrightarrow \Acal$ the inclusion map. Then
$(V,N)$ is a Banach space, $j: V \to \Acal$ is a continuous linear
injection and

\begin{equation} \label{F0}
W \subset j(B_V)=B_V \subset \bigcap_{n \in \N} M_n =
\bigcap_{n \in \N} (2^n W + 2^{-n}B)
\end{equation}


\br
\nt \textbf{Step 2:} The construction of the representation $(h,\a)$ of $(S,X)$ on $V$.

\br

The given action $S \times X \to X$ induces a natural linear norm preserving
continuous right action $C(X) \times S \to C(X)$ on the Banach space $\Acal=C(X)$.
It follows by the construction that $W$ and $B$ are $S$-invariant subsets in $\Acal$.
This implies that $V$ is an $S$-invariant subset of $\Acal$ and the restricted natural linear action
$V \times S \to V, \ \ (v,g) \mapsto vg$ satisfies $N(vs)\leq N(v)$.
Therefore, the co-homomorphism $h: S \to \Theta(V), \
h(s)(v):=vs$ is well defined.

Let $j^*: \Acal^* \to V^*$ be the adjoint map of $j: V \to \Acal$.
Define $\a: X \to V^*$ as follows. For every $x \in X \subset
C(X)^*$ set $\a(x)=j^*(x)$. Then $(h,\a)$ is a 
representation of $(S,X)$ on the Banach space $V$.


\sk

By the construction $F \subset W \subset B_V$.
Define $\nu: F \hookrightarrow B_V$ as the natural inclusion.
Then
\begin{equation} \label{F}
f(x)= \langle \nu(f), \a(x) \rangle \ \ \ \forall \ f \in F \ \ \forall \ x \in X.
\end{equation}


\br
\nt \textbf{Step 3:} Weak continuity of $h: S \to \Theta(V)$.
\br

By our construction $j^*: C(X)^* \to V^*$, being the adjoint of the bounded linear operator
$j: V \to C(X)$,
is a norm and weak$^*$ continuous linear operator. By Lemma \ref{l:generating}.2 
we obtain that $j^*(C(X)^*)$ is norm dense in $V^*$.
Since $V$ is Rosenthal, Haydon's theorem (Fact \ref{f:RosFr}.4)
gives $Q:=cl_{w^*}(co(Y))=cl_{norm}(co(Y))$, where $Y:=j^*(X)$. Now observe that
$j^*(P(X))=Q$.
Since $S \times X \to X$ is separately continuous, every orbit map $\tilde{x}: S \to X$ is continuous, and each orbit map
$\widetilde{j^*(x)}: S \to j^*(X)$ is weak$^*$ continuous. Then also $\widetilde{j^*(z)}: S \to V^*$ is weak$^*$ continuous
for each $z \in cl_{norm}(co(j^*(X)))=Q$. Since $sp(Q)$ is norm dense in $V^*$ (and $||h(s)|| \leq 1$
for each $s \in S$) it easily follows that $\tilde{j^*(z)}: S \to V^*$ is weak$^*$ continuous for every $z \in V^*$. This is equivalent
to the weak continuity of $h$.

\sk

If the action $S \times X \to X$ is continuous we may assume that $h$ is strongly continuous.
Indeed, by the definition of the norm $N$, we can show that the action of $S$ on $V$ is norm
continuous (use the fact that, for each $n \in \N$, the norm
$\norm{\cdot}_n$ on $\Acal$ is equivalent to the given norm on
$\Acal$).

\br

\nt \textbf{Step 4:} $V$ is a Rosenthal space.

\br

By results of \cite[Section 4]{GM-rose}, $W$ is a Rosenthal family for $B^*$ (and $X$). Furthermore,
a deeper analysis shows (we refer to \cite[Theorem 6.3]{GM-rose} for details)
that
$B_V$ is a Rosenthal family for $B_{V^*}$. Thus $V$ is Rosenthal by Fact \ref{f:RosFr}.

If the compact space $X$ is metrizable then $C(X)$ is
separable and it is also easy to see that $(V, N)$ is separable.

This proves Theorem \ref{t:general} and hence also Theorem \ref{t:tame=R-repr}.1.
\end{proof}

\sk

Now for the ``Asplund case", Theorem \ref{t:tame=R-repr}.2,
one can modify the proof of (1). The main idea is that the corresponding
results of \cite[Section 7]{Me-nz} and \cite[Section 9]{GM1}
can be adopted here, thus obtaining a modification of Theorem \ref{t:general}
which replaces a Rosenthal space by an Asplund space, and a ``Rosenthal family $F$" for $X$ by an ``Asplund set".
The latter means that for every countable subset $A \subset F$ the pseudometric $\rho_A$ on $X$ defined by
$$\rho_A(x,y):=\sup_{f \in A} |f(x)-f(y)|, \  x, y \in X$$ is separable.
 By \cite[Lemma 1.5.3]{Fa} this is equivalent to saying that $(C(X)^*,\rho_A)$ is separable.
Now $co(F \cup -F)$ is an Asplund set for $B^*$ by \cite[Lemma 1.4.3]{Fa}. The rest is similar to the proof of \cite[Theorem 7.7]{Me-nz}.
Checking the weak continuity of $h$ one can apply a similar idea
(using again Haydon's theorem as in (1)).

\sk

Finally note that if $X$ is metrizable then in (1) and (2) ``approximable" can be replaced by ``representable" using
an $l_2$-sum of a sequence of
separable Banach spaces (see Lemma \ref{l:sum}.3).  
\end{proof}

\sk

\begin{remark} \label{r:iso-DFJP}
The fundamental DFJP-factorization construction from \cite{DFJP} has an ``isometric modification".
According to \cite{LNO} one may assume in Theorem \ref{t:general}
that the bounded operator $j: V \to \Acal$ has the property $||j|| \leq 1$.
More precisely, we can replace in the Equation \ref{bLNO}
the sequence of sets $M_n:=2^n W + 2^{-n} B$ by $K_n:=a^{\frac{n}{2}} W + a^{-\frac{n}{2}} B$, where
$2 < a < 3$ is the unique solution of the equation $\sum_{n=1}^{\infty} \frac{a^n}{(a^n+1)^2} =1$.
For details see \cite{LNO}.
Taking into account this modification (which is completely compatible with our $S$-space setting)
for a set $F \subset C(X)$ with $\sup \{|f(x)| : \ x \in X, f \in F\} \leq 1$
we can assume that $\nu(F) \subset B$ and $\a(X) \subset B^*$. Hence the following sharper diagram commutes

\begin{equation} \label{diag2}
\xymatrix{ F \ar@<-2ex>[d]_{\nu} \times X
\ar@<2ex>[d]^{\a} \ar[r]  & [-1,1] \ar[d]^{id} \\
B \times B^* \ar[r]  &  [-1,1] }
\end{equation}
Note also that this modified version from \cite{LNO} of the DFJP-construction repairs in particular
the proof of \cite[Theorem 4.5]{Me-nz}. The latter was first corrected in the arxiv version of \cite[Theorem 4.5]{Me-nz}
using, however, diagrams like \ref{diag1}, where $\nu(F)$ and $\a(X)$ are bounded.
\end{remark}


\begin{thm} \label{t:mat} \
\ben
\item
Let $X$ be a compact $S$-space. The following conditions are equivalent:
\begin{enumerate}
\item $f \in \mathrm{Tame}(X)$ (respectively, $f \in \mathrm{Tame}_c(X)$).
\item There exist: a weakly (respectively, strongly) continuous representation $(h,\a)$
of $(S,X)$ on a Rosenthal Banach space $V$ and a vector $v \in V$ such that
$f(x)= \langle v, \a(x)
\rangle \ \ \  \forall \ x \in X.$
\end{enumerate}
\item
Let $S$ be a semitopological semigroup and $f \in C(S)$. The following conditions are equivalent:
\begin{itemize}
\item [(a)] $f \in \mathrm{Tame}(S)$ (respectively, $f \in \mathrm{Tame}_c(S)$).
\item [(b)] $f$ is a matrix coefficient of a weakly (respectively, strongly) continuous
co-representation of $S$ on a Rosenthal space.
That is, there exist: a Rosenthal space $V$, a weakly (respectively, strongly) continuous
co-homomorphism $h: S \to \Theta(V)$, and vectors $v \in V$ and $\psi
\in V^*$ such that $f(s)=\psi(vs)$ for every $s \in S$.
\end{itemize}
\item
Similar
(to (1) and (2))
results are valid for
\begin{itemize}
\item [(a)] Asplund functions and Asplund Banach spaces;

\item [(b)] WAP functions and reflexive Banach spaces.
\end{itemize}
\een
\end{thm}
\begin{proof}
(1) \ (b) $\Rightarrow$ (a):
$(\Theta(V)^{op},B^*)$ is a tame system for every Rosenthal space $V$ by Theorem \ref{t:tame2}.
The action is separately (jointly) continuous for the weak (respectively, strong) operator topology on $\Theta(V)^{op}$.

(a) $\Rightarrow$ (b): Let $f \in \mathrm{Tame}(X)$. This means by
Proposition \ref{p:tame-f} that the orbit $fS$ is a Rosenthal
family for $X$. Now we can apply Theorem \ref{t:general} to the
family $F:=fS$ (getting 
$X_f$ as $\a(X)$).

(2)  \  (a) $\Rightarrow$ (b):
$f \in \mathrm{Tame}(S)$ (respectively, $f \in \mathrm{Tame}_c(S)$) means that there exist:
a tame $S$-compactification $\g: S \to X$ of the $S$-space $S$ such that $S \times X \to X$ is separately
continuous (respectively, jointly continuous)
and a continuous function $f_0: X \to \R$ such that $f=f_0 \circ \nu$.
Apply Theorem \ref{t:general} to $f_0$ getting the
desired $V$ and vectors $v:=\nu(f)$ and $\psi:=\a(\g(e))$. Now
$$f(s)= \langle v, \a(\g(s)) \rangle = m(v,\psi)(s) \ \ \  \forall \ s \in S.$$

(b) $\Rightarrow$ (a): Since $h: S \to \Theta(V)$ is weakly (strongly) continuous the natural action
of $S$ on the compact space $X:=cl_{w^*}(S\psi)$ is separately (respectively, jointly) continuous.
Apply Theorem \ref{t:tame2} to establish that $(S,X)$ is tame. Finally observe that $f(s)=\lan v,s \psi \ran$
comes from the $S$-compactification $S \to X, s \mapsto s \psi$.

(3) (a) is similar to (1) using the Asplund version of Theorem \ref{t:general}.
For (b) note that the case of $f \in \mathrm{WAP}(S)$ was proved in \cite[Theorem 5.1]{Me-nz}.
The case of $f \in WAP_c(S)$ is similar using \cite[Theorem 4.6]{Me-nz}.
\end{proof}

If in Theorem \ref{t:mat}, $S:=G$ is a semitopological group then for any monoid co-homomorphism
$h: G \to \Theta(V)$ we have $h(G) \subset \Iso(V)$.
Recall also that $\mathrm{WAP}(G)=\mathrm{WAP_c}(G)$ (Lemma \ref{l:classes}.4).

\begin{prop} \label{p:tame-is-wruc}
Let $S \times X \to X$ be a separately continuous action.
Then:
\ben
\item $\mathrm{Tame}(X) \subset \mathrm{WRUC}(X)$. In particular, $\mathrm{Tame}(S) \subset \mathrm{WRUC}(S)$.
\item If $X$ is a compact tame (e.g., $\mathrm{HNS}$ or $\mathrm{WAP}$) system then $(S,X)$ is $\mathrm{WRUC}$.
\een
\end{prop}
\begin{proof}
(1) Let $f \in \mathrm{Tame}(X)$. Then there exist: a compact tame $S$-system $Y$, an
$S$-compactification $\nu: X \to Y$ and $\tilde{f} \in C(Y)$ such that $f=\tilde{f} \circ \nu$.
By Theorem \ref{t:general},
$\tilde{f}$ comes from a weakly continuous representation $(h,\a)$ of $(S,Y)$ on a Rosenthal space $V$. That is,
$\tilde{f}(y)= \langle \nu(\tilde{f}), \a(y)
\rangle \  \forall \ y \in Y.
$
Consider the restriction operator (Remark \ref{r:actions}.2),
$r: V \to C(X), \ r(v)(x)=\lan v, \a(x) \ran$. Then for the vector $r(\nu(\tilde{f}))=f$
the orbit map $S \to C(X), s \mapsto f s$ is weakly continuous.

(2) Since $X$ is tame we have $\mathrm{Tame}(X)=C(X)$. On the other hand, by (1) we have $\mathrm{Tame}(X) \subset \mathrm{WRUC}(X) \subset C(X)$.
Hence, $\mathrm{WRUC}(X)=C(X)$.
\end{proof}

\begin{remark} \label{drop-w-adm}
Proposition \ref{p:tame-is-wruc} 
allows us to strengthen some results of \cite{Me-nz}.
Namely, in 7.7, 7.11 and 7.12 of \cite{Me-nz} one may drop the assumption of WRUC-compatibility of $(S,X)$.
Theorem \ref{t:mat} unifies and strengthens some earlier results from \cite{Me-nz,GM-rose}.
\end{remark}

\subsection{Representations of topological groups}

\begin{thm} \label{t:GrRep}
Let $G$ be a topological group such that
$Tame_c(G)$ (respectively, $Asp_c(G)$, $\mathrm{WAP}(G)$) separates points and closed subsets.
Then there exists a Rosenthal (respectively, Asplund, reflexive) Banach space
$V$ and a topological group embedding $h: G \hookrightarrow \Iso(V)$ with respect to the strong topology.
\end{thm}
\begin{proof} We consider only the case of $\mathrm{Tame}(G)$. Other cases are similar.
The case of $\mathrm{WAP}(G)$ is known \cite{Me-nz, Me-hilb}.

For every topological group $G$ the involution $inv: g \mapsto g^{-1}$ defines a
topological isomorphism between $G$ and its opposite group $G^{op}$.
So it is equivalent to show that there exists a topological group embedding $h: G \to \Iso(V)^{op}$.
Let $\{f_i\}_{i \in I}$ be a collection of tame functions
which come from jointly continuous tame $G$-compactifications of $G$ and
separates points and closed subsets.
By Theorem \ref{t:mat}.2
for every $i \in I$ there exist: a Rosenthal space $V_i$, a strongly continuous
co-homomorphism $h_i: G \to \Iso(V_i)$, and vectors $v_i \in V_i$ and $\psi_i \in V_i^*$ such that $f_i(g)=\psi_i(v_ig)$ for every $g \in G$.
Consider the $l_2$-type sum $V:=(\Sigma_{i \in I} V_i)_{l_2}$ which is Rosenthal by virtue of Lemma \ref{l:sum}.2.
We have the natural homomorphism
$h: G \to \Iso(V)^{op}$ defined by $h(v)=(h_i(v_i))_{i \in I}$ for every $v= (v_i)_{i \in I} \in V$.
It is easy to show that $h$ is a strongly continuous homomorphism. Since $\{f_i\}_{i \in I}$ separates points and closed subsets, the family
of matrix coefficients $\{m(v_i, \psi_i)\}_{i \in I}$ generates the topology of $G$. It follows that $h: G \to \Iso(V)^{op}$ is a
topological embedding.
\end{proof}

Recall (see Remark \ref{r:ex}) that for the group $G:=H_+[0,1]$ every
Asplund  (hence also every WAP) function is constant
and every continuous representation $G \to \Iso(V)$ on an Asplund
(hence also reflexive) space $V$ must be trivial.
In contrast one may show that $G$ is Rosenthal representable


\begin{thm} \label{H_+}
The group $G:=H_+[0,1]$ is Rosenthal representable.
\end{thm}
\begin{proof}
Consider the natural action of $G$ on the closed interval
$X:=[0,1]$ and the corresponding enveloping semigroup $E=E(G,X)$.
Every element of $G$ is a (strictly) increasing self-homeomorphism of $[0,1]$.
Hence every element $p \in E$ is a nondecreasing function.
It follows that $E$ is
naturally homeomorphic to a subspace of the Helly compact space (of all nondecreasing selfmaps of $[0,1]$ in the pointwise topology).
Hence $E$ is a Rosenthal compactum. So by the dynamical BFT dichotomy,
Fact \ref{D-BFT}, the $G$-system $X$ is tame.
By Theorem \ref{t:tame=R-repr} we have a faithful representation
$(h,\a)$ of $(G,X)$ on a separable Rosenthal space $V$. Therefore we obtain a $G$-embedding $\a: X \hookrightarrow (V^*, w^*)$.
Then the strongly continuous homomorphism $h: G \to \Iso(V)^{op}$ is injective.
Since $h(G) \times \a(X) \to \a(X)$ is continuous
(and we may identify $X$ with $\a(X)$) it follows, by the minimality properties of the compact open topology, that $h$ is an embedding.
Thus $h \circ \inv: G \to Iso(V)$ is the required topological group embedding.
\end{proof}

\begin{remark} \label{r:moreH_+} \
\ben
\item
Recall that by \cite{me-fr} continuous group representations on Asplund spaces have the \emph{adjoint continuity property}.
In contrast this is not true for Rosenthal spaces.
Indeed, assuming the contrary we would have, from Theorem \ref{H_+}, that
the dual action of the group $H_+[0,1]$ on $V^*$ is continuous, but this is impossible
by the following fact
\cite[Theorem 10.3]{GM-suc} (proved also by Uspenskij (private communication)):
every adjoint continuous (co)representation of $H_+[0,1]$ on a Banach space is trivial.
\item
There exists a semigroup compactification $\nu: G=H_+[0,1] \to P$
into a tame semigroup $P$ such that $\nu$ is an embedding. Indeed, the associated enveloping semigroup compactification
$j: G \to E$ of the tame system $(G,[0,1])$ is tame.
 Observe that $j$ is a topological embedding because the compact open topology on $j(G) \subset \Homeo([0,1])$ coincides
with the pointwise topology.
\een
\end{remark}


\begin{question}
Is it true that every Polish topological group $G$ is Rosenthal representable ?
Equivalently, is this true for the universal Polish groups $G=\Homeo([0,1]^{\N})$ or
$G=\Iso(\U)$ (the isometry group of the Urysohn space $\U$) ?
By Theorem \ref{t:GrRep} a strongly related question is the question
 whether the algebra $\mathrm{Tame}(G)$ separates points and closed subsets.
%
%
\end{question}

\section{Banach representations of right topological semigroups and affine systems}
\label{s:appl}

%
%

\subsection{Tame representations}
\label{s:TameRepr}

\begin{thm} \label{t:Ros-E} \
\ben
\item
Every weakly continuous representation $(h,\a)$ of an
$S$-space $X$ on a Rosenthal Banach space is E-compatible.

\item
If the representation in (1) is $w^*$-generating
then the representation is strongly E-compatible.
 \een
\end{thm}

\begin{proof} (1)
Applying Haydon's theorem (Fact \ref{f:RosFr}.4) 
we get by Lemma \ref{norm-cl}.1 that the representation $(h,\a)$ is E-compatible.

(2) Use (1) 
taking into account Lemma \ref{norm-cl}.2.
\end{proof}

\begin{thm} \label{t:tame is inj}
\emph{(\cite{Ko} and \cite{Gl-tame} for metrizable systems)}
Every tame compact 
$S$-space $X$ is injective.
Hence, every affine $S$-compactification of a tame system is E-compatible.
\end{thm}
\begin{proof} In view of Definition \ref{d:inj}
we have to show that $m(f,\phi) \in \A(E(X),e)$ for  every $f \in C(X), \phi \in C(X)^*$.
By Theorem \ref{t:general},
$f$ comes from a Rosenthal representation.
There exist: a weakly continuous representation $(h,\a)$
of $(S,X)$ on a Rosenthal Banach space $V$ and a vector $v_0 \in V$
such that
$$f(x)= \langle v_0, \a(x) \rangle \ \ \  \forall \ x \in X.$$

Consider the restriction linear $S$-operator (Remark \ref{r:actions}.2)
$$r: V \to C(X), \ r(v)(x)=\langle v, \a(x) \rangle.$$
Let $r^*: C(X)^* \to V^*$ be the adjoint operator.
Since $m(f,\phi)= m(r(v_0),\phi))=m(v_0, r^*(\phi)),$ it is enough to show that
$m(v_0, r^*(\phi)) \in \A(E(X),e)$.

Analyzing the proof of Theorem \ref{t:general}
we may assume in addition, in view of Lemma \ref{l:generating}.2, that the representation $(h,\a)$ is generating.
By Theorem \ref{t:Ros-E}
the representation $(h,\a)$ is strongly E-compatible.
So by Lemma \ref{l:E-compatible}
we have $ m(v_0,r^*(\phi)) \in \A(E(\a(X)),e)$.
Since $\a: X \to \a(X)$
is a surjective $S$-map we have the natural surjective
homomorphism $E(X) \to E(\a(X))$. 
Hence, $\A(E(\a(X)),e) \subset \A(E(X),e)$.
Thus, $m(v_0, r^*(\phi)) \in \A(E(X),e)$, as required.
\end{proof}

\begin{thm} \label{t:tame-sc-case}
Let $\nu: S \to P$ be a right topological semigroup compactification.
 If $P$ is a tame semigroup (e.g., HNS-semigroup, semitopological, or metrizable)
then the $S$-system $P$ is injective and the algebra of the
compactification $\nu$ is introverted (in particular, $\nu$ is an operator compactification).
\end{thm}
\begin{proof} By Lemma \ref{l:both}.3, $P$ is a tame $S$-system.
Theorem \ref{t:tame is inj} guarantees that it is injective. Hence,
by Theorem \ref{t:introv<->inj} the algebra of the
compactification $\nu$ is introverted (Proposition \ref{t:conditions-new} implies that $\nu$ is an operator compactification).
\end{proof}

\begin{thm} \label{m-intro to intro}
Let $V \subset C(S)$ be an m-introverted 
Banach subalgebra. 
If $V \subset \mathrm{Tame}(S)$ (e.g., if $V$ is separable) then $V$ is
introverted. In particular, $\mathrm{Tame}(S)$, $\mathrm{Asp}(S)$,
$\mathrm{WAP}(S)$ are introverted
(and the same is true for $\mathrm{Tame_c}(S)$, $\mathrm{Asp_c}(S)$, $\mathrm{WAP_c}(S)$).
\end{thm}
\begin{proof}
Consider the corresponding semigroup compactification $\nu: S \to P$.
Since $V \subset \mathrm{Tame}(S)$ the system $(S,P)$ is tame. Then its enveloping semigroup $E(S,P)$ is a tame semigroup
(Lemma \ref{l:both}.2) and $E(S,P)$ can be naturally identified with $P$ (Lemma \ref{l:Env-all}.4). Now
combine Theorems \ref{t:tame-sc-case} and \ref{t:introv<->inj}.
By Remark \ref{r:univ} the subalgebras above are m-introverted
(if $V$ is separable then
$V \subset \mathrm{Asp}(S)$ by Proposition \ref{p:metrAlg}, hence, $V \subset \mathrm{Tame}(S)$).
\end{proof}

\subsection{Banach representation of enveloping semigroups}

By Theorem \ref{t:tame2} the semigroup $\E(V)$ is tame for every Rosenthal space $V$.  
We now show that, in the converse direction, every tame (respectively, HNS)
semigroup $P$, or equivalently, every enveloping semigroup of a tame (respectively, HNS) system, admits a faithful
representation 
on a Rosenthal (respectively, Asplund) Banach space $V$.
Fact \ref{f:sem-in-ref} (for semitopological semigroups and reflexive spaces) is a particular case of the following result.

\begin{thm} \label{t:Srepr} \emph{(Enveloping semigroup representation theorem)}
\begin{enumerate}
\item
Let $P$ be a tame 
semigroup.
Then there exist a Rosenthal Banach space $V$ and a
 $\Lambda(P)$-admissible
 embedding of $P$ into $\E(V)$.
\item
If $P$ is a HNS-semigroup  
then there is a
 $\Lambda(P)$-admissible
embedding of $P$ into $\E(V)$ where $V$ is an Asplund Banach space.
\item
If $P$ is a semitopological semigroup 
then there is
an
embedding of $P$ into $\Theta(V)=\E(V^*)$
where $V$ is a reflexive Banach space.
\end{enumerate}
\end{thm}
\begin{proof}
Let  $S=\Lambda(P)$ be the topological center of $P$.
Since $P$ is admissible, $S$ is a dense submonoid of $P$.
Denote by $j: S \to P$ the corresponding inclusion.
Now $P$, as an $S$-system, is tame (Lemma \ref{l:both}.3).
By Theorem \ref{t:tame=R-repr}
there exists a family of flow representations $\{(h_i,\a_i)\}_{i \in I}$
$$ \ h_i: S \to \Theta(V_i)^{op}, \ \ \a_i: P \to B_{V^*_i}$$
of $(S,P)$ on Rosenthal Banach spaces $V_i$, where each $h_i$ is a weakly continuous homomorphism and
$\{\a_i\}_{i \in I}$ separates points of $P$.
As in the proof of Theorem \ref{t:tame is inj}
we may assume (by Lemma \ref{l:generating}.2) that these representations are generating.
Then, by Theorem \ref{t:Ros-E}, they are strongly E-compatible.

Consider the $l_2$-type sum $V:=(\Sigma_{i \in I} V_i)_{l_2}$. Then we have the natural $l_2$-sum of representations
$h: S \to \Theta(V)^{op}$ defined by $h(v)=(h_i(v_i))_{i \in I}$ for every $v= (v_i)_{i \in I} \in V$.
Since  $V^*=(\Sigma_{i \in I} V_i)_{l_2}^*=(\Sigma_{i \in I} V_i^*)_{l_2}$ (Lemma \ref{l:sum}.1)
and each $h_i$ is weakly continuous
it is easy to show that $h$ is a weakly continuous homomorphism.
We have the corresponding standard operator compactification $j_K: S \to K=\overline{h(S)} \subset \E(V)$.
Since $h(S) \subset \Theta^{op}(V) = \Lambda(\E(V))$, the embedding $K \subset \E(V)$ is $S$-admissible (Definition \ref{d:reprCRTS}).
 By Lemma \ref{l:sum}.2 we know that $V$ is Rosenthal. So in order to complete the proof for ``Rosenthal case"
(other cases are similar) we have to check the following claim.

\begin{claim} \emph{The semigroup compactifications $j: S \to P$ and $j_K: S \to K$ are equivalent. }
\end{claim}

\nt  \emph{Proof of the claim:}
Let $\A_j$ and $\A_K$ be the corresponding subalgebras of $C(S)$. We will show that each of them equals to
$$\A:=\lan \bigcup_{i \in I} m(V_i,V_i^*) \ran.$$
Each $Y_i:=\a_i(P)$ is an $S$-factor of $P$. Consider its enveloping semigroup $E(S,Y_i)$
and the compactification $j_i: S \to E(S,Y_i)$.
Since the family of $S$-maps $\{\a_i: P \to Y_i\}_{i \in I}$ separates points of $P$ the induced
system of homomorphisms $r_{\a_i}: E(S,P) \to E(S,Y_i)$ separates points of $P=E(S,P)$.
So, $\lan \cup_{i \in I} \Acal(E(Y_i),e) \ran = \A_j$.
The representations $(h_i,\a_i)$ are strongly E-compatible.
By Lemma \ref{l:E-compatible} we get $m(V_i,V_i^*) \subset \Acal(E(Y_i),e)$.
By Lemma \ref{l:algebras}, $\A(E(Y_i),e)=\lan m(V_i,Y_i) \ran$.
So, $\lan m(V_i,V_i^*) \ran = \Acal(E(Y_i),e) \ \  \forall i \in I.$
This implies that $\lan \cup_{i \in I} m(V_i,V_i^*) \ran =\lan \cup_{i \in I} \Acal(E(Y_i),e) \ran$.
Therefore, $\A=\A_j$.

Now we show that $\A_K=\A$. First observe that the set $L:=\cup_{i \in I} V_i$ separates points of $V^*=(\Sigma_{i \in I} V_i^*)_{l_2}$ (and hence of $B_{V^*}$).
By Lemma \ref{l:W=E} .1 the standard operator compactification $j_K: S \to K$ is equivalent to
the Ellis compactification $S \to E=E(S,B_{V^*})$.
Apply Lemma \ref{l:Env-all}.1 to the $S$-system $X=B_{V^*}$ and $L$. Then $\A_K= \lan m(L,B_{V^*}) \ran = \lan m(L,V^*) \ran$.
For every $v \in V_i \subset L, \ \phi \in V^*, s \in S$ we have $\phi(h(s)(v))=\phi_i(h_i(s)(v))$. So, $m(v,\phi) = m(v, \phi_i)$.
Therefore, $m(\cup_{i \in I} V_i , V^*)=\cup_{i \in I} m(V_i, V^*)=\cup_{i \in I} m(V_i, V_i^*)$.
It follows that $\A_K= \lan m(L,V^*) \ran = \lan \cup_{i \in I} m(V_i, V_i^*) \ran = \A$, as desired. So the claim is proved.

\sk

If $P$ is a HNS-semigroup (or a semitopological semigroup)
then one may modify our proof accordingly to ensure that
$V$ is an Asplund (or a reflexive) Banach space
using Theorem \ref{t:tame=R-repr}.2 (respectively, \ref{t:wap}) and Lemma \ref{l:sum}.2.
\end{proof}

\begin{thm} \label{gen-Ellis}
\emph{(A generalized Ellis' theorem)}
Every tame compact right topological \underline{group} $P$ is a topological group.
\end{thm}
\begin{proof}
By Theorem \ref{t:Srepr} there exists a
 $\Lambda(P)$-admissible
embedding of $P$ into $\E(V)$ for some Rosenthal Banach space $V$.
Since $P$ is a group it is easy to see that its topological
center $G:=\Lambda(P)$ is a subgroup of $P$.
Now apply Theorem \ref{p:not-E-gr} to the compactification $\nu: G \hookrightarrow P$ (defined by the natural inclusion)
and conclude that $P$ is a topological group.
\end{proof}

Since every compact semitopological semigroup is tame,
Ellis' classical theorem (Fact \ref{t:Ellis}) now follows as a special case of
Theorem \ref{gen-Ellis}. (Note that we are not using Ellis' theorem
as an intermediate step in the proof of Theorem \ref{gen-Ellis}.)


Combining Corollary \ref{c:FrechetSisTame} and Theorem \ref{gen-Ellis} we also have:

\begin{cor}  \label{c:small}
Let $P$ be a compact admissible right topological group.
Assume that $P$, as a topological space, is Fr\'{e}chet.
Then $P$ is a topological group.
\end{cor}

In particular this holds in each of the following cases:
\ben
\item \emph{(Moors \& Namioka \cite{MN})} $P$ is first countable.
\item
\emph{(Namioka \cite{Na72}, Ruppert \cite{Rup-73})}
$P$ is metrizable.
\een


\begin{cor} \label{Gl-Gen}
\emph{(Glasner \cite{Gl-tame} for metrizable $X$)}
A distal minimal
(not necessarily, metric) compact $G$-system is tame
if and only if it is equicontinuous.
\end{cor}
\begin{proof}
We give the proof for the (non-trivial) ``only if" part.
When $X$ is distal, $E$ is a group by a well known theorem of Ellis.
Also $E:=E(X)$ is a tame semigroup by Lemma \ref{l:both}.2.
By Theorem \ref{gen-Ellis} we get that $E$ is a topological group.
Finally, $X$ is equicontinuous
because $X$ can be identified with the compact coset
$E$-space $E/H$, where $H=St(x_0)$ is the stabilizer of some point $x_0 \in X$. 
\end{proof}

\begin{cor} \label{c:notTameF}
$\mathrm{D}(G) \cap \mathrm{Tame}(G) = \mathrm{AP}(G)$ for every
topological group $G$. 
\end{cor}
\begin{proof}
Let $f \in \mathrm{D}(G) \cap \mathrm{Tame}(G)$. Then the cyclic $G$-space $X_f$
has the following properties: a) distal, b) minimal, c) tame.
Indeed, for every distal function on a topological group the cyclic system $(G,X_f)$ is minimal (see \cite[p.196]{BJM}).
Now Corollary \ref{Gl-Gen} concludes that $X_f$ is equicontinuous.
Hence, $f \in AP(G)$. This proves $\mathrm{D}(G) \cap \mathrm{Tame}(G) \subset \mathrm{AP}(G)$. The reverse inclusion is trivial.
\end{proof}



\begin{remark} \label{notTameF} (Non-tame functions)
\ben
\item
Corollary \ref{c:notTameF} implies that
$(\mathrm{D}(G) \setminus \mathrm{AP}(G)) \subset (\mathrm{RUC}(G) \setminus \mathrm{Tame}(G))$.
Hence any distal function on $G$ which not almost periodic is not tame.
As a concrete example for $G =\Z$, take $f(n) = \cos(2\pi n^2\alpha)$
with $\al$ any irrational real number.
\item
Any function $f \in l_{\infty}(\Z)$ such that the system $X_f$ either has
positive entropy, or
is minimal and weakly mixing,  is non-tame.
\een
\end{remark}

\subsection{Haydon's functions}

Recall (Section \ref{s:cyclAff}) that for every $f \in \mathrm{WRUC}(X)$ on an $S$-system $X$
 we have the cyclic affine $S$-compactification $\pi_f: X \to Q_f$, where
$Q_f$ is the pointwise closure of ${co} (X_f)$ in $C(S)$ 
and $X_f:=cl_p(\{m(f,\delta_f(x)\}_{x \in X})$ is
the cyclic $S$-system generated by $f$.

\begin{defin} \label{d:new}
We say that $f \in \mathrm{WRUC}(X)$ has the \emph{Haydon's Property} (or is a {\em Haydon function})
if the pointwise and norm closures of $co (X_f) $ in $C(S)$ (equivalently, in $l_{\infty}(S)$)
coincide. That is, if $$\overline{co}^{norm} (X_f)=\overline{co}^{p} (X_f).$$
\end{defin}

\begin{prop} \label{p:tame-func-n}
Every tame function $f \in \mathrm{Tame}(X)$ has Haydon's property.
\end{prop}
\begin{proof}
By Theorem \ref{t:mat}
there exist a weakly continuous representation $(h,\a)$
of $(S,X)$ on a Rosenthal Banach space $V$ and a vector $v \in V$ such that
$f(x)= \langle v, \a(x)
\rangle \ \ \  \forall \ x \in X.$

Consider the linear bounded $S$-operator (between left $S$-actions)
$$T: V^* \to C(S), \ \mu \mapsto m(v,\mu).$$
By Lemma \ref{l:prop}, $X_f:=cl_p(m(f,\delta_f(X)))$. By the choice of $v \in V$ we have $m(f,\delta_f(X))=m(v,\a(X))$.
So, $T(\a(X))=\delta_f(X)$.
Then $T(Y)=X_f$, where $Y:=\overline{\a(X)}^{w^*}$.
Since $T$ is weak$^*$-pointwise continuous, the compactness argument imply that
$T(\overline{co}^{w^*} \ (Y))= \overline{co}^p \ (X_f).$
By Haydon's theorem (Fact \ref{f:RosFr}.4),
we have $\overline{co}^{norm} (Y)=\overline{co}^{w^*} (Y).$
By the linearity and norm continuity of $T$ we get
$T (\overline{co}^{norm} (Y)) \subset \overline{co}^{norm} (T(Y))$.
Clearly,
$\overline{co}^{norm} (T(Y)) \subset \overline{co}^{p} (T(Y))$.
Summing up (and taking into account that $T(Y)=X_f$) we obtain
$\overline{co}^{norm} (X_f) = \overline{co}^{p} (X_f).$
\end{proof}


\begin{ex}
Let $\om \in \Om=\{0,1\}^\Z$ be a transitive point under the shift
$\sig : \Om \to \Om$. We consider $\om$ as an element of $l_\infty(\Z)$.
Then by assumption the cyclic flow $X_\om = \Om$, and it can be
easily checked that $\overline{co}^{norm} (X_\om)=\overline{co}^{p} (X_\om)$.
Thus $\om$ is a Haydon function which is clearly not tame.
Thus the converse of Proposition \ref{p:tame-func-n} is not true. However we do have the following proposition.
\end{ex}

\begin{prop}\label{Hay}
For a Haydon function $f: X \to \R$, the cyclic affine compactification
$$
\a: X \to Q_f=\overline{co}^{norm} (X_f)=\overline{co}^{p} (X_f)
$$
is E-compatible.
\end{prop}
\begin{proof} By Lemma \ref{l:f-introtype}, $Q_f$ is a subset of $C(S)$.
Therefore, the evaluation map $w: S \times Q_f \to \R$, where $w(s,\phi)=\phi(s)$,
is separately continuous. Since $f \in \mathrm{WRUC}(X)$, 
$\a: X \to Q_f$ is an affine $S$-compactification. In particular, the action $S \times Q_f \to Q_f$ is separately continuous. So,
the function
$$m_w(t,y): S \to \R, \ s \mapsto \tilde{t}(sy)=y(ts)$$ is continuous for every $y \in Q_f$ and $t \in S$.
Clearly, $S$ separates points of $Q_f$.
By Lemma \ref{l:Env-all}.1,
$\lan m_w(S,Q_f) \ran$ and $\lan m_w(S,X_f) \ran$ are the algebras of the Ellis compactifications
$j_Q: S \to E(Q_f)$ and $j_{X_f}: S \to E(X_f)$, respectively.
 Since all $s$-translations on $Q_f$ are affine maps we have $m_w(t, \sum_{i=1}^n c_i q_i)= \sum_{i=1}^n c_i m_w(t, q_i)$
 for every $\sum_{i=1}^n c_i=1$, $c_i > 0$. Also, $|m_w(t,y)(s) - m_w(t,y_0)(s)| \leq ||y-y_0||_{\infty}$.
 Since $Q_f=\overline{co}^{norm} (X_f)$, it follows 
 that $$m_w(S,Q_f) \subset \overline{sp}^{norm} (m_w(S,X_f)) \subset \lan m_w(S,X_f) \ran.$$
 Hence, the Ellis compactifications $j_{Q_f}$ and $j_{X_f}$ are equivalent.
 \end{proof}

\begin{ex}
The distal function $f(n) = \cos(2\pi n^2 \alpha)$ in $l_\infty(\Z)$
is not a Haydon function. This follows from Proposition \ref{Hay} and Proposition \ref{p:not-E-fl}.
\end{ex}

\subsection{Banach representations of affine $S$-systems}
\label{s:aff-rep}


As we have already mentioned in Remark \ref{r:aff-list}.4, all the affine
$S$-compactifications $\a: X \to Q$ of $X$ come, up to equivalence,
from representations of dynamical $S$-systems $X$ on \emph{Banach spaces}.
In particular, it follows that $Q$ is \emph{affinely} $S$-isomorphic to an affine
$S$-subsystem of the weak$^*$-compact unit ball $B^*$ of $V^*$ for some Banach space $V$.
This suggests the following question.

\begin{question} \label{q:aff-rep}
Which metric affine $S$-compactifications $X \to Q$
can be obtained via representations of $(S,X)$ on good Banach spaces $V$,
(say, Rosenthal, Asplund or reflexive)
where $Q$ is a weak$^*$-compact \emph{affine} $S$-subset of $V^*$ (as in Section \ref{s:motivation}).
\end{question}

First note that there is no obstruction in the purely topological case (i.e. for trivial actions).
Indeed, by Keller's theorem \cite[p. 98]{BP}
any metric compact convex affine set $Q$ in a locally convex linear space is affinely homeomorphic to a compact convex subset $K$
in the Hilbert space $l_2$.

\begin{thm} \label{t:aff-rep}
\emph{(A representation theorem for $S$-affine compactifications)}
Let $X$ be a tame (HNS, WAP) compact metric $S$-system. Then every $S$-affine compactification $\g: X \to Q$ comes
from a weakly continuous representation of $(S,X)$ on a separable
Rosenthal (respectively: Asplund, reflexive) Banach space $V$, where
$Q \subset V^*$ is a weak$^*$ compact affine subset.
If $S \times X \to X$ is continuous we can assume that $h$ is strongly continuous.
If $S=G$ is a group then $h(G) \subset \Iso(V)^{op} \subset \Theta(V)^{op}$.
\end{thm}
\begin{proof}
Let $(\g,Q)$ be an $S$-affine compactification of a tame system $X$.
As usual let $\mathrm{A(Q)}|_X \subset C(X)$ be the corresponding
affine compactification space. $\mathrm{A(Q)}|_X$ is a closed linear unital subspace of $C(X)$.
Moreover, it is separable because $X$ is compact metrizable.
Choose a countable subset $\{f_n\}_{n \in \N} \subset \mathrm{A(Q)}|_X$ such that
$||f_n|| \leq \frac{1}{2^{n-1}}$ and $sp(\{f_n\}_{n \in \N})$ is norm dense in $\mathrm{A(Q)}|_X$.
We can suppose that $f_1=\textbf{1}$.

Since $(S,X)$ is tame, every $f_n \in \mathrm{Tame}(X)=C(X)$.
So $f_nS$ is a Rosenthal family for $X$ (Proposition \ref{p:tame-f}) for any $n \in \N$.
Hence, $f_nS$ is an eventually fragmented family of maps $X \to \R$ by Fact \ref{f:sub-fr}.
Then $F:=\cup_{n \in \N} (f_nS)$ is again an eventually fragmented family, as can be shown
by diagonal arguments, and the condition $||f_n|| \leq \frac{1}{2^{n-1}}$.
Hence, $F$ is a Rosenthal family for $X$ by Fact \ref{f:sub-fr}.

Since $F$ is also $S$-invariant 
we can apply Theorem \ref{t:general}. We obtain:
a Rosenthal space $V$, an injective continuous operator $j: V \to C(X)$ and
a weakly continuous representation $(h,\a)$ of $(S,X)$ on the Rosenthal Banach space $V$.

As we have noticed in the proof of Theorem \ref{t:general},
one of the properties of this construction is that $F \subset V$. Hence,
$sp(F) \subset V$.
 Consider the associated $S$-affine compactification $\g_0: X \to Q_0 \subset V^*$.
Here $Q_0=\overline{co}^{w^*} (\a (X))$.
We claim that $(\g_0, Q_0)$ is equivalent to $(\g, Q)$. It suffices to show that
$\mathrm{A}(Q_0)|_X=\mathrm{A}(Q)|_X$.

Consider the restriction operators:
$$
r_X: V \to \mathrm{A}(Q)|_X \subset C(X), \ \ r_X(v)(x):=\lan v, \a(x) \ran.
$$
$$
r_{Q_0}: V \to \mathrm{A}(Q_0) \subset C(Q_0), \ \ r_Q(v)(y):=\lan v,y \ran.
$$
$$
r_0: C(Q_0) \to C(X), \ \ r_0(v)(x):=\lan v, \a(x) \ran =\lan v, \g_0(x) \ran.
$$
By the choice of $F$, clearly, $r_X(sp(F))$ and hence also $r_X(V)$ are norm dense in the Banach space $\mathrm{A}(Q)|_X$.
Now it suffices to show that $r_X(V)$ is dense also in $\mathrm{A}(Q_0)|_X$.
First, by Lemma \ref{l:dense}, $r_{Q_0}(V) + \R \cdot \textbf{1}$ is dense in $\mathrm{A}(Q_0)$.
Since $\textbf{1}=r_{Q_0}(f_1) \in r_{Q_0}(V_0)$ and $r_{Q_0}(V)$ is a linear subspace
we conclude that $r_{Q_0}(V) + \R \cdot \textbf{1}=r_{Q_0}(V)$.
Therefore, $r_{Q_0}(V)$ is norm dense in the Banach space $\mathrm{A}(Q_0)$.
Then $r_0(r_{Q_0}(V))$ is dense in $r_0(\mathrm{A}(Q_0))=A(Q_0)|_X$.
Finally, it is easy to check that $r_0(r_{Q_0}(V))=r_X(V)$. So we can conclude that indeed
$r_X(V)$ is dense also in $A(Q_0)|_X$, as desired.
This proves the Tame case.

For the Asplund (respectively, reflexive) case we use the corresponding version of Theorem \ref{t:general} as explained
in the proof of Theorem \ref{t:tame=R-repr}.2
(respectively, \cite[Theorem 4.5]{Me-nz}).  
\end{proof}

Theorem \ref{t:aff-rep} can be extended to general (not necessarily metrizable) $S$-systems $X$ under the assumption that
the space $\mathrm{A(Q)}|_X$ of the affine compactification $\g: X \to Q$ is \emph{$S$-separable}.
The latter condition means that there exists a countable subset
$C \subset \mathrm{A(Q)}|_X$ such that $sp(CS)$ is dense in $\mathrm{A(Q)}|_X$.
In this general case the corresponding Rosenthal space $V$ is not necessarily separable.

Since the space $V_f$ of any cyclic affine $S$-compactification $\pi_f: S \to Q_f$ is always $S$-separable
we conclude that $\pi_f$ can be affinely $S$-represented
on a Rosenthal space for every $f \in \mathrm{Tame}(S)$.




\bibliographystyle{amsplain}

\end{document}